\documentclass[10pt]{article}
\usepackage[utf8]{inputenc}
\usepackage[T1]{fontenc}
\usepackage{amsmath, amsthm, amssymb, amsfonts}
\usepackage[version=4]{mhchem}
\usepackage{stmaryrd}
\usepackage{bbold}
\usepackage{graphicx}
\usepackage[export]{adjustbox}
\usepackage[shortlabels]{enumitem}
\usepackage{hyperref}
\usepackage{xcolor}

\theoremstyle{plain}
    \newtheorem{theorem}{Theorem}
    \newtheorem{corollary}[theorem]{Corollary}
    \newtheorem{lemma}[theorem]{Lemma}
    \newtheorem{proposition}[theorem]{Proposition}
    \theoremstyle{definition}
    \newtheorem{assumption}{Assumption}

    \newtheorem{remark}[theorem]{Remark}
    \newtheorem*{remark*}{Remark}
    
\numberwithin{theorem}{section}

    \newcommand{\R}{\mathbb R}
    \newcommand{\Z}{\mathbb Z}
    \newcommand{\pr}{\mathbf P}
    \renewcommand{\P}{\mathbf P}
    \newcommand{\e}{\mathbf E}
    \newcommand{\E}{\mathbf E}
    \newcommand{\ind}{{\rm 1}}
    \newcommand\eps\varepsilon

\begin{document}

\title{Fluctuations of Discrete-Time Random Walks}
\author{Denis Denisov and Vitali Wachtel}
\maketitle
\begin{abstract}
These notes are devoted to fluctuations of one-dimensional random walks. We discuss various approaches to first-passage times and to the corresponding conditional distributions. After discussion of some classical methods, such as reflection principle for simple random walks and Wiener-Hopf factorisation, we proceed to the universality approach, which has been developed in recent past.
Considering one-dimensional case allows us to avoid some technical obstacles and to present the core of this method in a more transparent way. It turns out that the universality method is much more robust than the Wiener-Hopf factorisation and allows one to consider walks with non-identically distributed or even dependent increments.
\end{abstract}

\section{Introduction}

Let $S(n)$ be a real-valued random walk with independent increments $\{X_k\}$, so that
\[
S(n)=X_1+X_2+\cdots+X_n,\quad n\ge 1.
\]
We are concerned with the so-called oscillating case, that is when the walk $S(n)$ is such that, almost surely,
$$
\limsup_{n\to\infty} S(n)=\infty
\quad\text{and}\quad
\liminf_{n\to\infty} S(n)=-\infty.
$$
This implies that the sets $(-\infty,-x]$ and $(x,\infty)$ for every $x\in\R$ are recurrent and, in particular, stopping times 
\[
\tau_x := \inf\{n\ge 1: \; x+S(n) \le 0\},\quad x\ge 0
\]
are almost surely finite.
We will be interested in asymptotic properties of distributions of stopping times $\tau_x$ and of corresponding conditional probabilities $\pr(x+S(n)\ge y,\tau_x>n)$ and $\pr(x+S(n)=y,\tau_x>n)$. These characteristics are among the most classical objects of study in the probability theory.

The main purpose of these notes is to describe existing (classical and rather new ones) approaches to the problems mentioned above. We shall start with some special classes of random walks, where explicit calculations play an important role in the asymptotic analysis. Then we describe basic principles of the Wiener-Hopf factorisation, which is apparently the most powerful approach to fluctuations of $1$-dimensional random walks with i.i.d. increments and of $1$-dimensional Levy processes. The most crucial requirement in this approach is the classical duality lemma for $S(n)$. Thus, there is no real hope that the factorisation techniques can be adapted to cases when the increments are not identically distributed or even not independent. The major part of this text will be devoted to a rather new approach which will be called {\it universality method.} Assume that there exists a scaling sequence $\{c_n\}$ such that $\frac{S(nt)}{c_n}$ converges to a stable process. (In this paper we shall consider only the case when one has Brownian motion in the limit.)
Then it is rather natural to expect that the behaviour of $\tau_x$ should be similar to the behaviour of the corresponding first exit time of the limiting process. This is rather obvious in the case when the starting point $x$ is or order of the scaling $c_n$. But if $x$ grows slower or is even fixed then such a transfer not clear. We provide a number of technical tools which allow one to transfer the knowledge on the exit times for limiting process into results on exit times and corresponding conditional distributions for the pre-limiting walk $S(n)$. Since these tools are based on functional limit theorems and martingales techniques only, the universality method is more robust than the Wiener-Hopf factorisation and provides results for random walks with not identically distributed increments and for Markov chains. Furthermore, this method works also in the multi-dimensional setting, where factorisation techniques do not work. Actually, we have initiated the development of the universality approach in our papers \cite{denisov_wachtel10, denisov_wachtel15, denisov_wachtel19} where multi-dimensional random walks in cones have been considered.

These notes are based on several mini-courses that were given by the authors in Munich (2014), Novosibirsk (2016), Melbourne (2018) and Wroclaw (2023).

\vspace{6pt}

The structure of the notes is as follows.

In Section~\ref{sec:simple-walks} we consider simple symmetric random walks, where one can obtain, by using the classical Andr\'e Reflection Principle, explicit expressions for $\pr(\tau_x=n)$ and for $\pr(x+S(n)=y,\tau_x>n)$. Using these explicit formulas we derive also several limit theorems, which should provide an intuition for more general results which will be obtained in later sections.
Section~\ref{sec:left-cont} is devoted to left-continuous random walks. This class of walks is also known for the fact that it allows for a closed form expression for $\pr(\tau_x>n)$ in terms of local probabilities of $S(n)$. This relation can be used to find tail asymptotics for $\tau_x$ and to prove conditional limit theorems.

Section~\ref{sec:factorisation} deals with the Wiener-Hopf factorisation, which is the most classical and the most powerful approach to fluctuations of $1$-dimensional random walks with independent, identically distributed increments. Here we introduce the notion of dual stopping times, find a dual for $\tau_0$ and show that $\tau_x$ with $x>0$ does not possess a dual stopping time. Then we derive  factorisation identities corresponding to pairs of dual stopping time and show their handiness by deriving asymptotics for $\pr(\tau_x>n)$ under the assumption
$\lim_{n\to\infty}\pr(S(n)>0)=\varrho\in(0,1)$. In this section we follow primarily the approach from  Greenwood and Shaked~\cite{Greenwood_Shaked78}.

In Section~\ref{sec:universality} we consider random walks with independent but not necessarily identically distributed increments. The Wiener-Hopf factorisation does not apply in this setting. Instead, one can use the so-called universality method, which combines functional limit theorems and martingale techniques, and allows one to study first-passage times of stochastic processes. As a result, we obtain asymptotics for $\tau_x$ and prove conditional limit theorems for random walks satisfying the Lindeberg condition, i.e. minimal condition under which the central limit theorem holds. The most important peculiarity of the universality approach is that it does not use generating functions and Fourier transforms. The presentation in this section follows rather close the paper \cite{denisov_sakhanenko_wachtel18}.

In Section~\ref{sec:iid} we specialize the results from Section~\ref{sec:universality} to the case of i.i.d. increments. Here we also give a probabilistic construction of a positive harmonic function for random walks killed at leaving positive half-axis.

Section~\ref{sec:local} is devoted to the proof of the local conditional limit theorem. Although, one can prove such theorems via Wiener-Hopf factorisation and Fourier transforms, we suggest here an alternative method, which combines the standard, unconditional, local limit theorem and the integral conditional limit theorem considered in Section~\ref{sec:universality}. This method is more robust and works in many situations, where the Wiener-Hopf factorisation is not applicable.

In Section~\ref{sec:markov} we show that the universality method applies also in the case when the increments of the walk are even not independent. There we consider a discrete-time Markov chain from the domain of attraction of Brownian motion and derive tail asymptoics for the first when the chain becomes non-positive and prove the corresponding limit theorems.

\section{Simple Random Walks}\label{sec:simple-walks}
In this section we shall always assume that the walk $\{S(n)\}$ is simple, that is,
\[
\P(X_1=1)=p \quad\text{and}\quad \P(X_1=-1)=q=1-p
\quad\text{for some }p\in(0,1).
\]
This implies that  trajectories of $S(n)$ 
are continuous  on $\mathbb{Z}$. 
This continuity property allows one to calculate 
many characteristics of the walk $S(n)$ explicitly. 

We shall start with stopping times $\tau_x$.
Combining the continuity with the spatial homogeneity, we conclude that,
for every $x\ge1$, $\tau_x$ is a sum of $x$ independent copies of the
stopping time $\tau_1$.
Let $f(s)$ be the generating function of $\tau_1$:
\[
f(s)=\E\bigl[s^{\tau_1}\bigr],\quad s\in[0,1].
\]
After the first step, the new position of the walk is either 0 (with probability $q$) or 2 (with probability $p$). In the first case $\tau_1=1$, while in the second case the process is restarted from 2. Then the generating function $f(s)$ satisfies
\[
f(s)=qs+ p\, s\, (f(s))^2.
\]
Solving this quadratic equation for $f(s)$, we obtain
\begin{equation}\label{eq:1.1}
\E\bigl[s^{\tau_1}\bigr] = \frac{1-\sqrt{1-4pqs^2}}{2ps}.
\end{equation}
Thus, for every $x\ge1$, 
\begin{equation}\label{eq:1.2}
\E\bigl[s^{\tau_x}\bigr] = \left(\frac{1-\sqrt{1-4pqs^2}}{2ps}\right)^x.
\end{equation}
In the case when the walk starts at zero we have
\begin{equation}
\label{eq:tau0-tau1}
\P\left(\tau_{0}=1\right)=q
\quad\text{and}\quad 
\P\left(\tau_{0}=2 k\right)=p \P\left(\tau_{1}=2 k-1\right)
\end{equation}
for all $k \ge1$. 
Consequently,
\begin{align}
\label{eq:tau0-generat}
\nonumber
\E \bigl[s^{\tau_{0}}\bigr] & =q s+\sum_{k=1}^{\infty} s^{2 k} \P\left(\tau_{0}=2 k\right) \\
\nonumber
& =q s+p s \sum_{k=1}^{\infty} s^{2 k-1} \P\left(\tau_{1}=2 k-1\right)=q+p s \E \bigl[s^{\tau_{1}}\bigr] \\
& =q s+p s\left(\frac{1-\sqrt{1-4 p q s^{2}}}{2 p s}\right)=q s+\frac{1-\sqrt{1-4 p q s^{2}}}{2}.
\end{align}
These exact expressions for generating functions allow one to obtain explicit expressions for $\P\left(\tau_{x}=n\right)$ for every $n$.

In what follows we shall concentrate on oscillating walks. This means that we restrict ourselves to the symmetric  case $p=q=1 / 2$. For this choice of the parameter $p$ we have
$$
\E\left[s^{\tau_{1}}\right]=\frac{1}{s}\left(1-\sqrt{1-s^{2}}\right)
$$

and

$$
\E\left[s^{\tau_{x}}\right]=\frac{1}{s^{x}}\left(1-\sqrt{1-s^{2}}\right)^{x} \quad \text { for } \quad x \ge 1.
$$
Using now the equality
\begin{equation}
\label{eq:neg.bin}
(1-z)^{-1/2}=\sum_{k=0}^{\infty} \binom{2k}{k}\frac{1}{2^{2k}} z^k,
\end{equation}
we conclude that
\begin{align*}
1-(1-z)^{1 / 2} & =\int_{0}^{z} \frac{1}{2}(1-u)^{-1 / 2} d u=\frac{1}{2} \sum_{k=0}^{\infty}\binom{2 k}{k} \frac{1}{2^{2 k}} \int_{0}^{z} u^{k} d u \\
& =\frac{1}{2} \sum_{k=0}^{\infty}\binom{2 k}{k} \frac{1}{2^{2 k}} \frac{z^{k+1}}{k+1} .
\end{align*}
Consequently,
\[
\E\bigl[s^{\tau_1}\bigr]=
\frac{1}{s}\left(1-\sqrt{1-s^{2}}\right)
=
\frac{1}{s}\sum_{k=0}^{\infty}\binom{2k}{k}\frac{1}{2^{2k+1}}\frac{s^{2k+2}}{k+1}
\]
and 
\[
\P(\tau_1=2k+1)=\frac{1}{k+1}\binom{2k}{k}2^{-2k-1},\quad k\ge0.
\]
Applying Stirling's formula, one obtains
\[
\P(\tau_1=2k+1) \sim \frac{1}{2\sqrt{\pi}}\,k^{-3/2},\quad k\to\infty.
\]
Moreover,
\begin{equation}
\label{eq:tau1-tail}
\P(\tau_1>n)=\sum_{k\ge n/2}\P(\tau_1=2k+1)\sim \sqrt{\frac{2}{\pi}}\, n^{-1/2},\quad n\to\infty.
\end{equation}
More generally, for any fixed $x\ge 1$, one has
\begin{equation}
\label{eq:taux-tail}
\P(\tau_x>n)\sim x\sqrt{\frac{2}{\pi}}\, n^{-1/2},\quad n\to\infty.
\end{equation}
Let us now turn to the remaining case $x=0$. To obtain asymptotics for $\P(\tau_0>n)$ we can first use equalities \eqref{eq:tau0-tau1} and apply then \eqref{eq:tau1-tail}. But one can also obtain a closed form expression for the probability $\P(\tau_0>n)$. Letting $p=q=\frac{1}{2}$ in \eqref{eq:tau0-generat}, one obtains easily
\begin{align*}
\sum_{n=0}^\infty \P(\tau_0>n)s^n
&=\frac{1-\E[s^{\tau_0}]}{1-s}=\frac{1-s+\sqrt{1-s^2}}{2(1-s)}\\
&=\frac{1}{2}+\frac{1}{2}\frac{\sqrt{1-s^2}}{(1-s)}
=\frac{1}{2}+\frac{1}{2}\frac{1+s}{\sqrt{1-s^2}}.
\end{align*}

Combining this representation with \eqref{eq:neg.bin}, we conclude that
$$
\P(\tau_0>2k)=\P(\tau_0>2k+1)=\binom{2k}{k}\frac{1}{2^{2k+1}},\quad k\ge 1.
$$
This formula is a particular case of the following proposition, which follows
from the classical reflection principle for Dyck paths.

\begin{proposition}
\label{prop:reflection}
Let $\{S(n)\}$ be a symmetric simple random walk. Then, for all $x,y\ge1$,
\begin{equation}
\label{eq:refl.1}
\P(x+S(n)=y,\tau_x>n)
=\P(S(n)=y-x)-\P(S(n)=y+x)
\end{equation}
and
\begin{equation}
\label{eq:refl.2}
\P(x+S(n)\ge y,\tau_x>n)
=\P(S(n)\in[y-x,y+x)).
\end{equation}
\end{proposition}
\begin{proof}
The equality \eqref{eq:refl.1} is the probabilistic formulation of the reflection principle.
For completeness we give now a probabilistic proof of this fact. By the Markov property,
\begin{align*}
\P(x+S(n)=y,\tau_x>n)
&=\P(x+S(n)=y)-\P(x+S(n)=y,\tau_x\le n)\\
&=\P(x+S(n)=y)-\P(x+S(n)=y,\tau_x<n)\\
&=\P(x+S(n)=y)-\sum_{k=1}^{n-1}\P(\tau_x=k)\P(S(n-k)=y)
\end{align*}
and 
\begin{align*}
0&=\P(x+S(n)=-y,\tau_x>n)\\
&=\P(x+S(n)=-y)-\P(x+S(n)=-y,\tau_x\le n)\\
&=\P(x+S(n)=-y)-\P(x+S(n)=y,\tau_x<n)\\
&=\P(x+S(n)=-y)-\sum_{k=1}^{n-1}\P(\tau_x=k)\P(S(n-k)=-y)
\end{align*}
for all $x,y\ge1$. Taking the difference and using the symmetry of the distribution of the walk 
$\{S(n)\}$, we obtain
\begin{align*}
\P(x+S(n)=y,\tau_x>n)=\P(x+S(n)=y)-\P(x+S(n)=-y).    
\end{align*}
Thus, \eqref{eq:refl.1} is proved.

The second claim follows from \eqref{eq:refl.1} by summation:
\begin{align*}
\P(x+S(n)\ge y,\tau_x>n)
&=\sum_{z=y}^{\infty}\P(x+S(n)=z,\tau_x>n)\\
&=\sum_{z=y}^{\infty}\left[\P(S(n)=z-x)-\P(S(n)=z+x)\right]\\
&=\P(S(n)\ge y-x)-\P(S(n)\ge y+x)\\
&=\P(S(n)\in[y-x,y+x)).
\qedhere
\end{align*}
\end{proof}
\begin{corollary}
\label{cor:tau_x}
One has 
$$
\P(\tau_0>n)=\frac{1}{2}\P(S(n-1)\in\{0,1\})
$$
and, for every $x\ge 1$,
$$
\P(\tau_x>n)=\P(S(n)\in[1-x,x+1)]).
$$
\end{corollary}  
\begin{proof}
The second claim of the corollary follows from \eqref{eq:refl.2} with $y=1$.
To get the first one it suffices to combine the second claim with \eqref{eq:tau0-tau1}. 
\end{proof}

Combining these exact expressions with the de Moivre-Laplace theorem one can easily obtain 
various asymptotic relations for the simple symmetric random walk. First we give an alternative proof of the tail asymptotics for the stopping times $\tau_x$.
\begin{corollary}
\label{cor:tau_x-asymp}
As $n\to\infty$,
$$
\P(\tau_0>n)\sim\frac{1}{\sqrt{2\pi}}\frac{1}{\sqrt n}
$$
and, 
uniformly in $x=o(\sqrt n)$,
\begin{equation}
\label{eq:tau_x-asymp}
\P(\tau_x>n)\sim \sqrt{\frac{2}{\pi}}\frac{x}{\sqrt n}.
\end{equation}
Also, there exists a constant $C$ such that 
\begin{equation}\label{eq:tau_x.upper}
\P(\tau_x>n)\le C\frac{x+1}{\sqrt n}
\end{equation}
for all $x\ge 0$ and $n\ge 1$. 
\end{corollary}
\begin{proof}
Again, the first asymptotic relation is a combination of \eqref{eq:tau0-tau1} and of the second one with $x=1$.

According to Corollary~\ref{cor:tau_x}, for every $x\ge1$,

\begin{align}
\label{eq:tau-sum}
\P(\tau_x>n)=\sum_{z=1-x}^x\P(S(n)=z). 
\end{align}
By the de Moivre-Laplace theorem,
\begin{equation}
\label{eq:ML-odd}
\P(S(n)=z)=0\quad \text{if $n+z$ is odd}
\end{equation}
and
\begin{equation}
\label{eq:ML-even}
\sup_{z:\,m+z\text{ is even}}
\left|\P(S(n)=z)-\frac{2}{\sqrt{2\pi n}}e^{-z^2/2n}\right|
=o\left(\frac{1}{\sqrt{n}}\right).
\end{equation}
Using these relations and noting that the summation interval $[1-x,x]$ in \eqref{eq:tau-sum} contains exactly $x$ points $z$ such that $n+z$ is even, we obtain \eqref{eq:tau_x-asymp}.

To prove~\eqref{eq:tau_x.upper} we 
 use a concentration inequality that 
ensures existence of $C$ such that 
$\pr(S(n)=y)\le \frac{C}{2\sqrt n}$ for all $y$ and $n$. Then,
\[
\P(\tau_x>n)=\P(S(n)\in[1-x,x+1)])
=\sum_{y=1-x}^{1+x}
\pr(S(n)=y)
\le C\frac{1+x}{\sqrt n}.
\]
This completes the proof of the corollary.
\end{proof}
\begin{corollary}
\label{cor:large-x}
If $\frac{x}{\sqrt{n}}\to v>0$ then 
$$
\pr(\tau_x>n)\to \sqrt{\frac{2}{\pi}}\int_0^v e^{-u^2/2}du.
$$
\end{corollary}
\begin{proof}
Combining Corollary~\ref{cor:tau_x} with the integral de Moivre-Laplace theorem, we obtain
\begin{align*}
\pr(\tau_x>n)&=\P(S(n)\le x)-\P(S(n)\le-x)\\
&=\Phi\left(\frac{x}{\sqrt{n}}\right)-\Phi\left(\frac{-x}{\sqrt{n}}\right)+o(1)\\
&=2\int_{0}^{x/\sqrt{n}}\frac{1}{\sqrt{2\pi}}e^{-u^2/2}du+o(1).
\end{align*}
This gives the desired relation.
\end{proof}
\begin{corollary}
\label{cor:limit.th-simple}
If $x=o(\sqrt{n})$ as $n\to\infty$ then 
$$
\P\left(\frac{x+S(n)}{\sqrt{n}}>v\Big|\tau_x>n\right)\to e^{-v^2/2},\quad v>0.
$$
\end{corollary}
\begin{proof}
We give first a proof for the case $x\ge1$.  
According to \eqref{eq:refl.2},
$$
\P(x+S(n)\ge \lfloor v\sqrt{n}\rfloor,\tau_x>n)
=\sum_{z=\lfloor v\sqrt{n}\rfloor -x}^{\lfloor v\sqrt{n}\rfloor+x-1}\P(S(n)=z).
$$
Noting that the summation interval contains also here exactly $x$ points such that $z+n$ is even
and applying \eqref{eq:ML-odd}, \eqref{eq:ML-even}, we conclude that 
$$
\P(x+S(n)\ge \lfloor v\sqrt{n}\rfloor,\tau_x>n)
\sim \sqrt{\frac{2}{\pi}}\frac{x}{\sqrt n} e^{-v^2/2}.
$$
Combining this with Corollary~\ref{cor:tau_x-asymp}, we complete the proof for $x\ge1$. If $x=0$ then we have 
$$
\P(S(n)\ge \lfloor v\sqrt{n}\rfloor,\tau_0>n)
=\frac{1}{2}\P(1+S(n-1)\ge \lfloor v\sqrt{n}\rfloor,\tau_1>n-1)
$$
and
$$
\P(\tau_0>n)=\frac{1}{2}\P(\tau_1>n-1)
$$
Combining these equalities with the convergence in the case $x=1$, we finish the proof.
\end{proof}

One of the important properties of \eqref{eq:tau_x-asymp} is the fact that the dependence
on $x$ and $n$ factorizes. We shall see later that this happens for all oscillating random
walks and that the function which describes the dependence on the starting point plays an important role in studying many properties of walks conditioned to stay positive. 

The continuity of the trajectories of a simple random walk implies that the function
\[
V(x)=x
\]
satisfies the following \emph{harmonicity} property:
\begin{align*}
\E\bigl[V(x+S(1));\, \tau_x>1\bigr]&=\E\bigl[x+S(1);\, \tau_x>1\bigr]
=\E\bigl[x+S(1)\bigr]=x,\quad x\ge2
\end{align*}
and
\begin{align*}
\E\bigl[V(1+S(1));\, \tau_1>1\bigr]&=\E\bigl[1+S(1);\, \tau_1>1\bigr]
=2\P\bigl(1+S(1)=2\bigr)=1. 
\end{align*}
More generally,
\[
x=\E\bigl[x+S(n);\, \tau_x>n\bigr],\quad x,n\ge1.
\]

Using this function, one can define a new transition kernel on $\mathbb{Z}_+$ (the Doob $h$-transform of the original random walk):
\[
\widehat{p}(x,x+1)=\frac{V(x+1)}{V(x)}\, \P\Bigl(x+S(1)=x+1,\,\tau_x>1\Bigr)
=\frac{x+1}{x}\cdot\frac{1}{2},
\]
and similarly,
\[
\widehat{p}(x,x-1)=\frac{V(x-1)}{V(x)}\, \P\Bigl(x+S(1)=x-1,\,\tau_x>1\Bigr)
=\frac{x-1}{x}\cdot\frac{1}{2}.
\]
This transition kernel produces a Markov chain $\{\widehat{S}(n)\}$ on the positive integers which is 
commonly referred to as 
 \emph{random walk conditioned to stay positive}. 
 This terminology is justified by the following observation. For every $x_0\ge0$
 we have
\begin{align*}
& \P\left(x_{0}+S(k+1)=x+1 \mid x_{0}+S(k)=x, \tau_{x_{0}}>n\right) \\
&\phantom{XXXX}=\frac{\P\left(x_{0}+S(k+1)=x+1, x_{0}+S(k)=x, \tau_{x_{0}}>n\right)}{\P\left(x_{0}+S(k)=x, \tau_{x_{0}}>n\right)} \\
&\phantom{XXXX}=\frac{\P\left(x_{0}+S(k)=x, \tau_{x_{0}}>k\right) \frac{1}{2} \P\left(\tau_{x+1}>n-k-1\right)}{\P\left(x_{0}+S(k)=x, \tau_{x_{0}}>k\right) \P\left(\tau_{x}>n-k\right)} \\
&\phantom{XXXX} \longrightarrow \frac{1}{2} \frac{x+1}{x}=\widehat{p}(x,x+1)
\quad\text{as }n\to\infty,
\end{align*}
where we used asymptotics~\eqref{eq:tau_x-asymp} 
in the last line. 
Similarly one shows that 
$$
\P\left(x_{0}+S(k+1)=x+1 \mid x_{0}+S(k)=x, \tau_{x_{0}}>n\right) 
\to \widehat{p}(x,x-1)\quad\text{as }k\to\infty.
$$
Using Proposition~\ref{prop:reflection} one can obtain a limit theorem for the chain
$\{\widehat{S}(n)\}$.
\begin{proposition}
\label{prop:Doob-limit}
For every fixed $x\ge1$,
$$
\P_x\left(\frac{\widehat{S}(n)}{\sqrt{n}}\ge v\right)\to
\sqrt{\frac{2}{\pi}}\int_v^\infty u^2e^{-u^2/2}du,\quad v>0.
$$
\end{proposition}
\begin{proof}
By the definition of the chain $\{\widehat{S}(n)\}$,
\begin{align*}
&\P_x(\widehat{S}(n)\ge y)\\
&\hspace{1cm}=\sum_{z=y}^\infty\P_x(\widehat{S}(n)=z)
=\frac{1}{V(x)}\sum_{z=y}^\infty V(z)\P(x+S(n)=z,\tau_x>n)\\
&\hspace{1cm}=\frac{1}{x}\sum_{z=y}^\infty z\P(x+S(n)=z,\tau_x>n)\\
&\hspace{1cm}=\frac{y}{x}\P(x+S(n)\ge y,\tau_x>n)
+\frac{1}{x}\sum_{z=y+1}^\infty(z-y)\P(x+S(n)=z,\tau_x>n).
\end{align*}
Using summation by parts formula and applying Proposition~\ref{prop:reflection}, we obtain 
\begin{align}
\label{eq:doob.1}
\nonumber
&\P_x(\widehat{S}(n)\ge y)\\
\nonumber
&\hspace{1cm}=\frac{y}{x}\P(x+S(n)\ge y,\tau_x>n)
+\frac{1}{x}\sum_{u=y+1}^\infty\P(x+S(n)\ge u,\tau_x>n)\\
\nonumber
&\hspace{1cm}=\frac{y}{x}\P(x+S(n)\ge y,\tau_x>n)
+\frac{1}{x}\sum_{u=y+1}^\infty\P(x+S(n)\in[u-x,u+x))\\
&\hspace{1cm}=\frac{y}{x}\P(x+S(n)\ge y,\tau_x>n)
+\frac{1}{x}\sum_{z=y-x+1}^{y+x}\P(x+S(n)\ge z).
\end{align}
As we have seen in the proof of Corollary~\ref{cor:tau_x-asymp},
$$
\P(x+S(n)\ge y,\tau_x>n)\sim x\sqrt{\frac{2}{\pi n}}e^{-v^2/2}
$$
provided that $y\sim v\sqrt{n}$. Therefore,
\begin{equation}
\label{eq:doob.2}
\frac{y}{x}\P(x+S(n)\ge y,\tau_x>n)\sim 
\sqrt{\frac{2}{\pi}}ve^{-v^2/2}.
\end{equation}
Furthermore, by the central limit theorem,
\begin{equation}
\label{eq:doob.3}
\sum_{z=y-x+1}^{y+x}\P(x+S(n)\ge z)\sim 2x\int_v^\infty \frac{1}{\sqrt{2\pi}}e^{-u^2/2}du.
\end{equation}
Plugging \eqref{eq:doob.2} and \eqref{eq:doob.3} into \eqref{eq:doob.1},
we conclude that 
$$
\P_x(\widehat{S}(n)\ge y)\sim 
\sqrt{\frac{2}{\pi}}\left(ve^{-v^2/2}+\int_v^\infty e^{-u^2/2}du\right)
$$
provided that $y\sim v\sqrt{n}$. 
Integrating by parts completes the proof.
\end{proof}


\section{Left-continuous random walks.}\label{sec:left-cont}
There is a further class of lattice random walks where one can obtain some rather explicit expressions for the distribution of $\tau_x$.  It turns out that it is sufficient to assume that the walk can only move downwards in a continuous manner. 
More precisely, it suffices to assume that
$$
\P\left(X_1\in\{-1,0,1,\ldots\}\right)=1.
$$
If this condition holds then we shall say that the walk $\{S(n)\}$ is \emph{left-continuous}.

Similar to the case of the simple random walk, we have
$$
S(\tau_x)=0,\quad x\ge1
$$
and 
$$
S(\tau_0)=-\ind\{\tau_0=1\}.
$$
\begin{proposition}
\label{prop:tau-local}
For all $x,n\ge1$ one has 
\begin{equation}
\label{eq:tau-local}
\P(\tau_x=n)=\frac{x}{n}\P(S(n)=-x).
\end{equation}
\end{proposition}
\begin{proof}
We shall use the induction over $n$. Assume first that $n=1$.
It is clear that 
$\P(\tau_1=1)=\P(X_1\le -1)=\P(X_1=-1)$. Furthermore, the left-continuity implies that
$\P(\tau_x=1)=0=\P(X_1=-x)$ for all $x>1$. Thus, \eqref{eq:tau-local} holds for $n=1$ and all $x\ge1$.

To perform the induction step, we assume that \eqref{eq:tau-local} is valid for some $n\ge1$.
Combining this induction assumption with the Markov property, we obtain
\begin{align*}
 \P(\tau_x=n+1)&=\sum_{k=-1}^\infty\P(X_1=k)\P(\tau_{x+k}=n)\\
 &=\sum_{k=-1}^\infty\P(X_1=k)\frac{x+k}{n}\P(S(n)=-x-k)\\
 &=\frac{x}{n}\P(S(n+1)=-x)+\frac{1}{n}\E\left[X_1;S(n+1)=-x\right].
\end{align*}
Since all $X_k$ are independent and identically distributed,
$$
\E\left[X_1;S(n+1)=-x\right]=\E\left[X_k;S(n+1)=-x\right]
$$ 
for all $k\le n+1$. This implies that
$$
\E\left[X_1;S(n+1)=-x\right]=-\frac{x}{n+1}\P(S(n+1)=-x).
$$
Consequently,
\begin{align*}
\P(\tau_x=n+1)&=\frac{x}{n}\P(S(n+1)=-x)-\frac{x}{n(n+1)}\P(S(n+1)=-x)\\
&=\frac{x}{n+1}\P(S(n+1)=-x).
\end{align*}
Thus, the proof is complete.
\end{proof}
\begin{remark}
There exist several different proofs of this proposition. Here we followed an 
elementary approach due to~\cite{HofstadKean08}. 
\end{remark}

Proposition~\ref{prop:tau-local} connects with each other the mass functions of the random variables $\tau_x$ and $S(n)$. In particular, asymptotics for
$\pr(\tau_x=n)$ can be obtained from the local central limit theorem for lattice random variables. To formulate this result we recall first the notion of the aperiodicity of lattice distributions. Let $X$ be a random variable with values in $\Z$. We set
$$
d:={\rm g.c.d.}\{k-j:\,\pr(X=k)\pr(X=j)>0\}
$$
This number is called {\it period} of the distribution of $X$. If $d=1$ then we shall say that the distribution of $X$ is {\it aperiodic}. If $d>1$ then there exists an integer $a\in[0,d)$ such that the distribution of $Y:=\frac{X-a}{d}$
is aperiodic. 

Let $\{Y_k\}$ be a sequence of i.i.d. random variables with finite positive variance $\sigma_Y^2$ and with an aperiodic distribution. Then one has the following local version of the central limit theorem:
\begin{align}
\label{eq:lclt-aperiodic}
\sup_{y\in\Z}\left|\sqrt{n}\pr\left(\sum_{k=1}^nY_k=y\right)
-\frac{1}{\sqrt{2\pi\sigma_Y^2}}\exp\left\{-\frac{(y-n\E Y_1)^2}{2n\sigma_Y^2}\right\}\right|\to0.
\end{align}

If $\{X_k\}$ is a sequence of i.i.d. random variables with finite positive variance $\sigma^2$ and with period $d>1$. Then the random variables 
$Y_k:=\frac{X_k-a}{d}$ have expectation $\frac{\E X_1-a}{d}$ and variance
$\frac{\sigma^2}{d^2}$. Noting now that
\begin{align*}
\pr\left(\sum_{k=1}^nX_k=x\right)
&=\pr\left(\sum_{k=1}^n(a+dY_k)=x\right)
=\pr\left(\sum_{k=1}^nY_k=\frac{x-an}{d}\right)
\end{align*}
and using \eqref{eq:lclt-aperiodic}, we conclude that
\begin{align}\label{eq:lclt-1}
\sup_{x\in D_n}\left|\sqrt{n}\pr\left(\sum_{k=1}^nX_k=x\right)
-\frac{d}{\sqrt{2\pi\sigma^2}}\exp\left\{-\frac{(x-n\E X_1)^2}{2n\sigma^2}\right\}\right|\to0
\end{align}
and 
\begin{align}\label{eq:lclt-2}
\pr\left(\sum_{k=1}^nX_k=x\right)=0\quad\text{for all }x\notin D_n,
\end{align}
where 
$$
D_n:=\left\{x\in\Z:\, \frac{x-an}{d}\in\Z\right\}.
$$

In the case when $\E X_1=0$, $\sigma^2:=\E X_1^2\in(0,\infty)$ and the distribution of $X_1$ is aperiodic, \eqref{eq:lclt-1} implies that
$$
\P(S(n)=-x)\sim\frac{1}{\sigma\sqrt{2\pi n}}\quad \text{as }n\to\infty
$$
for every fixed $x$.
Consequently, as $n\to\infty$,
\begin{equation}
\label{eq:tau-local.asymp}
\P(\tau_x=n)\sim\frac{x}{\sigma\sqrt{2\pi}}n^{-3/2}
\end{equation}
and, by summing up $\P(\tau_x=n)$, 
\begin{equation}
\label{eq:tau-tail.asymp}
\P(\tau_x>n)\sim\frac{x}{\sigma}\sqrt{\frac{2}{\pi}}n^{-1/2}.
\end{equation}

If $X_1$ has period $d>1$, $\E X_1=0$ and $\pr(X_1\ge-1)=1$ then $a=d-1$.
Indeed, the assumptions $\E X_1=0$ and $\pr(X_1\ge-1)=1$ imply that
$\pr(X_1=-1)>0$. Consequently, $-1=a+dm$ for some $m\in\Z$. Recalling that, by definition, $a\in[0,d)$, we conclude that $m=-1$ and $a=d-1$. This implies that
$$
D_n=\left\{x\in\Z:\,\frac{x-(d-1)n}{d}\in\Z\right\}
=\left\{x\in\Z:\,\frac{x+n}{d}\in\Z\right\}.
$$
Combining this with \eqref{eq:lclt-1} and $\eqref{eq:lclt-2}$, we obtain, for every $x\ge1$,
$$
\pr(S(n)=-x)\sim\frac{d}{\sigma\sqrt{2\pi n}}\quad
\text{as } n\to\infty\text{ and }n\in E_x
$$
and 
$$
\pr(S(n)=-x)=0\quad\text{for all }x\notin E_x,
$$
where 
$$
E_x:=\left\{x+dm,\,m\ge0\right\}.
$$
Combining this with Proposition~\ref{prop:tau-local}, we finally conclude that
$$
\pr(\tau_x=n)\sim x\frac{d}{\sigma\sqrt{2\pi}}n^{-3/2}\quad
\text{as } n\to\infty\text{ and }n\in E_x
$$
and that \eqref{eq:tau-tail.asymp} remains valid also in the case periodic distributions.

The left-continuity allows one also to show that Corollary~\ref{cor:limit.th-simple} remains valid for left-continuous random walks. 
\begin{proposition}
\label{prop:limit.th-left}
Assume that $\E X_1=0$, $\sigma^2=\E X_1^2\in(0,\infty)$. Then, for every fixed $x\ge1$,
$$
\P\left(\frac{x+S(n)}{\sigma\sqrt{n}}\ge v\,\Big|\,\tau_x>n\right)\to e^{-v^2/2}.
$$
\end{proposition}
\begin{proof} 
To simplify a bit the proof we shall additionally assume that the distribution of $X_1$ is aperiodic.
Let $y=y_n:=\lfloor v\sqrt n \rfloor $. 
Repeating the arguments from the proof of the reflection principle in Proposition~\ref{prop:reflection}, we have 
\begin{align*}
 \P(x+S(n)\ge y,\tau_x>n)
=\P(x+S(n)\ge y)-\sum_{k=1}^{n-1}\P(\tau_x=k)\P(S(n-k)\ge y)   
\end{align*}
and 
\begin{align*}
0&=\P(x+S(n)\le -y,\tau_x>n)\\
&=\P(x+S(n)\le -y)-\sum_{k=1}^{n-1}\P(\tau_x=k)\P(S(n-k)\le -y).
\end{align*}
Taking the difference, we obtain 
\begin{align*}
&\P(x+S(n)\ge y,\tau_x>n)\\
&\hspace{1cm}=\P(x+S(n)\ge y)-\P(x+S(n)\le -y)\\
&\hspace{2cm}+\sum_{k=1}^{n-1}\P(\tau_x=k)\left[\P(S(n-k)\le-y)-\P(S(n-k)\ge y)\right]\\
&\hspace{1cm}=\P(S(n)\in[y-x,y))+\P(S(n)\in(-y-x,-y])\\
&\hspace{2cm}+\P(S(n)\ge y)-\P(S(n)\le-y)\\
&\hspace{2cm}+\sum_{k=1}^{n-1}\P(\tau_x=k)\left[\P(S(n-k)\le-y)-\P(S(n-k)\ge y)\right].
\end{align*}
By the local central limit theorem \eqref{eq:lclt-1},
\begin{equation*}
\P(S(n)\in[y-x,y))+\P(S(n)\in(-y-x,-y])\sim \frac{x\sqrt{2}}{\sigma\sqrt{\pi n}}e^{-v^2/2}
\end{equation*}
since  $y\sim v\sigma \sqrt{n}$. This implies that
\begin{align*}
&\P(x+S(n)\ge y,\tau_x>n)-\frac{x\sqrt{2}}{\sigma\sqrt{\pi n}}e^{-v^2/2}\\
&\hspace{1cm}=o(n^{-1/2})
+\P(S(n)\ge y)-\P(S(n)\le-y)\\
&\hspace{2cm}+\sum_{k=1}^{n-1}\P(\tau_x=k)\left[\P(S(n-k)\le-y)-\P(S(n-k)\ge y)\right]\\
&\hspace{1cm}=o(n^{-1/2})
+(\P(S(n)\ge y)-\P(S(n)\le-y))\P(\tau_x\ge n)\\
&\hspace{2cm}+\sum_{k=1}^{n-1}\P(\tau_x=k)\gamma_{n,k}(y),
\end{align*}
where
$$
\gamma_{n,k}(y):=\P(S(n)\ge y)-\P(S(n)\le-y)+\P(S(n-k)\le-y)-\P(S(n-k)\ge y).
$$
Combining \eqref{eq:tau-tail.asymp} with the central limit theorem, we infer that 
$$
(\P(S(n)\ge y)-\P(S(n)\le-y))\P(\tau_x\ge n)=o(n^{-1/2}).
$$
Consequently,
\begin{align}
\label{eq:interm.step}
\nonumber
&\left|\P(x+S(n)\ge y,\tau_x>n)-\frac{x\sqrt{2}}{\sigma\sqrt{\pi n}}e^{-v^2/2}\right|\\
&\hspace{2cm}\le o(n^{-1/2})+\sum_{k=1}^{n-1}\P(\tau_x=k)\gamma_{n,k},
\end{align}
where
$$
\gamma_{n,k}:=\sup_{y}|\gamma_{n,k}(y)|. 
$$
By the central limit theorem, $\gamma_{n,k}\to 0$ as both $n$ and $n-k\to\infty.$ Combining this with \eqref{eq:tau-local.asymp}, we conclude that 
\begin{align}
 \label{eq:large-k}
 \sum_{k\in(\varepsilon n, n-1]}\P(\tau_x=k)\gamma_{n,k}=o(n^{-1/2})
\end{align}
for every $\varepsilon\in(0,1)$. 

To deal with smaller values of $k$ we notice that by a telescoping argument, 
\begin{align*}
\gamma_{n,k}
&\le 2\sup_y|\P(S(n)\ge y)-\P(S(n-k)\ge y)|\\
&\le 2\sum_{j=n-k}^{n-1}\sup_y|\P(S(j+1)\ge y)-\P(S(j)\ge y)|.
\end{align*}
For the summands on the right hand side one has the following bound 
\begin{equation}
\label{eq:diff-bound}
\Delta_j:=\sup_{y}\left|\P(S(j+1)\ge y)-\P(S(j)\ge y)\right|\le \frac{C}{j}.
\end{equation}
We postpone the  proof 
of this bound and notice that it implies that 
$$
\gamma_{n,k}\le C\frac{k}{n}\quad\text{for all }k\le n/2.
$$
Therefore,
\begin{align}
 \label{eq:small-k}
 \sum_{k\le\varepsilon n}\P(\tau_x=k)\gamma_{n,k}
 \le \frac{C}{n}\sum k\P(\tau_x=k)\le C'\varepsilon^{1/2}n^{-1/2},
\end{align}
where in the last step we have used \eqref{eq:tau-local.asymp}.
Plugging \eqref{eq:large-k}, \eqref{eq:small-k} into \eqref{eq:interm.step}
and letting $\varepsilon\to0$, we conclude that
$$
\P(x+S(n)\ge y,\tau_x>n)\sim\frac{x\sqrt{2}}{\sigma\sqrt{\pi n}}e^{-v^2/2}
$$
since  $y\sim v\sigma\sqrt{n}.$

Thus, it remains to prove \eqref{eq:diff-bound}. Let $\varphi(t)=\E e^{itX_1}$ 
be the characteristic function of $X_1$. 
By the inversion formula, for 
a fixed $A>y$,
\begin{align*}
&\P(S(j+1) \in [y,y+A])-\P(S(j)\in [y,y+A])\\
&\hspace{1cm}=\sum_{k=y}^{y=A} 
\frac{1}{2\pi} \int_{-\pi}^\pi 
e^{-i t k} \left(\varphi^{j+1}(t)
-\varphi^j(t)\right)dt \\
&\hspace{1cm}=\frac{1}{2\pi} \int_{-\pi}^\pi 
e^{-i t y} \frac{1-e^{it(A-y)}}{1-e^{it}}
\varphi^j(t)(1-\varphi(t)) dt.
\end{align*}
Noting $\max_{t\in[-\pi,\pi]} \frac{|t|}{|1-e^{it}|}=\frac{\pi}{2}$,
we obtain 
\begin{align*}
\left|\P(S(j+1) \in [y,y+A])-\P(S(j)\in [y,y+A])\right|
\le \frac{1}{2}
\int_{-\pi}^\pi|\varphi(t)|^j\left|\frac{1-\varphi(t)}{t}\right|dt.
\end{align*}
Letting now $A\to\infty $, we conclude that  
$$
\Delta_j\le\frac{1}{2}\int_{-\pi}^\pi|\varphi(t)|^j\left|\frac{1-\varphi(t)}{t}\right|dt.
$$
Since $\E X_1=0$ and $\sigma^2=\E X_1^2<\infty$, $|1-\varphi(t)|\le \frac{\sigma^2t^2}{2}$
and, consequently,
$$
\Delta_n\le\frac{\sigma^2}{4}\int_{-\pi}^\pi|t||\varphi(t)|^ndt.
$$
Furthermore, the finiteness of the second moment implies the existence of $\varepsilon,\delta>0$
such that 
$$
|\varphi(t)|\le e^{-\varepsilon t^2},\quad |t|\le\delta.
$$
Therefore,
$$
\int_{-\delta}^\delta|t||\varphi(t)|^ndt
\le 2\int_{0}^\delta te^{-\varepsilon nt^2}dt\le \frac{1}{\varepsilon n}.
$$
The aperiodicity of $X_1$ implies that there exists $q<1$ such that $|\varphi(t)|\le q$ for
all $|t|\in(\delta,\pi]$. Therefore,
$$
\int_{\{t:|t|\in(\delta,\pi\}}|t||\varphi(t)|^ndt\le Cq^n.
$$
This completes the proof of \eqref{eq:diff-bound} and the proof of the proposition. 
\end{proof}
\begin{remark}
One can see that the reflection type arguments
work well not only for the simple symmetric  random walk but for non-symmetric random walk as well. 
The approach via the reflection principle  works for 
general random walks as well and 
very similar arguments can give even 
Berry-Esseen type estimates for the rate 
of convergence, see~\cite{DTW25} for further details. 
\end{remark}
\section{Dual stopping times and Wiener--Hopf factorisation}
\label{sec:factorisation}
In this section we describe principal elements of the Wiener--Hopf factorisation. This method seems to be the most powerful tool in the analysis of fluctuations of one-dimensional walks with independent, identically distributed increments. 

In contrast to simple and left-continuous in the case of general random walks one has no closed form expressions for distributions of stopping times $\tau_x$, but the Wiener--Hopf factorisation allows one to obtain explicit expressions for generating functions of $\tau_0$.
These equalities can serve as a starting point in the tail analysis of $\tau_x$.

The standard references for the Wiener--Hopf factorisation are books of Spitzer~\cite{spitzer1964principles} and of Borovkov~\cite{borovkov2013probability}. We shall present below a bit different approach to the factorisation, which has been suggested by Greenwood and Shaked~\cite{Greenwood_Shaked78}.

Let us start with the one-dimensional case, that is, $\{X_n\}$ are independent, identically distributed real-valued random variables.

Besides the stopping time $\tau_0$ we define
$$
\tau^+:=\inf\{n\ge1: S(n)>0\}
$$
and
$$
U^+_0:=0,\quad U^+_{k+1}:=\inf\{n>U^+_k:S(n)>S(U^+_k)\}\quad\text{for every }k\ge1.
$$
The random variables $\{U^+_k\}_{k\ge1}$ are called \emph{strict ascending ladder epochs.}
Clearly, $U^+_1=\tau^+$. Furthermore, it is immediate from the Markov property that 
$\tau^+_k:=U^+_k-U^+_{k-1}$ are independent copies of $\tau^+$. The independent random variables 
$\chi^+_k:=S(U^+_k)-S(U^+_{k-1})$ are called \emph{strict ascending ladder heights.}

Similarly we define descending ladder variables. First we define 
$$
U^-_0:=0,\quad U^-_{k+1}:=\inf\{n>U^-_k:S(n)\le S(U^-_k)\}\quad\text{for every }k\ge1.
$$
These random times are called \emph{weak descending ladder epochs.} Furthermore,
the variables $\tau^-_k:=U^-_k-U^-_{k-1}$ are independent copies of $\tau_0$.
Finally we define \emph{weak descending ladder epochs} by the equalities 
$\chi^-_k:=S(U^-_k)-S(U^-_{k-1})$, $k\ge1$.

For the ladder epochs $U^+_k$ one has
\begin{align*}
&\sum_{k=1}^\infty \P\left(\tau_{1}^{+}+\ldots+\tau_{k}^{+}=n\right)\\
&\phantom{XXX} =\P\left(\tau_1^{+}+\cdots+\tau_k^{+}=n \text{ for some } k\ge1\right)\\
&\phantom{XXX} =\P\bigl(U^+_k=n \text{ for some }k\ge1\bigr)\\
&\phantom{XXX} =\P\Bigl(S(n)>S(n-1), S(n)>S(n-2), \ldots, S(n)>S(1), S(n)>0\Bigr)\\[1mm]
&\phantom{XXX} =\P\Bigl(X_n>0,\, X_n+X_{n-1}>0,\, \dots,\, X_n+X_{n-1}+\cdots+X_1>0\Bigr)\\[1mm]
&\phantom{XXX} = \P\bigl(\tau_0>n\bigr),
\end{align*}
in the last step we have used the classical duality lemma for random walks.
Multiplying both sides by $s^n$ and summing over $n$, we obtain
\begin{align*}
\sum_{n=1}^\infty s^n \P(\tau_0>n)
& =\sum_{n=1}^{\infty} s^{n} \sum_{k=1}^{\infty} \P\left(\tau_{1}^{+}+\cdots+\tau_{k}^{+}=n\right) \\
& =\sum_{k=1}^{\infty}\left(\E\left[s^{\tau^+}\right]\right)^{k}=\frac{1}{1-\E\left[s^{\tau^{+}}\right]}-1.
\end{align*}
Thus,
\[
\frac{1-\E[s^{\tau_0}]}{1-s}=\frac{1}{1-\E[s^{\tau^{+}}]},
\]
or, equivalently,
\begin{equation}\label{WH}
\bigl(1-\E[s^{\tau_0}]\bigr)\bigl(1-\E[s^{\tau^{+}}]\bigr)=1-s.
\end{equation}

If the distribution of $X_1$ is symmetric and has no atoms then $\tau_0$ and $\tau^{+}$ have the same distribution and it follows from \eqref{WH} that
\[
\E[s^{\tau_0}]=1-\sqrt{1-s}.
\]
Therefore,
\[
\P(\tau_0>n) = \binom{2n}{n}2^{-2n}\quad\text{for every }n\ge1.
\]
In general, without the symmetry assumption, \eqref{WH} gives one equation for the two unknown generating functions $\E[s^{\tau_0}]$ and $\E[s^{\tau^{+}}]$. One cannot determine distributions of $\tau_0$ and $\tau^+$, but \eqref{WH} allows us to extract some useful asymptotic relations. For instance:
\begin{enumerate}[(a)]
    \item If $\E[\tau^{+}]<\infty$, then $\P(\tau_0=\infty)>0$.
    \item $\P(\tau^{+}>n)$ is regularly varying with index $-\gamma\in(0,1)$, if and only if $\P(\tau_0>n)$ is regularly varying with index $-1+\gamma$.
\end{enumerate}
The last statement can be used, for example, to derive asymptotics for $\tau^+$ in the case of left-continuous walks. The distribution of $\tau_0$ has been studied in the previous section.

We now show that \eqref{WH} is valid also for pairs of stopping times, which a dual to each other. We now rigorously define this notion of duality for random walks in $\R^d$. From now on we assume that $\{X_k\}$ are independent, identically distributed random vectors in $\R^d$.
Furthermore, without restricting the generality, we shall assume that all the random variables are defined on the following 'standard' space of elementary events 
$\Omega=\{(x_1,x_2,\ldots), x_k\in\R^d\}$. Let $\theta_k$ denote the $k$-fold shift:
$$
\theta_k(x_1,x_2,\ldots)=(x_{k+1},x_{k+2},\ldots),\quad (x_1,x_2,\ldots)\in\Omega.
$$
For every $n\ge 1$ we define also the function
$$
r_n(x_1,x_2,\ldots,x_n,x_{n+1},\ldots)
=(x_n,x_{n-1},\ldots,x_1,x_{n+1},\ldots)
,\quad (x_1,x_2,\ldots)\in\Omega.
$$

Let $\tau$ be a stopping time for the walk $\{S(n)\}$. We define recursively 
$$
\tau_0(\omega)\equiv0,\ 
\tau_1(\omega)=\tau(\omega)
\quad\text{and}\quad
\tau_k(\omega):=\tau_{k-1}(\omega)+\tau(\theta_{\tau_{k-1}(\omega)}(\omega)),\quad k\ge2.
$$
Let $\mathcal{M}_\tau$ denote the random set $\{\tau_k\}$. We shall say that a stopping time $\eta$ is dual to the stopping time $\tau$ if 
\begin{equation}
\label{eq:def-duality}
\left\{\omega: n\in\mathcal{M}_\tau\right\}
=\left\{\omega:\eta(r_n(\omega))>n\right\} 
\quad\text{for every }n\ge0.
\end{equation}
This implies that if $\tau$ and $\tau'$ are dual to the same $\eta$ then
$\mathcal{M}_{\tau}(\omega)=\mathcal{M}_{\tau'}(\omega)$ for every $\omega$ and, consequently,
$\tau=\tau'$. Moreover, every $\tau$ is dual to at most one $\eta$. Theorem~1 in \cite{Greenwood_Shaked78} proves that if $\tau$ is dual to $\eta$ then $\eta$ is dual to
$\tau-$. (This fact is not obvious, since the definition \eqref{eq:def-duality} is not symmetric.) So, we can speak of pairs of dual stopping times.

It is easy to check that the stopping times $\tau_0$ and $\tau^+$ are dual to each other. 
By the symmetry arguments, $\inf\{n\ge1: S(n)<0\}$ and $\inf\{n\ge1: S(n)\ge 0\}$ are also dual to each other. We next show that $\tau_x$ has no dual provided that $x$ is large enough. To this end we prove the following property of the duality.
\begin{lemma}
\label{lem:dual-pairs}
A stopping time $\tau$ has a dual if and only if for each $n$ and $\omega$,
$n\in\mathcal{M}_\tau(\omega)$ implies that $j\in\mathcal{M}_\tau(\theta_{n-j}(\omega))$
for every $j\in\{1,2,\ldots,n\}$.
\end{lemma}
\begin{proof}
Assume first that $\tau$ possesses a dual stopping time $\eta$. Then, by the definition,
$$
n\in\mathcal{M}_\tau(\omega)
\ \Leftrightarrow \  
\eta(r_n(\omega))>n
\ \Leftrightarrow \  
\eta(r_n(\omega))>j\text{ for all }j\in\{1,2,\ldots,n\}.
$$
Since $\eta$ is a stopping time, 
$$
\{\omega: \eta(r_n(\omega))>j\}=\{\omega: \eta(r_j(\theta_{n-j}(\omega)))>j\},\quad j\le n.
$$
Using the definition of the duality once again, we have
$$
\{\omega: \eta(r_n(\omega))>j\}=
\{\omega:j\in\mathcal{M}_\tau(\theta_{n-j}(\omega))\},\quad j\le n.
$$
Consequently, we have the desired property: 
\begin{equation}
\label{eq:duality1}
\{\omega: n\in\mathcal{M}_\tau(\omega)\}
=\bigcap_{j=1}^n\{\omega:j\in\mathcal{M}_\tau(\theta_{n-j}(\omega))\}.
\end{equation}

Let us now assume that \eqref{eq:duality1} holds.
Set 
$$
A_n:=\{\omega: n\in\mathcal{M}_\tau(r_n(\omega))\},\quad n\ge0.
$$
Since $\tau$ is a stopping time, whether $\omega=(x_1,x_2,\ldots)$ belongs to $A_n$ depends on $x_1,x_2,\ldots,x_n$ only. If we now show that the sequence $A_n$ is monotone decreasing then we may define a stopping time $\eta$ be the level sets $\{\eta>n\}=A_n$, $n\ge0$. Clearly, this stopping time will be dual to $\tau$.

To show that $A_n$ decreases we notice that, according to \eqref{eq:duality1},
$$
A_n=\bigcap_{j=1}^n\{\omega:j\in\mathcal{M}_\tau(\theta_{n-j}(r_n(\omega)))\}.
$$

Noting that $\theta_{n-j}(r_n(\omega))=(x_j,x_{j-1},\ldots,x_1,x_{n+1},\ldots)$ and recalling that to decide whether $\omega$ belongs to $\{\omega:j\in\mathcal{M}_\tau(r_j(\omega))\}$ one needs to know $x_1,x_2,\ldots,x_j$ only, we conclude that 
$$
A_n=\bigcap_{j=1}^n\{\omega:j\in\mathcal{M}_\tau(r_j(\omega))\}.
$$
So, the sequence $A_n$ is decreasing and the proof is complete.
\end{proof}
Let $x>0$ be such that $\tau_x=\inf\{n\ge1: S(n)\le-x\}>\inf\{n\ge 1:S(n)<0\}$ with positive probability. This restriction implies that $\P(X_1>-x)>0$. Thus, it happens, with positive probability, that $n\in\mathcal{M}_{\tau_x}$ but $X_n>-x$ and, consequently,
$\tau_x(\theta_{n-1}(\omega))>1$. Applying Lemma~\ref{lem:dual-pairs}, we infer that $\tau_x$ has no dual.

\begin{lemma}
Let $\tau$ and $\eta$ be dual stopping times for a $d$-dimensional random walk $\{S(n)\}$. Let $T$ be an independent of $\{S(n)\}$ geometrically distributed random variable, $\P(T\ge n)=u^n$ for some $u\in(0,1]$. Define the measures $H_{\eta,u}$ and $G_{\tau,u}$ by the equalities
$$
H_{\eta,u}(A)=\P(S(\eta)\in A,\eta\le T),\quad
G_{\tau,u}(A)=\sum_{n=0}^\infty \P(S(n)\in A,\tau>n,T\ge n)
$$
for every Borel subset $A$ of $\R^d$.
Then 
$$
G_{\tau,u}=\sum_{k=0}^\infty H_{\eta,u}^{(*k)}
\quad\text{and}\quad 
(\delta_0-H_{\eta,u})*G_{\tau,u}=\delta_0\quad\text{for }u\in(0,1),
$$
where $\delta_0$ is a unit mass at zero. 
\end{lemma}
\begin{proof}
Let $A$ be a Borel subset of $\R^d$.
It is immediate from the definition of the duality that 
$$
\P(S(n)\in A,\tau>n)
=\sum_{k=1}^\infty \P(S(\eta_1+\eta_2+\ldots+\eta_k)\in A,\eta_1+\eta_2+\ldots+\eta_k=n)
$$
for every $n$. Multiplying this by $u^n$ and taking then the sum over all $n$, we get
\begin{align*}
&G_{\tau,u}(A)\\
&=\sum_{n=0}^\infty u^n\P(S(n)\in A,\tau>n)\\
&=\delta_0(A)+\sum_{k=1}^\infty\sum_{n=1}^\infty u^n
\P(S(\eta_1+\eta_2+\ldots+\eta_k)\in A,\eta_1+\eta_2+\ldots+\eta_k=n).
\end{align*}
Applying now the Markov property, we get 
\begin{align}
\label{eq:G-eq}
G_{\tau,u}(A)=\sum_{k=0}^\infty H_{\eta,u}^{(*k)}(A).
\end{align}
For every $u<1$, the measures $G_{\tau,u}$ and $H_{\eta,u}^{(*k)}$ are finite. 
This implies that \eqref{eq:G-eq} can be written as
$$
G_{\tau,u}=\delta_0+H*G.
$$
Thus, the second claim of the lemma is immediate consequence of the first one.
\end{proof}
We next prove a Spitzer-Pollaczek factorisation related to a pair of dual stopping times.
\begin{theorem}
\label{thm:SP-factor}
Let $\tau$ and $\eta$ be dual stopping times for the walk $\{S(n)\}$.
Let $F$ denote the distribution of $X_1$.
Then, for every $u\in(0,1]$,
$$
\delta_0-uF=(\delta_0-H_{\eta,u})*(\delta_0-H_{\tau,u}).
$$
\end{theorem}
\begin{proof}
We first assume that $u<1$.  Since $T$ is geometrically distributed and independent of the walk $\{S(n)\}$,
\begin{align}
\label{eq:repr1}
\nonumber
\P(S(T)\in A,\tau\ge T)
&=(1-u)\delta_0(A)+u\P(S(T)\in A,\tau\ge T|T\ge1)\\
\nonumber
&=(1-u)\delta_0(A)+u\P(S(T+1)\in A,\tau\ge T+1)\\
&=(1-u)\delta_0(A)+u(K_{\tau,u}*F)(A),
\end{align}
where $K_{\tau,u}(A)=\P(S(T)\in A,\tau>T)$. Furthermore,
$$
\P(S(T)\in A,\tau\ge T)
=K_{\tau,u}(A)+\P(S(T)\in A,\tau=T).
$$
Noting that 
\begin{align*}
\P(S(T)\in A,\tau=T)
&=\sum_{n=0}^\infty \P(T=n)\P(S(n)\in A,\tau=n)\\
&=(1-u)\sum_{n=0}^\infty \P(T\ge n)\P(S(n)\in A,\tau=n)\\
&=(1-u)\P(S(\tau)\in A,\tau\le T)=(1-u)H_{\tau,u}(A),
\end{align*}
we have
$$
\P(S(T)\in A,\tau\ge T)
=K_{\tau,u}(A)+(1-u)H_{\tau,u}(A).
$$
Combining this with \eqref{eq:repr1}, we conclude that 
$$
\frac{1}{1-u}K_{\tau,u}*(\delta_0-uF)=\delta_0-H_{\tau,u},\quad u<1.
$$
It follows from the definition of $K_{\tau,u}$ that 
\begin{align*}
\frac{1}{1-u}K_{\tau,u}(A)
&=\frac{1}{1-u}\sum_{n=0}^\infty \P(T=n)\P(S(n)\in A,\tau>n)\\
&=\sum_{n=0}^\infty \P(T\ge n)\P(S(n)\in A,\tau>n)\\
&=\sum_{n=0}^\infty \P(S(n)\in A,\tau>n,T\ge n)=G_{\tau,u}(A).
\end{align*}
This completes the proof in the case $u<1.$ Letting $u\uparrow 1$, we obtain the desired equality also in the case $u=1$. 
\end{proof}
One can rewrite the statement of Theorem~\ref{thm:SP-factor} in terms of double transforms. 
Let $\varphi(\lambda)$ denote the characteristic function of the vector $X_1$, that is,
$$
\varphi(\lambda)=\E e^{i(\lambda,X_1)},\quad \lambda\in\R^d.
$$
Then, for every pair $(\tau,\eta)$ of dual stopping times we have 
\begin{equation}
\label{eq:factor-transform}
1-u\varphi(\lambda)=
\left(1-\E\left[u^\tau e^{i(\lambda,S(\tau))}\right]\right)
\left(1-\E\left[u^\eta e^{i(\lambda,S(\eta))}\right]\right).
\end{equation}
Letting here $\lambda=0$, we obtain 
$$
1-u=(1-\E u^\tau)(1-\E u^\eta).
$$
This equality generalises \eqref{WH} to arbitrary pairs of dual stopping times.

Notice that the equation \eqref{eq:factor-transform} contains two unknown transforms. Formally, we can not determine both transforms from that equation. To solve it, we derive now a further factorisation for the function $1-u\varphi(\lambda)$. Let us partition $\R^d$ into half-spaces
$W_1$ and $W_2$. Then, for every $u<1$ one has
\begin{align*}
&\log(1-u\varphi(\lambda))\\
&\hspace{1cm}=\sum_{n=1}^\infty\frac{u^n}{n}\varphi^n(\lambda)
=\sum_{n=1}^\infty\frac{u^n}{n}\E e^{i(\lambda,S(n))}\\
&\hspace{1cm}=\sum_{n=1}^\infty\frac{u^n}{n}\E\left[e^{i(\lambda,S(n))};S(n)\in W_1\right]
+\sum_{n=1}^\infty\frac{u^n}{n}\E\left[e^{i(\lambda,S(n))};S(n)\in W_2\right].
\end{align*}
Consequently, for every $u<1$,
\begin{align}
\label{eq:second-factor}
\nonumber
1-u\varphi(\lambda)&=
\exp\left\{\sum_{n=1}^\infty\frac{u^n}{n}\E\left[e^{i(\lambda,S(n))};S(n)\in W_1\right]\right\}\\
&\hspace{2cm}\times\exp\left\{\sum_{n=1}^\infty\frac{u^n}{n}\E\left[e^{i(\lambda,S(n))};S(n)\in W_2\right]\right\}.
\end{align}
It turns out that one can claim that the components of factorisations in \eqref{eq:factor-transform} and in \eqref{eq:second-factor} are equal to each other. More precisely, one has the following result.
\begin{theorem}
\label{thm:components}
Let $\tau$ and $\eta$ be dual stopping times for the walk $\{S(n)\}$. If the values of $S(\tau)$ and $S(\eta)$ belong to two disjoint half-spaces $W_\tau$ and $W_\eta$ and if 
$W_\tau\cup W_\eta=\R^d$ then 
\begin{equation}
\label{eq:comp1}
1-\E\left[u^\tau e^{i(\lambda,S(\tau))}\right]
=\exp\left\{\sum_{n=1}^\infty\frac{u^n}{n}\E\left[e^{i(\lambda,S(n))};S(n)\in W_\tau\right]\right\}
\end{equation}
and 
\begin{equation}
\label{eq:comp2}
1-\E\left[u^\eta e^{i(\lambda,S(\eta))}\right]
=\exp\left\{\sum_{n=1}^\infty\frac{u^n}{n}\E\left[e^{i(\lambda,S(n))};S(n)\in W_\eta\right]\right\}.
\end{equation}
\end{theorem}
\begin{proof}
Set 
$$
\psi_\tau(u,\lambda)=\E\left[u^\tau e^{i(\lambda,S(\tau))}\right]
\quad\text{and}\quad 
\psi_\eta(u,\lambda)=\E\left[u^\eta e^{i(\lambda,S(\eta))}\right].
$$
For every $u<1$ we have 
$$
\log(1-\psi_\tau(u,\lambda))=\sum_{n=1}^\infty\frac{1}{n}\psi_\tau^n(u,\lambda)
$$
and 
$$
\log(1-\psi_\eta(u,\lambda))=\sum_{n=1}^\infty\frac{1}{n}\psi_\eta^n(u,\lambda).
$$
Combining this with the factorisations \eqref{eq:factor-transform} and \eqref{eq:second-factor},
we obtain
\begin{align*}
&\sum_{n=1}^\infty\frac{1}{n}\psi_\tau^n(u,\lambda)
+\sum_{n=1}^\infty\frac{1}{n}\psi_\eta^n(u,\lambda)\\
&\hspace{1cm}
=\sum_{n=1}^\infty\frac{u^n}{n}\E\left[e^{i(\lambda,S(n))};S(n)\in W_\tau\right]
+\sum_{n=1}^\infty\frac{u^n}{n}\E\left[e^{i(\lambda,S(n))};S(n)\in W_\eta\right].
\end{align*}
Using now the bijection between measures and their Fourier transforms, we conclude that 
$$
\sum_{n=1}^\infty\frac{1}{n}\psi_\tau^n(u,\lambda)
=\sum_{n=1}^\infty\frac{u^n}{n}\E\left[e^{i(\lambda,S(n))};S(n)\in W_\tau\right]
$$
and
\begin{equation}
\label{eq:tau-0-gener}
\sum_{n=1}^\infty\frac{1}{n}\psi_\eta^n(u,\lambda)
=\sum_{n=1}^\infty\frac{u^n}{n}\E\left[e^{i(\lambda,S(n))};S(n)\in W_\eta\right]
\end{equation}
for every $u<1$. This gives the desired equalities.
\end{proof}

Specialising Theorem~\ref{thm:components} to the stopping times $\tau^{+}$ and $\tau_{0}$
for a one-dimensional walk $\{S(n)\}$ and letting $\lambda=0$, we obtain
$$
1-\E\left[u^{\tau_+} \right]=\exp \left\{-\sum_{n=1}^{\infty} \frac{u^{n}}{n} \P(S(n)>0)\right\}
$$
and
\begin{align}
\label{eq:tau0-factor}
1-\E\left[u^{\tau_{0}}\right]=\exp \left\{-\sum_{n=1}^{\infty} \frac{u^{n}}{n} \P(S(n) \leq 0)\right\}.
\end{align}
These exact equalities are quite complicated, since we have no convenient exact expressions for $\P(S(n)>0)$ and $\P(S(n) \leq 0)$. But typically we know the asymptotic behaviour of that probabilities. The most classical case is when we assume that $\E\left[X_{1}\right]=0$ and 
$\operatorname{Var}\left[X_{1}\right]=: \sigma^{2} \in(0, \infty)$. 
Then, by the central limit theorem, $\P(S(n) \leq 0) \rightarrow 1 / 2$. 

For the stopping time $\tau_0$ we have the following result.
\begin{theorem}
\label{thm:tau-zero}
Assume that $\P(S(n)>0)\to\varrho\in(0,1)$ as $n\to\infty$. Then there exists a slowly varying function $\ell(n)$ such that 
$$
\P(\tau_0>n)=n^{\varrho-1}\ell(n).
$$
\end{theorem}
\begin{proof}
It follows from \eqref{eq:tau0-factor} that
\begin{align*}
\sum s^{n} \P\left(\tau_{0}>n\right)= & \frac{1-\E\left[s^{\tau_{0}}\right]}{1-s}=\exp \left\{-\log (1-s)-\sum_{n=1}^{\infty} \frac{s^{n}}{n} \P(S(n) \le 0)\right\} \\
= & \exp \left\{\sum_{n=1}^{\infty} \frac{s^{n}}{n} \P(S(n)>0)\right\} \\
= & (1-s)^{-\varrho} \exp \left\{\sum_{n=1}^{\infty} \frac{s^{n}}{n}\left(\P(S(n)>0)-\varrho\right)\right\}.
\end{align*}
The assumption $\P(S(n)>0)\to\varrho$ implies that the function
$$
L(x):=\exp \left\{\sum_{n=1}^{\infty} \frac{(1-1/x)^{n}}{n}\left(\P(S(n)>0)-\varrho\right)\right\}
$$
is slowly varying at infinity. Combining this with the equality
$$
\sum s^{n} \P\left(\tau_{0}>n\right)=(1-s)^{-\varrho}L\left(\frac{1}{1-s}\right)
$$
and applying the Tauberian theorem, see, for example, Theorem XIII.5.5 in \cite{Feller-book}, we conclude that, as $n\to\infty$,
$$
\P(\tau_0>n)\sim \frac{1}{\Gamma(\varrho)}n^{\varrho-1}L(n).
$$
This is equivalent to the claim in the theorem. 
\end{proof}
In the case of zero drift and finite variance, the condition of the above theorem holds with $\varrho=1/2$. Furthermore, one can show that
$$
\sum \frac{1}{n}\left(\P(S(n)>0)-\frac{1}{2}\right)<\infty
$$
for every random walk with zero drift and finite variance.
This summability property implies that $\ell(n)$ is asymptotically equivalent to a positive constant $C_0$ and, consequently,
$$
\P\left(\tau_{0}>n\right) \sim \frac{C_{0}}{\sqrt{n}} \text { as } n \rightarrow \infty .
$$

As we have mentioned before, stopping time $\tau_{x}$ with $x>0$ has no dual. This indicates that it is not possible to obtain a factorisation which involves $\tau_x$. But the knowledge on the tail behaviour of $\tau_0$ allows one to conclude that for every oscillating random walk there exists a positive function $H(x)$ such that
\[
\lim_{n\to \infty}\frac{\P\left(\tau_{x}>n\right)}{\P\left(\tau_{0}>n\right)} =H(x) \quad\text{ for every } x>0.
\]
We shall give a probabilistic proof of this equality in
Proposition~\ref{prop:V-H}.
The function $H$ is a renewal function of weak descending ladder heights:
$$
H(x)= \sum \P\left(\chi^-_{1}+\chi^-_2+\ldots+\chi^-_{k}<x\right).
$$

\vspace {6pt}

Wiener-Hopf factorisation is a very powerful tool in studying first-passage problems for one-dimensional processes with independent stationary increments (random walks and Levy processes). We gave above just one example of its usage, but there is a very large number of papers which use factorisation identities to study boundary crossing problems with one or two boundaries. The most classical references for basics of the Wiener-Hopf factorisation for random walks are textbook by Spitzer~\cite{spitzer1964principles} and by Borovkov~\cite{borovkov2013probability}.  For the case of Levy processes we refer to Doney~\cite{doney2007} and to Kyprianou~\cite{kyprianou2014fluctuations}.

\section{A Universality Approach to Exit Times}
\label{sec:universality}

In this section we shall describe an alternative approach to first-passage times for discrete time random walks. This approach is based on the following universality idea: if the random walk $\{S(n)\}$ belongs to the domain of attraction of the Brownian motion then the tail behaviour of first-passage times of $\{S(n)\}$ should be similar to that of the Brownian motion. It turns out that this idea is quite robust and allows one to consider random walks with time-inhomogeneous increments. It is worth recalling that the duality lemma holds only for identically distributed increments and, consequently, the Wiener--Hopf factorisation is not applicable in the case when
$\{X_k\}$ have different distributions.

Let $\{S(n)\}$ be a $1$-dimensional random walk with independent increments $\{X_k\}$:
\[
S(n)=X_1+\cdots+X_n,\quad n\ge 1.
\]
We shall assume that 
\begin{equation}
\label{eq:RW-moments}
\E X_k=0\quad\text{and}\quad \sigma_k^2:=\E X_k^2\in(0,\infty)
\quad\text{for every}\quad k\ge1.
\end{equation}
Set
$$
B_0^2:=0\quad\text{and}\quad
B_n^2=\sum_{k=1}^n\sigma_k^2,\quad n\ge 1.
$$
We shall also assume that the classical Lindeberg condition holds:
\begin{equation}
\label{eq:LC}
L_n(\varepsilon)
:=\frac{1}{B_n^2}\sum_{k=1}^n\E\left[X_k^2;|X_k|\ge\varepsilon B_n\right]
\to0\quad\text{as}\quad n\to\infty
\end{equation}
for every $\varepsilon>0$. It is well-known that this condition is necessary and sufficient for the validity of the functional central limit theorem. More precisely, if we set
\[
s(t):=S(k)+X_{k+1}\frac{t-B_k^2}{\sigma_{k+1}^2},\quad t\in [B_k^2, B_{k+1}^2],\; k\ge 0
\]
and
\[
s_n(t):=\frac{s(tB_n^2)}{B_n},\quad t\in[0,1],
\]
then $s_n$ converges weakly on $C[0,1]$ towards the standard Brownian motion $W$ if and only if \eqref{eq:LC} holds.

Since we want to use similarities between the walk $S(n)$ and the Brownian motion $W(t)$, let us first take a look at first-passage times for $W(t)$. Define
\begin{align*}
\tau_x^{(bm)} &= \inf\{t>0: x+W(t)\le 0\},\quad x>0. 
\end{align*}
Using the reflection principle,  
one can show that
\begin{align*}
\P\bigl(\tau_x^{(bm)}>1\bigr)&=\P\Bigl(\min_{0\le s\le 1}W(s)>-x\Bigr)\\
&=1-2\P\bigl(W(1)\ge x\bigr)
=2\int_0^x \frac{1}{\sqrt{2\pi}} e^{-u^2/2}\,du.
\end{align*}
This implies that
\[
\P\bigl(\tau_x^{(bm)}>1\bigr) \sim \sqrt{\frac{2}{\pi}}\, x
\quad\text{as }x\to0.
\]
Thus, by the scaling property of the Brownian motion, 
for every fixed $x>0$,
\begin{equation}
\label{eq:bm-tail}
\P\bigl(\tau_x^{(bm)}>t\bigr)\sim \sqrt{\frac{2}{\pi}}\,\frac{x}{\sqrt{t}}
=
\sqrt{\frac{2}{\pi}}\,\frac{-\E[W(\tau_x^{bm})]}{\sqrt{t}}
,\quad t\to\infty.
\end{equation}

It follows from the function central limit theorem that 
\begin{align*}
\P(\tau_{uB_n}>B_n^2)
&=\P\left(\min_{k\le n}S(k)>-uB_n\right)\\
&\sim\P\left(\min_{t\in[0,1]}W(t)>-u\right)=\P(\tau_u^{(bm)}>1)
\end{align*}
for every fixed $u>0$. Furthermore, since every convergence possesses a certain rate of convergence, we have
$$
\P(\tau_{u_nB_n}>B_n^2)\sim \P(\tau_{u_n}^{(bm)}>1)
\sim \sqrt{\frac{2}{\pi}}\, u_n
$$
provided that $u_n\to0$ sufficiently slow. But one can not immediately infer from the central limit theorem that the same relation is valid in the case of fixed starting point, which corresponds to $u_n=xB_n^{-1}$. So, our main purpose will be to find a way of deriving tail asymptotics for first-passage times of $S(n)$ from the functional central limit theorem. 
It turns out that, in contrast to the previous sections, we can consider not only fixed but also moving boundaries.
For a real-valued sequence $\{g(n)\}$ we define the stopping time
\[
T_g := \inf\{n\ge 1: S(n)\le g(n)\}.
\]
In particular, if $g(n)\equiv -x$, then $T_g=\tau_x$. 
We shall only assume that the boundary $\{g(n)\}$ satisfies 
\begin{equation}
\label{eq:bound-cond}
g(n)=o(B_n).
\end{equation}
This condition means that, from the point of view of the central limit theorem, the boundary is asymptotically zero.

\begin{theorem}\label{th:3.1}
Assume that \eqref{eq:RW-moments}, \eqref{eq:LC} and \eqref{eq:bound-cond} hold. If, in addition,
$\P(T_g>n)>0$ for every $n$ then 
\[
\P(T_g>n) \sim \sqrt{\frac{2}{\pi}}\,\frac{U_g(B_n^2)}{B_n},
\]
where
$$
U_{g}\left(B_{n}^{2}\right)=\E\left[S(n)-g(n); T_{g}>n\right] \sim \E\left[-S\left(T_{g}\right); T_{g} \le n\right]
$$
is a positive slowly varying function.
\end{theorem}

We have
$$
\P\left(T_{g}>u\right) 
\sim \sqrt{\frac{2}{\pi}} 
\frac{\E\left[-S\left(T_{g}\right), T_{g} \le n\right]}{B_{n}}
$$
Comparing this with \eqref{eq:bm-tail}, we see that the only difference is a "partial" expected value of $-S(T_{g})$. 
It may happen that $\E\left[-S\left(T_{g}\right);  T_{g} \leq n\right]$ does not converge to a positive constant. One example of this type will be discussed later and further examples can be found in \cite{denisov_sakhanenko_wachtel18}. 

We split the proof of Theorem~\ref{th:3.1} into several blocks, each block corresponds to one subsection below.
\subsection{Upper bounds for the tail of $T_g$.}
\begin{lemma}\label{lem:upper}
For every $n\ge 1$ and every $x\in\mathbb{R}$,
\[
\P\Bigl(S(n)>x,\; T_g>n\Bigr) \ge \P\bigl(S(n)>x\bigr) \P(T_g>n).
\]
\end{lemma}
\begin{proof}
For $x\le g(n)$ the inequality is trivial. 
For $x>g(n)$ we shall use the induction. 

If $n=1$ and $x>g(1)$ then
\[
\P\left(S(1)>x,\; T_g>1\right)=\P\left(S(1)>x\right)\ge
\P\left(S(1)>x\right)\P\left(T_g>1\right).
\]

Assume that the inequality holds for some $n\ge 1$. 
For every$x>g(n+1)$ we then have
\begin{align*}
& \P\left(S(n+1)>x, T_{g}>n+1\right) \\
&\hspace{1cm} 
 =\int_{\mathbb{R}} \P\left(S(n)+y>x, S(n)+y>g(n+1), T_{g}>n\right) \P\left(X_{n+1} \in d y\right) \\
&\hspace{1cm} 
 =\int_{\mathbb{R}} \P\left(S(n)+y>x, T_{g}>n\right) \P\left(X_{n+1} \in d y\right).
\end{align*}
Using the induction assumption, we obtain
\begin{align*}
& \P\left(S(n+1)>x, T_{g}>n+1\right) \\
&\hspace{1cm}
 \geq \int_{\mathbb{R}} \P(y+S(n)>x) \P\left(T_{g}>n\right) \P\left(X_{n+1} \in d y\right) \\
&\hspace{1cm}
 =\P(S(n+1)>x) \P\left(T_{g}>n\right) \ge \P(S(n+1)>x) \P\left(T_{g}>n+1\right).
\qedhere
\end{align*}
\end{proof}

\begin{lemma}\label{lem:upper2}
For every $n\ge 1$,
\[
\P(T_g>n) \le \frac{\E\bigl[S(n)-g(n);\,T_g>n\bigr]}{\E\bigl[S(n)-g(n);\,S(n)>g(n)\bigr]}.
\]
Furthermore, if $g(n)=o(B_n)$ then there exists a constant $C>0$ such that
\[
\P(T_g>n) \le C\, \frac{\E\bigl[S(n)-g(n);\,T_g>n\bigr]}{B_n}
=C\frac{U_g(B_n^2)}{B_n}
.
\]
\end{lemma}
\begin{proof}
Write
\[
\E\bigl[S(n)-g(n);\,T_g>n\bigr]=\int_0^\infty \P\Bigl(S(n)-g(n)>x,\;T_g>n\Bigr)\,dx.
\]
Using Lemma~\ref{lem:upper} yields
\begin{align*}
\E\bigl[S(n)-g(n);\,T_g>n\bigr] &\ge \P(T_g>n)\int_0^\infty \P\bigl(S(n)-g(n)>x\bigr)\,dx\\
&=\E\bigl[S(n)-g(n);\,S(n)>g(n)\bigr].
\end{align*}
Thus, the first claim is proved.

To prove the second claim, we have to estimate the denominator from below. Since $\frac{S(n)}{B_n}$ converges towards $W(1)$ and $g(n)=o\left(B_{n}\right)$, 
the sequence $\frac{S(n)-g(n)}{B_{n}}$ converges to the same limit. Then, by Fatou's lemma,
\begin{align*}
& \liminf _{n \rightarrow \infty} \frac{1}{B_{n}} \E[S(n)-g(n) ; S(n)-g(n)>0] \\
&\hspace{1cm} =\liminf _{n \rightarrow \infty} \E\left[\left(\frac{S(n)-g(n)}{B_{n}}\right)^{+}\right] \ge \E\left[W^{+}(1)\right] \\
&\hspace{1cm} =\int_{0}^{\infty} x \frac{1}{\sqrt{2 \pi}} e^{-x^{2} / 2} d x=\frac{1}{\sqrt{2 \pi}} .
\end{align*}
This implies the existence of a positive constant $c$ such that 
$$
\E\left[S(n)-g(n), T_{g}>n\right] \ge \frac{1}{c} B_{n} 
$$
for all $n\ge 1$. This completes the proof of the lemma.
\end{proof}

\subsection{Some martingale identities.}
\begin{lemma}\label{lem:3.5}
For every $m\ge 1$,
\[
\E\bigl[S(m)-g(m);\,T_g>m\bigr] = -\E\Bigl[S(T_g);\,T_g\le m\Bigr] - g(m) \P(T_g>m).
\]
and 
$$
\E\left[S\left(\nu_{m}\right)-g\left(\nu_{m}\right) ; T_{g}>\nu_{m}\right]=-\E\left[S\left(T_{g}\right) ; T_{g} \le \nu_{m}\right]-\E\left[g\left(\nu_{m}\right) ; T_{g}>\nu_{m}\right],
$$
where 
\[
\nu_m = \min\Bigl\{ m,\; \inf\{ k\ge 1: S(k)-g(k) > B_m\}\Bigr\}. 
\]
\end{lemma}
\begin{proof}
Let $\theta$ be a bounded stopping time. By the optional stopping theorem applied to the martingale $S(n)$, we have
\[
0 = \E\bigl[S(T_g\wedge\theta)\bigr]
= \E\Bigl[S(T_g);\,T_g\le \theta\Bigr] + \E\Bigl[S(\theta);\,T_g>\theta\Bigr].
\]
Consequently,
$$
\E\left[S(\theta) ; T_{g}>\theta\right]=-\E\left[S\left(T_{g}\right) ; T_{g} \le \theta\right]
$$
and
$$
\E\left[S(\theta)-g(\theta);T_g>\theta\right]=-\E\left[S\left(T_{g}\right);T_{g} \le \theta\right]-\E\left[g(\theta) ; T_{g}>\theta\right] .
$$
Taking here $\theta=m$ and $\theta=\nu_{m}$, we obtain the desired equalities.
\end{proof}
Set 
$$
G(n):=\max _{k \leq n}|g(k)|.
$$
\begin{corollary}\label{cor:3.6}
For all integers $n\ge m\ge 1$, the following estimates hold:
$$
\E\left[S\left(\nu_{m}\right)-g\left(\nu_{m}\right) ; T_{g}>\nu_{m}\right]-\E\left[S(n)-g(n) ; T_{g}>n\right] \le 2 G(n) \P\left(T_{g}>\nu_m\right), 
$$
$$
\E\left[S(m)-g(m) ; T_{g}>m\right]-\E\left[S(n)-g(n) ; T_{g}>n\right] \le 2 G(n) \P\left(T_{g}>m\right),
$$
\begin{align*}
&\left|\E\left[S\left(\nu_{m}\right)-g\left(\nu_{m}\right) ; T_{g}>m\right]-\E\left[S(n)-g(n) ; T_{g}>n\right]\right| \\
&\hspace{2cm}\le 2 G(n) \P\left(T_{g}>\nu_{m}\right)+\E\left[-\left(S\left(T_{g}\right)-g\left(T_{g}\right)\right) ; \nu_{m}<T_{g} \le n\right]
\end{align*}
and
\begin{multline*}
\max _{m \le k \le n}\left|\E\left[S(k)-g(k) ; T_{g}>k\right]-\E\left[S(n)-g(n) ; T_{g}>n\right]\right| \\
\le 2 G(n) \P\left(T_{g}>m\right)+\E\left[-\left(S\left(T_{g}\right)-g\left(T_{g}\right)\right) ; m<T_{g} \le n\right].
\end{multline*}
\end{corollary}
\begin{proof}
By Lemma~\ref{lem:3.5},
\begin{align*}
& \E {\left[S\left(\nu_{m}\right)-g\left(\nu_{m}\right) ; T_{g}>\nu_{m}\right]-\E\left[S(n)-g(n) ; T_{g}>n\right] } \\
&\hspace{1cm}=-\E\left[S\left(T_{g}\right) ; T_{g} \le \nu_{m}\right]-\E\left[g\left(\nu_{m}\right) ; T_{g}>\nu_{m}\right]\\
&\hspace{2cm}+\E\left[S\left(T_{g}\right) ; T_{g} \le n\right]+g(n) \P\left(T_{g}>n\right) \\
&\hspace{1cm}=\E\left[S\left(T_{g}\right) ; \nu_{m}<T_{g} \le n\right]-\E\left[g\left(\nu_{m}\right) ; T_{g}>\nu_{m}\right]+g(n) \P\left(T_{g}>n\right) \\
&\hspace{1cm}=\E\left[S\left(T_{g}\right)-g\left(T_{g}\right)_{;} \nu_{m}<T_{g} \le n\right] \\
&\hspace{2cm}+\E\left[g\left(T_{g}\right)-g\left(\nu_{m}\right) ; \nu_{m}<T_{g} \le n\right]+\E\left[g(n)-g\left(\nu_{m}\right) ; T_{g}>n\right] .
\end{align*}
This equality implies the first and the third estimates due to
$$
S\left(T_{g}\right)-g\left(T_{g}\right) \le 0
$$
and 
\begin{align*}
 \left\lvert\left(g\left(T_{g}\right)-g\left(\nu_{m}\right)\right) \mathbf{1}\left\{ \nu_{m}<T_{g} \le n\right\} \right\rvert 
 \le 2 G(n) \mathbf{1}\left\{\nu_{m}<T_{g} \le n\right\},\\
\left\lvert \left(g(n)-g(\nu_{m})\right) \mathbf{1} \left\{T_{g}>n\right\} \right\rvert 
\le 2 G(n) \mathbf{1}\left\{T_{g}>n\right\}.
\end{align*}
Similarly,
\begin{align*}
& \E\left[S(k)-g(k) ; T_{g}>k\right]-\E\left[S(n)-g(n) ; T_{g}>n\right] \\
& =\E\left[S\left(T_{g}\right)-g\left(T_{g}\right) ; k<T_{g} \le n\right] \\
& \quad+\E\left[g\left(T_{g}\right)-g(k); k<T_{g} \leq n\right]+(g(n)-g(k)) \P\left(T_{g}>n\right).
\end{align*}
This yields the second bound. To get the last bound we notice that
\begin{align*}
& \left|\E\left[S(k)-g(k); T_{g}>k\right]-\E\left[S(n)-g(n) ; T_{g}>n\right]\right| \\
& \quad \le 2 G(n) \P\left(T_{g}>k\right)+\E\left[g\left(T_{g}\right)-S\left(T_{g}\right); k<T_{g} \leq n\right] .
\end{align*}
Noting that the right hand side is decreasing in $k$, we complete the proof.
\end{proof}
We next provide upper bounds for expected values of the overshoot $S(T_g)-g(T_g)$ which appears on the right hand sides of the last two estimates of Corollary~\ref{cor:3.6}. 
Set 
$$
\varepsilon_n:=\inf\left\{\varepsilon>0:L_n^2(\varepsilon)\le\varepsilon^2\right\}.
$$
If the Lindeberg condition holds then, clearly, $\varepsilon_n\to0$ as $n\to\infty$.
\begin{lemma}\label{lem:overshoot}
For all $n\ge m\ge 1$ one has 
$$
\E\left[-\left(S\left(T_{g}\right)-g\left(T_{g}\right)\right) ; \nu_{m}<T_{g} \le n\right]
\le 2\left(G(n)+\varepsilon_n B_n\right)\P(T_g>\nu_m)
$$
and
$$
\E\left[-\left(S\left(T_{g}\right)-g\left(T_{g}\right)\right) ; m<T_{g} \le n\right]
\le 2\left(G(n)+\varepsilon_n B_n\right)\P(T_g>m).
$$
\end{lemma}
\begin{proof}
We prove the first inequality only. The proof of the second one is very similar and, in some parts, even simpler.

Since $S(T_g-1)>g(T_g-1)$,
\begin{align*}
-(S(T_g)-g(T_g))&=-X_{T_g}-(S(T_g-1)-g(T_g-1))+g(T_g)-g(T_g-1)\\
&\le -X_{T_g}+g(T_g)-g(T_g-1).
\end{align*}
Consequently,
\begin{align*}
&\E\left[-\left(S\left(T_{g}\right)-g\left(T_{g}\right)\right) ; \nu_{m}<T_{g} \le n\right]\\
&\hspace{1cm}\le \E\left[-X_{T_g}+g(T_g)-g(T_g-1);\nu_{m}<T_{g} \le n\right] \\
&\hspace{1cm}\le \E\left[-X_{T_g};\nu_{m}<T_{g} \le n\right]+2G(n)\P(T_g>\nu_m).
\end{align*}
Thus, it remains to show that 
$$
\E\left[-X_{T_g};\nu_{m}<T_{g} \le n\right]\le 2\varepsilon _n\P(T_g>\nu_m).
$$
Fix some $\varepsilon>0$. Then
\begin{align*}
&\E\left[-X_{T_g};\nu_{m}<T_{g} \le n\right]\\
&\hspace{1cm}
\le \varepsilon B_n\P(T_g>\nu_m)
+\E\left[-X_{T_g};-X_{T_g}>\varepsilon B_n,\,\nu_{m}<T_{g} \le n\right].
\end{align*}

For the expected value on the right hand side we have 
\begin{align*}
&\E\left[-X_{T_g};-X_{T_g}>\varepsilon B_n,\,\nu_{m}<T_{g} \le n\right]\\
&\hspace{1cm}=\sum_{j=2}^n\E\left[-X_{j};-X_{j}>\varepsilon B_n,\,T_{g}=j>\nu_{m}\right] \\
&\hspace{1cm}\le\sum_{j=2}^n \E\left[-X_{j};-X_{j}>\varepsilon B_n\right]
\P(T_g>j-1,\nu_m\le j-1)\\
&\hspace{1cm}\le\P(T_g>\nu_m)\sum_{j=2}^n \E\left[-X_{j};-X_{j}>\varepsilon B_n\right].
\end{align*}
Applying now the Markov inequality, we obtain 
\begin{align*}
&\E\left[-X_{T_g};-X_{T_g}>\varepsilon B_n,\,\nu_{m}<T_{g} \le n\right]\\
&\hspace{1cm}
\le \P(T_g>\nu_m)\frac{1}{\varepsilon B_n}\sum_{j=1}^n\E[X_j^2;|X_j|>\varepsilon B_n]
=\P(T_g>\nu_m)\frac{L_n^2(\varepsilon)}{\varepsilon}B_n.
\end{align*}
Consequently,
$$
\E\left[-X_{T_g};\nu_{m}<T_{g} \le n\right]
\le\left(\frac{L_n^2(\varepsilon)}{\varepsilon^2}+1\right)\varepsilon B_n
\P(T_g>\nu_m).
$$
Letting $\varepsilon\to\varepsilon_n$ and recalling the definition of $\varepsilon_n$, we get 
$$
\E\left[-X_{T_g};\nu_{m}<T_{g} \le n\right]
\le 2\varepsilon_n B_n\P(T_g>\nu_m).
$$
This completes the proof of the lemma.
\end{proof}
\begin{lemma}\label{lem:3.7}
If $m\le n$ and $B_m \ge R\, G(n)$ for some sufficiently large $R$, then
there exists a constant $C$ such that 
\[
\P\Bigl(T_g>\nu_m\Bigr) \le \P\Bigl(T_g>m\Bigr) \le
C\,\frac{\E\bigl[S(n)-g(n);\,T_g>n\bigr]}{B_m}.
\]
\end{lemma}
\begin{proof}
By the definition of $\nu_m$, $S(\nu_m)-g(\nu_m)>B_m$ on the event $\{\nu_m<m\}$.
This implies that 
\begin{align*}
\P\left(T_{g}>\nu_m\right)&=\P\left(T_{g}>\nu_{m}=m\right)+\P\left(T_{g}>\nu_{m}, \nu_{m}<m\right)\\
&\leq 
\P\left(T_{g}>m\right)+\P\left(S\left(\nu_{m}\right)-g\left(\nu_{m}\right)>B_{m}, T_g>\nu_m\right).
\end{align*}
We have shown in Lemma~{lem:upper2} that 
\[\P(T>m) \le C_{1} \frac{\E\left[S(m)-g(m));T_g>m\right]}{B_{m}}.
\]
Furthermore, by the Markov inequality,
$$
\P\left(S\left(\nu_{m}\right)-g\left(\nu_{m}\right)>B_{m}, T_{g}>\nu_{m}\right) \le \frac{\E\left[S\left(\nu_{m}\right)-g\left(\nu_{m}\right); T_{g}>\nu_{m}\right]}{B_{m}} .
$$
According to Corollary~\ref{cor:3.6},
$$
\E\left[S\left(\nu_m\right)-g\left(\nu_{m}\right); T_{g}>\nu_{m}\right] \le \E\left[S(n)-g(n) ; T_{g}>n\right]+2 G(n) \P\left(T_{g}>\nu_{m}\right)
$$
and
$$
\E\left[S(m)-g(m) ; T_{g}>m\right] \le \E\left[S(n)-g(n) ; T_{g}>n\right]+2 G(n) \P\left(T_{g}>m\right).
$$
Using these bounds and noting that $\P\left(T_{g}>m\right) \le \P\left(T_{g}>\nu_{m}\right)$, we have
$$
\P\left(T_{g}>\nu_{m}\right) \le\left(C_{1}+1\right) \frac{\E\left[S(n)-g(n); T_{g}>n\right]}{B_{m}}+2\left(C_{1}+1\right) \frac{G(n)}{B_{m}} \P\left(T_{g}>\nu_{m}\right).
$$
If we choose $R>4(C_1+1)$ then the assumption $B_m>R G(n)$ implies that
\[
\P\left(T_{g}>\nu_m\right) \le 2\left(C_{1}+1\right) 
\frac{\E\left[S(n)-g(n) ; T_{g}>n\right]}{B_{m}}. 
\]
Thus, the proof is complete.
\end{proof}

\begin{corollary}\label{cor:3.8}
Under the conditions of Lemma~\ref{lem:3.7},
\begin{align*}
&\left|\E\Bigl[S(\nu_m)-g(\nu_m);\,T_g>\nu_m\Bigr]-E\Bigl[S(n)-g(n);\,T_g>n\Bigr]\right|\\
&\hspace{1cm}
\le C(G(n)+\varepsilon_n)E\Bigl[S(n)-g(n);\,T_g>n\Bigr]
\end{align*}
and 
\begin{align*}
&\max_{m\le k\le n}\left|\E\Bigl[S(k)-g(k);\,T_g>k\Bigr]-E\Bigl[S(n)-g(n);\,T_g>n\Bigr]\right|\\
&\hspace{1cm}
\le C(G(n)+\varepsilon_n)E\Bigl[S(n)-g(n);\,T_g>n\Bigr].
\end{align*}
\end{corollary}
\begin{proof}
Applying Lemma~\ref{lem:3.7} to the right hand sides of inequalities in Lemma~\ref{lem:overshoot}, we obtain the desired estimates.
\end{proof}

\subsection{Estimates in the boundary problem.}
In this paragraph we are going to use the central limit theorem and to obtain a representation for the probability $\P(T_g>n)$, which will then lead to the claim of Theorem~\ref{th:3.1}.

For all $1\le k\le m<n$ and $y>0$ we set
\[
Q_{k,n}(y):= \P\Bigl(y+\min_{k\le i\le n}\bigl\{Z_i-Z_k\bigr\}>0\Bigr),
\]
where
\[
Z_j := S(j)-g(j).
\]
Applying the strong Markov property, one obtains
\begin{align}
\label{eq:Tg-repr}
\nonumber
\P\left(T_{g}>n\right) & =\P\left(\min _{1 \le j\le n} Z_{j}>0\right) \\
\nonumber
& =\P\left(T_{g}>\nu_{m}, Z_{\nu_{m}}+\min _{\nu_m \le j \le n}\left(Z_{j}-Z_{\nu_{m}}\right)>0\right)\\
& =\E\left[Q_{\nu_{m, n}}\left(Z_{\nu_m}\right) ; T_{g}>\nu_{m}\right] .
\end{align}
To analyse the right hand side in \eqref{eq:Tg-repr} one needs good estimates for $Q_{k,n}$.
We will obtain such bounds from the central limit theorem. Next lemma is an immediate consequence
of the central limit theorem.
\begin{lemma}\label{lem:3.9}
For each $n \ge 1$ we can define $\{S(n)\}_{n \ge 1}$ and a Brownian motion $W_{n}(t)$ on a common probability space so that
$$
\begin{aligned}
& \P\left(\max _{t \leq B_{n}^{2}}\left|s(t)-W_{n}(t)\right|>\pi_{n} B_{n}\right) \\
&\hspace{1cm} =\P\left(\max _{t \leq 1}\left|s_{n}(t)-\frac{1}{B_{n}} W_{n}\left(t B_{n}^{2}\right)\right|>\pi_{n}\right) \leq \pi_{n}
\end{aligned}
$$
for some sequence $\pi_{n} \downarrow 0$.
\end{lemma}

Set 
$$
B_{k, n}^{2}:=B_{n}^{2}-B_{k}^{2}>0
\quad\text{and}\quad 
\varepsilon_{k, n}:=\frac{\pi_{n} B_{n}+G(n)}{B_{k, n}}.
$$
Define also 
$$
Q(y):=P\left(y+\min _{t \leq 1} W(t)>0\right)=2 \int_{0}^{y^+} \varphi(u) d u,
\quad y \in \R.
$$
In the next lemma we compare $Q_{k,n}$ and $Q$.
\begin{lemma}\label{lem:3.10}
For all $k<n$ and $y \ge 0$ one has 
$$
\left|Q_{k,n}(y)-Q\left(\frac{y}{B_{k,n}}\right)\right| \le \pi_{n}+4 \varphi(0) \varepsilon_{k,n}.
$$
\end{lemma}
\begin{proof}
For every $k <n$ we define
\begin{align*}
q_{k, n}(y): & =\P\left(y+\min _{k \le j \le n}(S(j)-S(k))>0\right) \\
& =\P\left(y+\min _{B_{k}^{2} \le t \le B_{n}^{2}}\left(s(t)-s\left(B_{k}^{2}\right)\right)>0\right).
\end{align*}
Noting that $\left|\left(Z_{j} -Z_{k}\right)-(S(j)-S(k))\right| \leq 2G(n)$
and setting $y_\pm=y\pm2G(n)$, we obtain
\[
q_{k,n}(y_{-}) \le 
Q_{k,n}(y) \le q_{k,n} (y_+).
\]
Applying Lemma~\ref{lem:3.9}, we have
\begin{align*}
q_{k, n}\left(y_{+}\right) &\le\pi_{n}+\P\left(y_{+}+\min _{B_{k}^{2} \leq t \le B_{n}^{2}}\left(W_{n}(t)-W_{n}\left(B_{n}^{2}\right)\right)>-2 \pi_{n} B_{n}\right) \\
& =\pi_{n}+\P\left(\frac{y_{+}+2 \pi_{n} B_{n}}{B_{k, n}}+\min _{t \leq 1} W(t)>0\right) \\
& =Q\left(\frac{y}{B_{k, n}}+2 \varepsilon_{k, n}\right)+\pi_{n}.
\end{align*}
It is immediate from the definition of $Q$ that
$$
Q\left(\frac{y}{B_{k, n}}+2 \varepsilon_{k, n}\right) \le Q\left(\frac{y}{B_{k_{i}, n}}\right)+4 \varphi(0) \varepsilon_{k, n} .
$$
As a result we have
$$
Q_{k, n}(y) \le Q\left(\frac{y}{B_{k, n}}\right)+4 \varphi(0) \varepsilon_{k, n}+\pi_{n}
$$
Similar arguments give the lower bound
$$
Q_{k,n}(y) \ge Q\left(\frac{y}{B_{k, n}}\right)-4 \varphi(0) \varepsilon_{k, n}-\pi_{n},
$$
which finishes the proof of the lemma.
\end{proof}
\begin{lemma}\label{lem:3.11}
If $m$ is such that $B_{m} \le \frac{3}{5} B_{n}$ then, for all $k \le m$,
$$
\left|B_{n} Q_{k, n}(y)-2 y \varphi(0)\right| \le\left(3 \pi_{n}+\frac{2 G(n)}{B_{n}}\right) B_{n}+2 y \frac{B_{m}^{2}}{B_{n}^{2}}+y \mathbf{1}\left\{y \ge 3 B_{m}\right\}.
$$
\end{lemma}
\begin{proof}
By the triangle inequality,
\begin{align*}
&\left|B_{n} Q_{k, n}(y)-2 y \varphi(0)\right|\\
&\hspace{1cm}
\le \left|B_{n} Q_{k, n}(y)-B_{n} Q\left(\frac{y}{B_{k,m}}\right)\right|
+\left|B_{n} Q\left(\frac{y}{B_{k, n}}\right)-2 y \varphi(0)\right|.
\end{align*}

If $B_{m} \le \frac{3}{5} B_{n}$ then
\begin{align*}
B_{k, n} \ge B_{m, n} \ge \frac{4}{5} B_{n}.
\end{align*}
Using this bound and noting that $\varphi(0)=\frac{1}{\sqrt{2\pi}}<\frac{2}{5}$, we obtain
\begin{align*}
\quad \pi_{n}+4 \varphi(0)\varepsilon_{k, n} \le \pi_{n}+4 \cdot \frac{2}{5} \cdot \frac{\pi_{n} B_{n}+G(n)}{\frac{4}{5} B_{n}} 
 =3 \pi_{n}+\frac{2G(n)}{B_{n}}.
\end{align*}
Combining these estimates with Lemma~\ref{lem:3.10}, we conclude that 
$$
\left|B_{n} Q_{k, n}(y)-B_{n} Q\left(\frac{y}{B_{k m}}\right)\right| 
\le \left(3 \pi_{n}+\frac{2 G(n)}{B_{n}}\right)B_n.
$$
Thus, it remains to show that 
\begin{align}
\label{eq:remains}
\left|B_{n} Q\left(\frac{y}{B_{k, n}}\right)-2 y \varphi(0)\right|
\le 2 y \frac{B_{m}^{2}}{B_{n}^{2}}+y \mathbf{1}\left\{y \ge 3 B_{m}\right\}.
\end{align}
Since $Q(y) \leq 2 y \varphi(0)$ and $B_{k,n}\ge\frac{4}{5}B_n$,
\begin{align}
\label{eq:ub-remains}
\nonumber
B_{n} Q\left(\frac{y}{B_{k, n}}\right) & -2 y \varphi(0) \le 2 y \varphi(0)\left(\frac{B_{n}}{B_{k, n}}-1\right) \\
& \le y \frac{B_{n}^{2}-B_{k, n}^{2}}{B_{k, n}\left(B_{k, n}+B_{n}\right)}=y \frac{B_{k}^{2}}{\frac{4}{5}\left(1+\frac{4}{5}\right) B_{n}^{2}} \le y \frac{B_{m}^{2}}{B_{n}^{2}}.
\end{align}
Clearly,
$$
B_{n} Q\left(\frac{y}{B_{k,n}}\right)-2 y \varphi(0) \ge-2 y \varphi(0) \ge-y .
$$
Combining this with \eqref{eq:ub-remains}, we see that \eqref{eq:remains} holds for $y>3B_m$. 

Assume now that $y \le 3 B_m$.
Noting that 
$$
\varphi(x)=\varphi(0)e^{-x^2/2} \ge \varphi(0)\left(1-x^{2} / 2\right),
$$
we obtain
$$
Q(y)=2 \int_{0}^{y} \varphi(x) d x \ge 2 \varphi(0) \int_{0}^{y}\left(1-\frac{x^{2}}{2}\right) d x \ge 2 y \varphi(0)-\varphi(0) \frac{y^{3}}{3}.
$$
Thus,
\begin{align*}
& B_{n} Q\left(\frac{y}{B_{k, n}}\right)-2 y\varphi(0) \ge B_{n} Q\left(\frac{y}{B_{n}}\right)-2 y \varphi(0) \\
& \hspace{1cm} \ge-B_{n} \frac{\varphi(0)}{3}\left(\frac{y}{B_{n}}\right)^{3} \ge-\frac{\varphi(0)}{3} \frac{\left(3 B_{m}\right)^{2}}{B_{n}^{2}} y=-3 \varphi(0) \frac{B_{m}^{2}}{B_{n}^{2}} y
\ge -2\frac{B_{m}^{2}}{B_{n}^{2}} y.
\end{align*}
This completes the proof of the lemma.
\end{proof}
Plugging the inequality in Lemma~\ref{lem:3.11} into \eqref{eq:Tg-repr}, we conclude that
if $B_{m} \le \frac{3}{5} B_{n}$ then
\begin{align}
\label{eq:Tg-approx}
\nonumber
& \left|B_{n} \P\left(T_{g}>n\right)-2 \varphi(0) \E\left[S\left(\nu_{m}\right)-g\left(\nu_{m}\right) ; T_{g}>\nu_{m}\right]\right| \\
\nonumber
&\hspace{1cm} \le\left(3 \pi_{n}+\frac{2 G(n)}{B_{n}}\right) B_{n} 
\P\left(T_{g}>\nu_{m}\right)+2 \frac{B_{m}^{2}}{B_{n}^{2}} \E\left[S\left(\nu_{m}\right)-g\left(\nu_{m}\right) ; T_{g}>\nu_{m}\right] \\
&\hspace{2cm} +\E\left[S\left(\nu_{m}\right)-g\left(\nu_{m}\right) ; T_{g}>\nu_{m}, S\left(\nu_{m}\right)-g\left(\nu_{m}\right)>3 B_{m}\right] . 
\end{align}
Thus, we need to show that every summand on the right hand side is negligibly small in comparison with $\E\left[S\left(\nu_{m}\right)-g\left(\nu_{m}\right) ; T_{g}>\nu_{m}\right]$
for an appropriately chosen $m=m(n)$.
\subsection{Proof of Theorem~\ref{th:3.1}.}
The Lindeberg condition \eqref{eq:LC} implies that $\sigma_n^2=o(B_n^2)$. Combining this property
with the assumption $G(n)=o(B_n)$, we infer that there exists $m(n)<n$ such that 
\begin{align}
\label{eq:m_n-prop}
\frac{B_{m(n)}}{B_n}\to0
\quad\text{and}\quad 
\frac{G(n)+\varepsilon_n+\pi_nB_n}{B_{m(n)}}\to0.
\end{align}
Applying Corollary~\ref{cor:3.8} with this $m(n)$, we conclude that 
\begin{equation}
\label{eq:E-nu-asymp}
\E\left[S\left(\nu_{m(n)}\right)-g\left(\nu_{m(n)}\right) ; T_{g}>\nu_{m(n)}\right]
\sim \E\left[S\left(n\right)-g\left(n\right) ; T_{g}>n\right]
\end{equation}
and
\begin{equation}
\label{eq:Ug-slow}
\max_{m(n)\le k\le n}\left|\frac{\E\left[S\left(n\right)-g\left(n\right) ; T_{g}>n\right]}{\E\left[S\left(k\right)-g\left(k\right) ; T_{g}>k\right]}-1\right|=o(1).
\end{equation}
The latter relation implies that the function $U_g(x)$ is slowly varying.

Combining the first relation in \eqref{eq:m_n-prop} and \eqref{eq:E-nu-asymp}, we 
\begin{align}
\label{eq:part1}
\nonumber
&2\frac{B_{m(n)}^2}{B_m^2}
\E\left[S\left(\nu_{m(n)}\right)-g\left(\nu_{m(n)}\right) ; T_{g}>\nu_{m(n)}\right]\\
&\hspace{1cm}
=o\left(\E\left[S\left(n\right)-g\left(n\right) ; T_{g}>n\right]\right).
\end{align}
Furthermore, applying first Lemma~\ref{lem:3.7} and using then the second relation in \eqref{eq:m_n-prop}, we conclude that 
\begin{align}
\label{eq:part2}
\nonumber
&\left(3 \pi_{n}+\frac{2 G(n)}{B_{n}}\right) B_{n}\P\left(T_{g}>\nu_{m(n)}\right)\\
\nonumber
&\hspace{1cm}\le C\left(3\pi_nB_n+G(n)\right)
\frac{\E\left[S\left(n\right)-g\left(n\right) ; T_{g}>n\right]}{B_{m(n)}}\\
&\hspace{2cm} =o\left(\E\left[S\left(n\right)-g\left(n\right) ; T_{g}>n\right]\right).
\end{align}

Noting that 
\begin{align*}
S(\nu_m)-g(\nu_m)
&=S(\nu_m-1)+X_{\nu_m}-g(\nu_m-1)+g(\nu_m-1)-g(\nu_m)\\
&\le B_m+X_{\nu_m}+2G(m).
\end{align*}
The assumption $G(m)=o(B_m)$ implies that 
$$
S(\nu_m)-g(\nu_m)\le\frac{3}{2}B_m+X_{\nu_m}
$$
for all sufficiently large $m$. Consequently,
\begin{align*}
 &\E\left[S\left(\nu_{m}\right)-g\left(\nu_{m}\right) ; T_{g}>\nu_{m}, S\left(\nu_{m}\right)-g\left(\nu_{m}\right)>3 B_{m}\right]\\
 &\hspace{1cm}\le\sum_{j=1}^m\E\left[\frac{3}{2}B_m+X_j;T_g>j-1,X_j>\frac{3}{2}B_m\right]\\
 &\hspace{1cm}\le 2\sum_{j=1}^m\E\left[X_j;X_j>\frac{3}{2}B_m\right]\P(T_g>j-1).
\end{align*}
Using Lemma~\ref{lem:upper2} and applying Potter's bound for slowly varying functions, we obtain
\begin{align*}
\P(T_g>j-1)\le C_1\frac{U_g(B_{j-1}^2)}{B_{j-1}} 
\le C_2\frac{B_m^{1/3}U_g(B_m^2)}{B_{j-1}^{4/3}}
\le C_3\frac{B_m^{1/3}U_g(B_m^2)}{B_{j}^{4/3}}
\end{align*}
for all $j\le m$. Consequently, 
\begin{align*}
&\E\left[S\left(\nu_{m}\right)-g\left(\nu_{m}\right) ; T_{g}>\nu_{m}, S\left(\nu_{m}\right)-g\left(\nu_{m}\right)>3 B_{m}\right]\\
&\hspace{1cm}\le 2C_3B_m^{1/3}U_g(B_m^2)\sum_{j=1}^m B_j^{-4/3} \E\left[X_j;X_j>\frac{3}{2}B_m\right].
\end{align*}
Applying now the Markov inequality, we obtain 
\begin{align*}
&\E\left[S\left(\nu_{m}\right)-g\left(\nu_{m}\right) ; T_{g}>\nu_{m}, S\left(\nu_{m}\right)-g\left(\nu_{m}\right)>3 B_{m}\right]\\
&\hspace{1cm}\le \frac{4}{3}C_3B_m^{-2/3}U_g(B_m^2)
\sum_{j=1}^m B_j^{-4/3} \E\left[X^2_j;|X_j|>B_m\right].
\end{align*}
Set now 
$$
a_j:=\E\left[X^2_j;|X_j|>B_m\right]
\quad\text{and}\quad 
A_j=\sum_{k=1}^ja_k.
$$
Noting that $A_j\le B_j^2$ for all $j$, we get
\begin{align*}
\sum_{j=1}^m B_j^{-4/3} \E\left[X^2_j;|X_j|>B_m\right]
&\le \sum_{j=1}^m A_j^{-2/3} a_j=\sum_{j=1}^m \frac{A_j-A_{j-1}}{A_j^{2/3}}\\
&\le \int_0^{A_m}x^{-2/3}dx=3A_m^{1/3}.
\end{align*}
Consequently,
\begin{align*}
&\E\left[S\left(\nu_{m}\right)-g\left(\nu_{m}\right) ; T_{g}>\nu_{m}, S\left(\nu_{m}\right)-g\left(\nu_{m}\right)>3 B_{m}\right]\\
&\hspace{1cm}\le CB_m^{-2/3}A_m^{1/3} U_g(B_m^2).
\end{align*}
Noting that $B_m^{-2/3}A_m^{1/3}=L_m^{1/3}(1)$ and using \eqref{eq:Ug-slow}, we conclude that 
\begin{align}
\label{eq:part3}
\nonumber
&\E\left[S\left(\nu_{m}\right)-g\left(\nu_{m}\right) ; T_{g}>\nu_{m}, S\left(\nu_{m}\right)-g\left(\nu_{m}\right)>3 B_{m}\right]\\
&\hspace{1cm}=o\left(\E\left[S\left(n\right)-g\left(n\right) ; T_{g}>n\right]\right).
\end{align}
Plugging \eqref{eq:part1}, \eqref{eq:part2} and \eqref{eq:part3} into \eqref{eq:Tg-approx}, we obtain 
$$
B_n\P(T_g>n)\sim 2\varphi(0)U_g(B_n^2).
$$
Thus, the proof is complete.
\subsection{Conditional functional limit theorem and some properties of the function $U_g$}
Our approach to the tail behaviour of the stopping time $T_g$ was based on the comparison with the Brownian motion. The same scheme allows one to obtain a functional limit theorem for conditioned random walks.  To formulate this result we have to introduce the limit. Let $B=\{B(t),\,t\in[0,1]\}$ be the standard Brownian motion. It has been shown by Durrett, Iglehart and Miller~\cite{DIM1977} that, as $\varepsilon\to0$, the family of processes 
$$
\{B|\min_{s\le 1} B(s)>-\varepsilon\},\quad \varepsilon>0
$$
converges weakly on the space $C[0,1]$ of continuous functions on the time interval $[0,1]$. The limit is called {\it Brownian meander} and we shall
denote it by $M=\{M(t),\,t\in[0,1]\}$.
\begin{theorem}
\label{thm:fclt}
Under the conditions of Theorem~\ref{th:3.1}, the distribution of $s_n$ conditioned on $T_g>n$ converges weakly on $C[0,1]$ towards the Brownian meander
$M$. In particular,
$$
\P(S(n)>g(n)+xB_n|T_g>n)\to e^{-x^2/2},\quad x\ge0.
$$
\end{theorem}
The universality approach allows one to avoid the most standard approach to the proof of functional limit theorems, which consists in proving convergence of finite dimensional distributions and in showing the tightness. Instead, one can work directly with bounded continuous functionals on $C[0,1]$. The proof of Theorem~\ref{thm:fclt} can be found in \cite{denisov_sakhanenko_wachtel18}. 
Durrett~\cite{durrett1978} has used the classical method via convergence of finite dimensional distributions and tightness to prove conditional functional limit theorem for some classes of null recurrent Markov chains.

As we have mentioned before, the only difference between Theorem~\ref{th:3.1} and \eqref{eq:bm-tail} is the appearance of the slowly varying function $U_g$. We now show that it may happen that this function is not asymptotically constant even in the case when the boundary is constant.

If $g(n) \equiv -x$ for some $x\ge0$ then $S(\tau_x)\le0$. (Recall that in the case of constant boundaries we use the notation $\tau_x$ instead of $T_g$.) Applying the monotone convergence theorem, we conclude that 
$$
\E\left[-S\left(\tau_{x}\right) ; \tau_{x} \leq n\right] 
\rightarrow  \E\left[-S\left(\tau_{x}\right)\right]=:V_x\in(0, \infty] .
$$
Let us next derive a necessary condition for $V_x<\infty$. The monotonicity of 
$\E\left[-S\left(\tau_{x}\right) ; \tau_{x} \leq n\right]$ and Theorem~\ref{th:3.1} imply that
$$
\P\left(\tau_x>n\right) \ge \frac{C_{0}}{B_{n}},\quad  n \ge 1
$$
for some constant $C_{0}>0$.
\begin{lemma}\label{lem:3.13}
Let $x\ge 0$ be fixed and assume that the conditions of Theorem~\ref{th:3.1} hold.
For every $\varepsilon>0$ there exists $N_{\varepsilon}$ such that
\begin{align*}
\E\left[-x-S\left(\tau_x\right)\right] \ge & \frac{C_{0}}{4}\left(1-e^{-\varepsilon^{2} / 8}\right) \sum_{n=N_\varepsilon+1}^\infty \frac{1}{B_{n}} \E\left[-X_{n} ;-X_{n}>\varepsilon B_{n}\right] .
\end{align*}
\end{lemma}
\begin{proof}
We know from Theorem~\ref{thm:fclt} that, for every $\varepsilon>0$,
$$
\P\left(\left.\frac{x+S(n)}{B_{n}}<\frac{\varepsilon}{2} \,\right\rvert\, \tau_x>n\right) \rightarrow 1-e^{-\varepsilon^{2} / 8}.
$$
Thus, there exists $N_\varepsilon$ such that
$$
\P\left(\left.\frac{x+S(n)}{B_{n}}<\frac{\varepsilon}{2} \,\right\rvert\, \tau_x>n\right) \ge \frac{1}{2}\left(1-e^{-\varepsilon^{2} / 8}\right) \text { for all } n \geq N_{\varepsilon}. 
$$
For $\E\left[-x-S\left(\tau_x\right)\right]$ we have the following representation:
\begin{align*}
&\E\left[-x-S\left(\tau_x\right)\right]\\
&\hspace{1cm}=\sum_{n=1}^{\infty} \E\left[-x-S(n); \tau_x=n\right]\\
&\hspace{1cm}=\sum_{n=1}^{\infty} \E\left[-x-S(n-1)-X_{n} ; \tau_x>n-1, x+S(n-1)+X_{n}\le0\right]. 
\end{align*}
We next  notice that if $x+S(n-1) \le \frac{\varepsilon}{2} B_{n-1}$ and $X_{n}<-\varepsilon B_{n}$ then
$$
-x-S(n-1)-X_{n}
>-\frac{\varepsilon}{2} B_{n-1}+\frac{\varepsilon}{2}B_n-\frac{X_{n}}{2}
>-\frac{X_n}{2}.
$$
Consequently,
\begin{multline*}
\E\left[-x-S(n-1)-x_{n} ; \tau_x>n-1, x+S(n-1)+X_{n}<0\right] \\
\quad \ge \E\left[-\frac{X_{n}}{2} ;-X_{n}>\varepsilon B_{n}\right] \P\left(x+S(n-1)<\frac{\varepsilon}{2} B_{n-1}, \tau_x>n-1\right)
\end{multline*}
and
\begin{align*}
\E\left[-x-S\left(\tau_x\right)\right] 
&\ge \sum_{n=N_\varepsilon+1}^{\infty} \frac{\left(1-e^{-\varepsilon^{2} / 8}\right)}{4} \P\left(T_{0}>n-1\right) \E\left[-X_{n} ;-X_{n}>\varepsilon B_{n}\right]\\
&\ge \frac{C_0\left(1-e^{-\varepsilon^{2} / 8}\right)}{4}
\sum_{n=N_\varepsilon+1}^{\infty}\frac{1}{B_n}\E\left[-X_{n} ;-X_{n}>\varepsilon B_{n}\right].
\qedhere
\end{align*}
\end{proof}

Using this lemma we can construct an example with $\E[-S(\tau_x)]=\infty$ for every $x\ge0$.
Let $X_n$ have the following distribution:
$$
\P(X_n=\pm\sqrt{n})=\frac{p_n}{2}
\quad\text{and}\quad 
P(X_n=\pm a_n)=\frac{1-p_n}{2},
$$
where 
$$
p_n=\frac{1}{n\log(n+2)}
\quad\text{and}\quad 
a_n=\sqrt{\frac{1-np_n}{1-p_n}.}
$$
It is easy to see that $\E X_n=0$ and $\E X_n^2=1$ for every $n\ge1$.
Thus, $B_n^2=n$. Next, noting that $a_n\in(0,1)$ for all $n$, one infers that,
for $n>\varepsilon^{-2}$,
\begin{align*}
L_n(\varepsilon)
&=\frac{1}{n}\sum_{k=1}^n\E[X_k^2;|X_k|>\varepsilon n]\\
&=\frac{1}{n}\sum_{k\in(\varepsilon^2n,n]}kp_k
=\frac{1}{n}\sum_{k\in(\varepsilon^2n,n]}\frac{1}{\log(k+2)}
=O\left(\frac{1}{\log n}\right).
\end{align*}
So, the Lindeberg condition holds and, consequently, $\{X_n\}$ satisfies all the conditions of Theorem~\ref{th:3.1}. Then, for every $x>0$ there exists a slowly varying function $U_x$ such that
$$
\P(\tau_x>n)\sim \sqrt{\frac{2}{\pi}}\frac{U_x(n)}{n^{1/2}}
\quad\text{as }n\to\infty. 
$$
Notice next that, for every $N\ge1$,
\begin{align*}
\sum_{k=N+1}^\infty \frac{1}{B_k}\E\left[-X_k;X_k<-\frac{B_k}{2}\right]
&=\sum_{k=N+1}^\infty \frac{1}{\sqrt{k}}\E\left[-X_k;X_k<-\frac{\sqrt{k}}{2}\right]\\
&=\frac{1}{2}\sum_{k=N+1}^\infty \frac{1}{k\log(k+2)}=\infty.
\end{align*}
Applying now Lemma~\ref{lem:3.13} with $\varepsilon=\frac{1}{2}$, we infer that 
$\E[-S(\tau_x)]=\infty$ for every $x\ge0$.
Surprisingly, this is not possible in the case of i.i.d. summands. If all $X_{n}$ have the same distribution with $\E X_n=0$ and $\E X_n^2=\sigma^2\in(0,\infty)$ then
\begin{align*}
& \sum_{n=N+1}^{\infty} \frac{1}{B_{n}} \E\left[-X_{n};-X_{n}>\varepsilon B_{n}\right]\\
&\hspace{1cm}=\sum_{n=N+1}^{\infty} \frac{1}{\sqrt{n \sigma^{2}}} \E\left[-X_{1};-X_{1}>\varepsilon \sigma \sqrt{n}\right] \\
&\hspace{1cm} \le\frac{1}{\sqrt{\sigma^{2}}} \E\left[-X_{1} 
\sum_{n=1}^{\infty} \frac{1}{\sqrt{n}}
\mathbb{1}\left\{-X_{1}>\varepsilon \sigma \sqrt{n} \right\}\right] \le C \E\left[X_{1}^{2}\right]<\infty.
\end{align*}
So, Lemma~\ref{lem:3.13}  does not produce infinite lower bound. 
In the next section we shall show  that $\E\left[-S\left(\tau_{x}\right)\right]$ is finite for every $x$ provided that $\E\left[X_{1}\right]=0$ and $\E\left[X_{1}^{2}\right]=\sigma^{2}\in(0, \infty)$. 
\section{Conditioned random walks with i.i.d. increments.}
\label{sec:iid}
In this section we concentrate primerly on the case when the increments $\{X_k\}$ are independent and identically distributed.  If we additionally assume that 
\begin{equation}
\label{eq:1and2-moments}
\E X_1=0
\quad\text{and}\quad 
\E X_1^2:=\sigma^2\in(0,\infty)    
\end{equation}
then the Lindeberg condition is valid and we may apply Theorem~\ref{th:3.1}. Specialising this result to the case of i.i.d. increments and of constant boundaries, we obtain
\begin{align}
\label{eq:tau-special}
\P(\tau_x>n)\sim \sqrt{\frac{2}{\pi\sigma^2}}\frac{U(\sigma^2 n)}{\sqrt{n}}
\quad\text{as }n\to\infty,
\end{align}
the slowly varying function $U_x$ is given by 
$$
U_x(\sigma^2 n)=\E[S(n);\tau_x>n]=\E[-S(\tau_x);\tau_x\le n].
$$
It is easy to see that $U_x$ is monotone increasing and that 
$$
\lim_{n\to\infty} U_x(\sigma^2 n)=\E[-S(\tau_x)]\in(0,\infty].
$$
Our first purpose is to show that the function $V(x):=\E[-S(\tau_x)]$ is finite and to determine the asymptotic, as $x\to\infty$, behaviour of this function. We represent $V(x)$ as follows:
\begin{align}
\label{eq:def-f}
V(x)=\E\left[-S\left(\tau_{x}\right)\right]=x-\E\left[x+S\left(\tau_{x}\right)\right]
=: x+f(x). 
\end{align}
\begin{proposition}\label{lem:3.14}
Assume that \eqref{eq:1and2-moments} holds.
Then the function $fx)$ defined in \eqref{eq:def-f} is finite. Moreover,
$$
f(x)=o(x)\quad\text{as }x \rightarrow \infty.
$$
\end{proposition}
To prove this proposition we shall construct an appropriate positive supermartingale.
To formulate the corresponding result we introduce the following notations:
\begin{align*}
& a(x):=-\E\left[x+X_1; x+X_1 \leq 0\right)=\int_{x}^{\infty} \P(X_1 \leq-y) d y, \\
& \overline{a}(x)=\int_x^\infty \P(X_1>y)dy,\\
& b(x)=\int_{x}^{\infty} a(y) d y \quad \text { and } \quad m(x)=\int_{0}^{x} b(y) d y .
\end{align*}
Using integration by parts, we obtain
\begin{align*}
&\E[(x+X_1)^2;x+X_1<0]
=\int_{-\infty}^{-x}(y+x)^2d\P(X_1\le y)\\
&\hspace{1cm}=-2\int_{-\infty}^{-x}(y+x)\P(X_1\le y) dy
=2\int_x^\infty(y-x)\P(X_1\le-y)dy\\
&\hspace{1cm}=2\int_x^\infty\left(\int_x^ydz\right)\P(X_1\le-y)dy
=2\int_x^\infty\left(\int_z^\infty \P(X_1\le-y)dy\right)dz\\
&\hspace{1cm}=2\int_x^\infty a(z)dz.
\end{align*}
The assumption $\E X_1^2<\infty$ implies that the function $a(z)$ is integrable. Consequently,
$$
b(x)=\frac{1}{2}\E\left[(x+X_1)^{2} ; X_1<-x\right] \rightarrow 0
\quad\text{as } x\to\infty.
$$
This, in its turn, implies that 
$$
\frac{m(x)}{x} \rightarrow 0\quad \text{as }x\to\infty.
$$

\begin{lemma}\label{lem:3.15}
If \eqref{eq:1and2-moments} is valid then there exist positive constants $A$ and $R$
$$
W(x)=x+A m(x)+R
$$
is superharmonic for $S(n)$ killed at $\tau_{x}$. In other words,
\begin{equation}\label{eq:start}
\E\left[W(x+S(n)) ; \tau_{x}>n\right] \leq W(x) \text { for all } x \ge 0 
\text{ and all }n\ge1\text {. } 
\end{equation}
\end{lemma}
We postpone the proof of Lemma~\ref{lem:3.15} and show that \eqref{eq:start} yields the claim of  Proposition~\ref{lem:3.14}. Indeed, \eqref{eq:start} implies that 
\begin{align*}
W(x) \ge \E\left[W(x+S(n)) ; \tau_{x}>n\right] & \ge \E\left[x+S(n) ; \tau_{x}>n\right] \\
& =x-\E\left[x+S\left(\tau_{x}\right) ; \tau_{x} \le n\right] . 
\end{align*}
Letting here $n \rightarrow \infty$, 
we conclude that
$$
x+f(x) \le W(x)=x+A m(x)+R .
$$
In other words, $f(x) \le A m(x)+R<\infty$ and $f(x)=o(x)$ due to the fact that $m(x)=o(x)$.
\begin{proof}[Proof of Lemma~\ref{lem:3.15}]
We want to show that
$$
\Delta(x):=\E\left[W(x+X_1) ; \tau_{x}>1\right]-W(x) \le 0 
\quad\text{for all}\quad x \ge 0.
$$
Let $F(x)$ denote the distribution function of $X_1$ and set $\overline{F}(x):=1-F(x)$. 
Since $x=\E[x+X_1]$, we have
\begin{align}
\label{eq:delta1}
\nonumber  
\Delta(x)&=  \E[x+X_1+A m(x+X)+R ; X_1>-x]-x-A m(x)-R \\
\nonumber
&= -\E\left[x+X_1; X_1 \le-x\right]-R F(-x)-A m(x) F(-x) \\
\nonumber
&\hspace{1cm} +A \E[m(x+X_1)-m(x) ; X_1>-x] \\
&= a(x)-R F(-x)-A m(x) F(-x)+A \int_{-x}^{\infty}(m(x+y)-m(x)) d F(y) .
\end{align}
Recalling that $m(x)=\int_0^x b(z)dz$, we have 
\begin{align*}
\int_0^\infty(m(x+y)-m(x))dF(y)
&=\int_0^\infty\left(\int_0^yb(x+z)dz\right)dF(y)\\
&=\int_0^\infty b(x+z)\overline{F}(z)dz
\end{align*}
and 
\begin{align*}
\int_{-x}^0(m(x+y)-m(x))dF(y)
&=-\int_{-x}^0\left(\int_y^0b(x+z)dz\right)dF(y)\\
&=-\int_{-x}^0b(x+z)F(z)dz+F(-x)\int_{-x}^0b(x+z)dz\\
&=-\int_{-x}^0b(x+z)F(z)dz+F(-x)m(x).
\end{align*}
Thus, for the integral in \eqref{eq:delta1} we have
\begin{align*}
& \int_{-x}^{\infty}(m(x+y)-m(x)) d F(y)\\
&\hspace{1cm}=\int_{-x}^{0}(m(x+y)-m(x)) d F(y)
+\int_{0}^{\infty}(m(x+y)-m(x)) d F(y) \\
&\hspace{1cm} 
=m(x) F(-x)-\int_{-x}^{0} b(x+y) F(y) d y+\int_{0}^{\infty} b(x+y) \overline{F}(y) d y \\
&\hspace{1cm} = m(x)F(-x)+\int_{0}^{x} b(x-y) a^{\prime}(y) d y-\int_{0}^{\infty} b(x+y)\overline a^{\prime}(y) d y,
\end{align*}
where we have used the equalities 
\[
a'(x)=-F(-x), \quad \overline a'(x)=-\overline{F}(x). 
\]
Then, integrating by parts, we obtain
\begin{align*}
&\int_{-x}^{\infty}(m(x+y)-m(x)) d F(y)\\
&\hspace{1cm}=m(x) F(-x)+b(0) a(x)-b(x) a(0)-\int_{0}^{x} a(x-y) a(y) d y \\
&\hspace{2cm}+b(x) \bar{a}(0)-\int_{0}^{\infty} a(x+y) \bar{a}(y) d y .
\end{align*}
We next notice that
 \[
b(x) \bar{a}(0)-b(x) a(0)=b(x) \E[X]=0
\quad\text{and}\quad 
b(0)=\frac{1}{2}\E\left[\left(X_1^{-}\right)^{2}\right].
\]
Consequently,
\begin{align*}
&\int_{-x}^{\infty}(m(x+y)-m(x)) d F(y)\\
&=m(x) F(-x)+\frac{1}{2}\E[(X_1^-)^2] a(x)-\int_{0}^{x} a(x-y) a(y) d y 
-\int_{0}^{\infty} a(x+y) \bar{a}(y) d y .
\end{align*}
Plugging this into \eqref{eq:delta1} and noting that the integrals are nonnegative, we obtain

\begin{align*}
\Delta(x) & = a(x)-R F(-x)+\frac{A}{2}\E[(X_1^-)^2] a(x)\\
&\hspace{2cm}-A \int_{0}^{x} a(x-y) a(y) d y-A \int_{0}^{\infty} a(x+y) \bar{a}(y) d y \\
& \le a(x)-R F(-x)-A \int_{0}^{x} a(x-y) a(y) d y .
\end{align*}
Since $a(x)$ is decreasing,
$$
\int_{0}^{x} a(x-y) a(y) d y=2 \int_{0}^{x / 2} a(x-y) a(y) d y \ge 2 a(x)(b(0)-b(x / 2)).
$$
Using this bound and choosing $A=\frac{4}{E\left[(X_1^-)^{2}\right]}=\frac{2}{b(0)}$, we have
\begin{align*}
\Delta(x) & \le-R F(-x)+a(x)-2 A a(x)(b(0)-b(x / 2)) \\
& =-RF(-x)+a(x)(2 A b(x / 2)-2 A b(0)+1) \\
& =-RF(-x)+a(x)(2 A b(x / 2)-3).
\end{align*}
Since $b(x) \rightarrow 0$ there exists $x_{0}>0$ such that $2 A b(x_0 / 2)=3$. 
Therefore, $\Delta(x) \le 0$ for all $x \ge x_{0}$.\\
If $F\left(-x_{0}\right)>0$ then we can choose $R=a(0) / F\left(-x_{0}\right)$.
For this choice we have
\[
\Delta(x) \le-R F\left(-x_{0}\right)+a(0)=0
\] for all $x \le x_{0}$.

Finally suppose that $F\left(-x_{0}\right)=0$. This means that $\P\left(X_1>-x_{0}\right)=1$
and, consequently, $a\left(x_{0}\right)=0$. If we take $R=3 x_{0}$ then, applying the mean value theorem, we have for all $x<x_0$,
\begin{align*}
\Delta(x) \le a(x)-R F(-x) & =a\left(x_{0}\right)-\left(a\left(x_{0}\right)-a(x)\right)-R F(x) \\
& =\left(x_{0}-x\right) F(-\xi)-3x_0 F(-x)\\
&\le x_0F(-\xi)-3x_0F(-x)
\end{align*}
for some $\xi\in(x,x_0)$. Noting that $F(-\xi)\le F(-x)$, we complete the proof.
\end{proof}
Combining now \eqref{eq:tau-special} with Proposition~\ref{lem:3.14}, we conclude that
\begin{equation}
\label{eq:tau-V}
\P(\tau_x>n)\sim \sqrt{\frac{2}{\pi\sigma^2}}\frac{V(x)}{n^{1/2}}
\quad\text{as }n\to\infty,
\end{equation}
for every fixed $x$. So, the function $V(x)$ describes the dependence of the tail of $\tau_x$
on the starting position $x$. It turns out that the function $V$ has also a further, rather important, property.
\begin{lemma}\label{lem:3.16}
Assume that \eqref{eq:1and2-moments} holds.
Then the  function $V(x)=-\E\left[S\left(\tau_{x}\right)\right]$ is harmonic for $S(n)$ killed at $\tau_{x}$. That is,
$$
\E\left[V(x+S(1)) ; \tau_{x}>1\right]=V(x) \quad \text { for all } x>0.
$$
\end{lemma}
\begin{proof}
Since $\E[X]=0$,
\begin{align}
\label{eq:V1}
\nonumber
x & =\int_{\{y>-x\}}(x+y) \P(X_1 \in d y)+\int_{\{y \leq-x\}}(x+y) \P(X_1 \in d y) \\
& =\E\left[x+S(1) ; \tau_{x}>1\right]+\int_{\{y \leq-x\}}(x+y) \P(X_1 \in d y) .
\end{align}
Recall the definition of the function $f(x)$ in \eqref{eq:def-f}.
By the Markov property,
\begin{align}
\label{eq:V2}
\nonumber
-f(x)&=\E\left[x+S\left(\tau_{x}\right)\right]\\
\nonumber
&=\int_{\{y>-x\}} \E\left[x+y+S\left(\tau_{x+y}\right)\right] \P(X_1 \in d y)
+\int_{\{y \leq-x\}}(x+y) \P(X_1 \in d y)\\
&=\E\left[-f(x+S(1)); \tau_{x}>1\right]+\int_{\{y \leq-x\}}(x+y) \P(x \in d y).
\end{align}
Taking the difference of \eqref{eq:V1} and \eqref{eq:V2}, we have
\begin{align*}
V(x)=x+f(x) & =\E\left[x+S(1)+f(x+S(1)) ; \tau_{x}>1\right] \\
& =\E\left[V(x+X) ; \tau_{x}>1\right].
\end{align*}
Thus, the lemma is proved.
\end{proof}

We next relate the constructed above harmonic function to weak descending ladder epochs
$\{\chi^-_k\}$ which have been introduced in the section on the Wiener-Hopf factorisation.
\begin{proposition}
\label{prop:V-H}
Let $S(n)$ be a random walk with independent, identically distributed increments. 
If the function $\P(\tau_0>n)$ is regularly varying with index $-\varrho\in(-1,0)$
then, for every $x>0$,
$$
\lim_{n\to\infty}\frac{\P(\tau_x>n)}{\P(\tau_0>n)}=H(x),
$$
where $H(x)$ is the renewal function of the weak descending ladder heights
$\{\chi^-_k\}$, that is,
$$
H(x)=1+\sum_{m=1}^\infty\P\left(\chi_1^-+\chi^-_2+\ldots+\chi^-_m<x\right),\quad x\ge0.
$$
\end{proposition}
\begin{proof}
Decomposing the the  trajectory of the random walk according to ladder epochs, one obtains easily the representation
\begin{equation}
\label{eq:VH1}
\tau_x=\tau_1^-+\tau_2^-+\ldots+\tau^-_{\eta(x)},
\end{equation}
where 
$$
\eta(x):=\inf\{k\ge1:\chi^-_1+\chi^-_2+\ldots+\chi^-_k\ge x\}.
$$
Since $\{(\tau^-_k,\chi^-_k)\}$ are independent, the $\sigma$-algebras 
$\sigma(\tau^-_1,\tau^-_2,\ldots,\tau^-_m,{\rm 1}\{\eta_x\le m\})$ and 
$\sigma(\tau^-_{m+1},\tau^-_{m+2},\ldots)$ are independent for every $m$. This allows us to apply Theorem 1 from Korshunov~\cite{Korshunov2009} which implies that
\begin{equation}
\label{eq:VH2} 
\E\min\{\tau_x,n\}\sim \E\eta(x)\E\min\{\tau_0,n\}\quad\text{as }n\to\infty.
\end{equation}
Noting that 
\begin{align*}
\E\min\{Y,n\}
&=\sum_{k=1}^n k\P(Y=k)+n\P(Y>n)\\
&=\sum_{k=1}^n k\left(\P(Y>k-1)-\P(Y>k)\right)+n\P(Y>n)\\
&=\sum_{k=0}^{n-1}(k+1)\P(Y>k)-\sum_{k=1}^nk\P(Y>k)+n\P(Y>n)\\
&=\sum_{k=0}^{n-1}\P(Y>k)
\end{align*}
for every integer-valued random variable $Y$, we can rewrite \eqref{eq:VH2} as follows:
\begin{equation}
\label{eq:VH4}
\sum_{k=0}^{n-1}\P(\tau_x>k)\sim \E\eta(x) \sum_{k=0}^{n-1}\P(\tau_0>k)
\quad\text{as }n\to\infty.
\end{equation}
The assumption that $\P(\tau_0>n)$ is regularly varying with index $-\varrho$ implies that 
\begin{align}
\label{eq:tau-sum-1}
\sum_{k=0}^{n-1}\P(\tau_0>k)\sim\frac{1}{1-\varrho}n\P(\tau_0>n)
\quad\text{as }n\to\infty.
\end{align}
Combining this with \eqref{eq:VH4}, we infer that 
$$
\sum_{k\in[an,n)}\P(\tau_x>n)
\sim \E\eta(x)\frac{1-a^{1-\varrho}}{1-\varrho}n\P(\tau_0>n)
$$
for every $a\in(0,1)$.
Using now the fact that the sequence $\P(\tau_x>k)$ decreases, we get the bounds 
$$
\frac{\P(\tau_x>n)}{\P(\tau_0>n)}\le \E\eta(x)\frac{1-a^{1-\varrho}}{(1-a)(1-\varrho)}+o(1)
$$
and 
$$
\frac{\P(\tau_x>an)}{\P(\tau_0>n)}\ge \E\eta(x)\frac{1-a^{1-\varrho}}{(1-a)(1-\varrho)}+o(1).
$$
Consequently,
$$
\limsup_{n\to\infty}\frac{\P(\tau_x>n)}{\P(\tau_0>n)}
\le \E\eta(x)\frac{1-a^{1-\varrho}}{(1-a)(1-\varrho)}
$$
and 
$$
\liminf_{n\to\infty}\frac{\P(\tau_x>n)}{\P(\tau_0>n)}
\ge \E\eta(x)\frac{1-a^{1-\varrho}}{(1-a)(1-\varrho)}a^\varrho.
$$
Letting here $a\to1$, we conclude that 
$$
\lim_{n\to\infty}\frac{\P(\tau_x>n)}{\P(\tau_0>n)}=\E\eta(x).
$$
It remains to notice that
\begin{align*}
\E\eta(x)
&=\sum_{k=0}^\infty\P(\eta(x)>k)\\
&=1+\sum_{k=1}^\infty\P(\chi^-_1+\chi^-_2+\ldots+\chi^-_k<x)=H(x).
\qedhere
\end{align*}
\end{proof}
If the moment conditions in \eqref{eq:1and2-moments} are valid then, as we have already shown that
$$
\P(\tau_x>n)\sim V(x)n^{-1/2},
$$
where $V(x)=\E[-S(\tau_x)]$. Combining this with Proposition~\ref{prop:V-H}, we conclude that 
\begin{equation}
\label{eq:V_and_H}
V(x)=V(0)H(x)=\sqrt{\frac{2}{\pi\sigma^2}}\E[\chi^-_1]H(x),
\quad x\ge0.
\end{equation}
Using now Lemma~\ref{lem:3.16}, we infer that the function $H(x)$ is also harmonic for $S(n)$ killed at leaving $[0,\infty)$:
\begin{equation}
\label{eq:H-harm}
H(x)=\E[H(x+S(1));\tau_x>1],\quad x>0.
\end{equation}
Our next purpose is to show that \eqref{eq:H-harm} remains valid without moment assumptions in \eqref{eq:1and2-moments}.
\begin{proposition}
\label{prop:H-harmonic}
Assume that the random walk $S(n)$ is oscillating:
$$
\liminf_{n\to\infty} S(n)=-\infty 
\quad\text{and}\quad 
\limsup_{n\to\infty} S(n)=\infty
\quad\text{a.s.}
$$
Then the function $H(x)$ is harmonic for $S(n)$ killed at leaving $[0,\infty)$, that is, \eqref{eq:H-harm} holds.
\end{proposition}
\begin{proof}
The assumption $\liminf_{n\to\infty} S(n)=-\infty$ implies that all ladder heights $\chi^-_k$
are weel-defined. Then, by the definition of $H$,
\begin{align*}
H(x)
&=1+\E\left[\sum_{k=1}^\infty{\rm 1}\left\{\chi^-_1+\chi^-_2+\ldots+\chi^-_k<x\right\}\right]\\
&=1+\E\left[\sum_{k=1}^\infty{\rm 1}
\left\{\chi^-_1+\chi^-_2+\ldots+\chi^-_k<x\right\};S(1)>0\right]\\
&\hspace{1cm}+\E\left[\sum_{k=1}^\infty{\rm 1}
\left\{\chi^-_1+\chi^-_2+\ldots+\chi^-_k<x\right\};S(1)\in(-x,0]\right]\\
&\hspace{1cm}+\E\left[\sum_{k=1}^\infty{\rm 1}
\left\{\chi^-_1+\chi^-_2+\ldots+\chi^-_k<x\right\};S(1)\le -x \right].
\end{align*}
The last term is equal zero because $\chi^-_1\ge x$ on the event $\{S(1)\le-x\}$.
Furthermore, using the Markov property, we obtain 
\begin{align*}
&\E\left[\sum_{k=1}^\infty{\rm 1}
\left\{\chi^-_1+\chi^-_2+\ldots+\chi^-_k<x\right\};S(1)>0\right]\\
&\hspace{1cm}=\E\left[H(x+S(1))-H(S(1));S(1)>0\right]
\end{align*}
and
\begin{align*}
&\E\left[\sum_{k=1}^\infty{\rm 1}
\left\{\chi^-_1+\chi^-_2+\ldots+\chi^-_k<x\right\};S(1)\in(-x,0]\right]\\
&\hspace{1cm}=\E\left[H(x+S(1));S(1)\in(-x,0]\right].
\end{align*}
Consequently,
$$
H(x)=\E[H(x+S(1));x+S(1)>0]+1-\E[H(S(1));S(1)>0].
$$
Thus, it remains to show that 
\begin{align}
\label{eq:H_x_0}
\E[H(S(1));S(1)>0]=1.
\end{align}
By the total probability law,
\begin{align*}
H(x)
&=1+\sum_{k=1}^\infty\P\left(\chi^-_1+\chi^-_2+\ldots+\chi^-_k<x\right)\\
&=1+\sum_{k=1}^\infty\sum_{n=1}^\infty\P(U_k^-=n, S(n)\in(-x,0])\\
&=1+\sum_{n=1}^\infty\P(S(n)\le S(j)\text{ for all }j\le n, S(n)\in(-x,0]).
\end{align*}
Applying now the duality lemma, we obtain
\begin{align*}
H(x)
=1+\sum_{n=1}^\infty \P(S(j)\le0\text{ for all }j\le n, S(n)\in(-x,0]).
\end{align*}
Therefore,
\begin{align}
\label{eq:H-S>0}
\nonumber
&\E[H(S(1));S(1)>0]\\
\nonumber
&=\int_0^\infty H(y)\P(S(1)\in dy)\\
\nonumber
&=
\P(S(1)>0)+\sum_{n=1}^\infty\int_0^\infty \P(S(j)\le0\text{ for all }j\le n, S(n)\in(-y,0])\P(S(1)\in dy)\\
\nonumber
&=
\P(S(1)>0)+\sum_{n=1}^\infty \P(S(j)\le0\text{ for all }j\le n, S(n+1)>0)\\
&=\P(\tau^+=1)+\sum_{n=1}^\infty\P(\tau^+=n+1)
=\P(\tau^+<\infty).
\end{align}
Noting now that the assumptions $\limsup_{n\to\infty} S(n)=\infty$ implies that 
$\P(\tau^+<\infty)=1$, we complete the proof of \eqref{eq:H_x_0}.
\end{proof}
\begin{remark}
Formula \eqref{eq:H-S>0} implies that 
$$
\E[H(S(1));S(1)>0]<1
$$
provided that $\lim_{n\to\infty}S(n)=-\infty$ with probability one. Consequently,
$$
\E[H(x+S(n));x+S(1)>0]<H(x),\quad x>0
$$
for any random walk which tends to $-\infty$.
\hfill$\diamond$
\end{remark}
Under the conditions of Proposition~\ref{prop:V-H} onas also the following uniform upper bound.
\begin{lemma}
\label{lem:tau-upper}
Assume that the conditions of Proposition~\ref{prop:V-H} hold. Then there exist a constant $C$ such that, uniformly in $x>0$,
$$
\frac{\P(\tau_x>n)}{\P(\tau_0>n)}\le C H(x),\quad n\ge0.
$$
\end{lemma}
\begin{proof}
Combining the representation \eqref{eq:VH1} we the Markov inequality, we obtain 
\begin{align*}
\P(\tau_x\ge n)
&=\P\left(\sum_{k=1}^{\sigma(x)}\tau^-_k\ge n\right)
=\P\left(\sum_{k=1}^{\sigma(x)}(\tau^-_k\wedge n)\ge n\right)\\
&\le\frac{1}{n}\E\left[\sum_{k=1}^{\sigma(x)}(\tau^-_k\wedge n)\right].
\end{align*}
Applying now the Wald identity and recalling that $\E[\sigma(x)]=H(x)$, we conclude that 
\begin{equation}
\label{eq:tau-wald}
\P(\tau_x\ge n)\le\frac{H(x)}{n}\E\left[\tau_0\wedge n\right].
\end{equation}
Due to \eqref{eq:tau-sum-1},
\begin{align*}
\E[\tau_0\wedge n]=\sum_{k=0}^{n-1}\P(\tau_0>k)
\sim\frac{1}{1-\varrho}n\P(\tau_0>n)
\quad\text{as }n\to\infty.
\end{align*}
Combining this with \eqref{eq:tau-wald}, we infer that there exists a constant $C$ such that
$$
\P(\tau_x\ge n)\le C H(x) \P(\tau_0\ge n)
$$
for all $x>0$ and all $n\ge 1$. Thus, the proof is complete.
\end{proof}
The existence of increasing harmonic function for a walk killed at $\tau_x$
allows one to obtain an alternative,  uniform in starting point $x$, upper bound for the tail
of the stopping time $\tau_x$.
\begin{lemma}\label{lem:H-V_upper}
For every oscillating random walk one has
\begin{align}
\label{eq:H_upper_bound}
\P(\tau_x>n)\le\frac{H(x)}{\E[H(S(n));S(n)>0]},\quad n\ge1
\end{align}
for all $x\ge0$.
If \eqref{eq:1and2-moments} is valid then there exists a constant $C$ such that
\begin{align}
\label{eq:V_upper_bound}
\P(\tau_x>n)\le C\frac{V(x)}{\sqrt{n}},\quad n\ge1
\end{align}
for all $x\ge0$.
\end{lemma}
\begin{proof}
Applying Lemma~\ref{lem:upper} to the stopping time $\tau_x$, we have 
\begin{align*}
\P(x+S(n)>y|\tau_x>n)\ge\P(S(n)>y)\quad\text{for all }x\in\R.
\end{align*}
Combining this with the fact that $H$ increases, we conclude that 
\begin{align*}
\E[H(x+S(n))|\tau_x>n]
&\ge \E[H(x+S(n));x+S(n)>0]\\
&\ge \E[H(S(n));S(n)>0].
\end{align*}
This is equivalent to \eqref{eq:H_upper_bound}.

If \eqref{eq:1and2-moments}, then we can combine \eqref{eq:V_and_H} and \eqref{eq:H_upper_bound} to conclude that 
$$
\P(\tau_x>n)\le\frac{V(x)}{\E[V(S(n));S(n)>0]},\quad n\ge1.
$$
Thus, it remains to bound the denominator from below. Recall that $V(x)\sim x$
as $x\to\infty$. Combining this with the central limit theorem, we obtain
\begin{align*}
\liminf_{n\to\infty} \frac{\E[V(S(n));S(n)>0]}{\sqrt{n}}
&\ge \liminf_{n\to\infty} \frac{V(\sqrt{n})\P(S(n)>\sqrt{n})}{\sqrt{n}}\\
&=\overline{\Phi}\left(\frac{1}{\sigma}\right)>0.
\end{align*}
This completes the proof of the lemma.
\end{proof}

Existence of a positive harmonic function $V$ for $S(n)$ killed at leaving $(0,\infty)$ allows one to perform the Doob $h$-transform and to define a random walk conditioned to stay positive at {\it all} times. This is a Markov process which is given by the transition kernel 
$$
\widehat{P}(x,dy)=\frac{V(y)}{V(x)}\P(x+S(1)\in dy,\tau_x>1),
\quad x,y>0.
$$
Let $\widehat{\P}$ denote the corresponding probability measure. The connection between this new measure and the conditioning on $\tau_x>n$ is given by the following lemma.
\begin{lemma}
\label{lem:doob-transform}
Assume that \eqref{eq:1and2-moments} holds. Then, for every fixed $k\ge1$ and every $A\in\sigma(S(1),S(2),\ldots,S(k))$,
$$
\widehat{\P}_x(A)=\lim_{n\to\infty}\P(A|\tau_x>n).
$$
\end{lemma}
\begin{proof}
By the Markov property at time $k$,
$$
\P(A,\tau_x>n)=\int_0^\infty\P(x+S(k)\in dz,A\cap\{\tau_x>k\})\P(\tau_z>n-k).
$$
Consequently,
$$
\P(A|\tau_x>n)=\int_0^\infty\P(x+S(k)\in dz,A\cap\{\tau_x>k\})
\frac{\P(\tau_z>n-k)}{\P(\tau_x>n)}.
$$
Using \eqref{eq:tau-V}, we conclude that
$$
\lim_{n\to\infty}\frac{\P(\tau_z>n-k)}{\P(\tau_x>n)}=\frac{V(z)}{V(x)}
$$
for every $z>0$. Furthermore, combining the finiteness of $\E X_1^2$ and \eqref{eq:V_upper_bound}, we infer that the Lebesgue theorem on dominated convergence applies. As a result,
$$
\lim_{n\to\infty}\P(A|\tau_x>n)
=\int_0^\infty\P(x+S(k)\in dz,A\cap\{\tau_x>k\})\frac{V(z)}{V(x)}.
$$
By the definition of $\widehat{\P}$, the integral on the right hand side equals
$\widehat{\P}_x(A)$. Thus, the proof is complete.
\end{proof}

Bertoin and Doney~\cite{bertoin_doney} have shown that the measure $\widehat{\P}$
 can also be obtained by conditioning the walk on the event
 $$
 D_{N,x}=\{\text{the sequence }\{x+S(n)\}\text{ hits }(N,\infty)\text{ before }
 (-\infty,0]\}.
 $$
 Mo precisely, they have proven that
 $$
 \lim_{N\to\infty}\P(A|D_{N,x})=\widehat{\P}_x(A)
 $$
 for every $A\in\sigma(S(1),S(2),\ldots,S(k))$ with some fixed $k$.

One can derive limit theorems for the walk $S(n)$ under the measure $\widehat{\P}$
from Theorem~\ref{thm:fclt} specialised to the case of i.i.d. increments. In Section~\ref{sec:markov} we shall formulate and prove a more general result, which is valid for Doob $h$ transforms of Markov chains.
\section{Local Limit Theorem for Conditional Distributions}
\label{sec:local}
Specializiing Theorem~\ref{thm:fclt}
to the case of i.i.d. increments, we obtain
\begin{equation}\label{eq:1}
\P\left(\left.\frac{x+S(n)}{\sigma\sqrt{n}}>y \, \right\rvert\, \tau_{x}>n\right) \rightarrow e^{-y^{2} / 2}\quad \text{ for every }y>0.
\end{equation}
Then it is rather natural to expect that the local probabilities behave also regularly and that one has
\begin{align*}
\P\left(x+S(n)\in(y,y+1]\,|\tau_x>n\right) 
&=\P\left(\frac{x+S(n)}{\sigma\sqrt{n}}
\in\left(\frac{y}{\sigma\sqrt{n}},\frac{y+1}{\sigma\sqrt{n}}\right]\Big|\tau_x>n\right)\\
&\approx \exp\left\{-\frac{y^2}{2\sigma^2n}\right\}-
\exp\left\{-\frac{(y+1)^2}{2\sigma^2n}\right\}\\
&\approx \frac{y}{\sigma^2n}\exp\left\{-\frac{y^2}{2\sigma^2n}\right\}.
\end{align*}
In this section we derive this relation rigorously for lattice random walks.
\begin{theorem}\label{th:4.1}
Assume that $X_1$ takes values on $\Z$, has period $d$ 
and $\E X_1=0$, $\E X_1^2=\sigma^2\in(0,\infty)$. Then, for every fixed $x\ge0$,
$$
\sup _{y:\,y-x\in D_n}\left|\sqrt{n}  \P\left(x+S(n)=y \mid \tau_{x}>n\right)-d\frac{y}{\sigma^2\sqrt{n}} e^{-y^{2} / 2\sigma^2 n}\right| \rightarrow 0,\quad n\to\infty. 
$$
Furthermore,
$$
\P\left(x+S(n)=y \mid \tau_{x}=0\right)\quad
\text{for all }y\text{ such that }y-x\notin D_n.
$$
\end{theorem}

This local limit theorem can be proved by using recursive formulae
for conditional local probabilities, which follow from the Wiener-Hopf factorisation identities, see \cite{VW09}. The proof we present below, is a simplified version of the proof in \cite{denisov_wachtel15}, where multidimensional walks in cones have been considered. Since factorisation techniques dot not apply in higher dimensions, the proof strategy below is more robust than the strategy in \cite{VW09}.

We start by deriving some upper bounds for local probabilities.
\begin{lemma}\label{lem:4.2}
There exist constants $C_1$ and $C_2$ such that
\begin{align*}
\P\left(x+S(n)=y ; \tau_{x}>n\right) &\le C_1 n^{-1 / 2} \P\left(\tau_{x}>\frac{n}{2}\right)\\
&\le C_2\frac{V(x)}{n}.
\end{align*}
\end{lemma}
\begin{proof}
It is immediate from \eqref{eq:lclt-1} and \eqref{eq:lclt-2} that there exists a constant $C_0$ such that 
\begin{equation}\label{eq:3}
P(S(j)=z) \le \frac{C_{0}}{\sqrt{j}} \text { uniformly in } z \in \mathbb{Z} \text {. } \end{equation}
Therefore, applyng the Markov property at time $m=\lfloor\frac{n}{2}\rfloor$, we obtain
\begin{align*}
&\P\left(x+S(n)=y, \tau_{x}>n\right)\\
&\hspace{1cm} =\sum_{z>0} \P\left(x+S(m)=z, \tau_{x}>m\right) \P(z+S(n-m)=y,\tau_z>n-m) \\
&\hspace{1cm} \leq \sum_{z>0} \P\left(x+S(m)=z, \tau_{x}>m\right) \P(S(n-m)=y-z) \\
&\hspace{1cm} \le \frac{C_{0}}{\sqrt{n-m}} \P\left(\tau_{x}>m\right) .
\end{align*}
Thus, the first inequality is proved. The second one follows from \eqref{eq:V_upper_bound}.
\end{proof}
\begin{lemma}\label{lem:4.3}
For all $x, y>0$ we have
$$
\P\left(x+S(n)=y ; \tau_{x}>n\right) \leq \frac{C_{3}(x+1)(y+1)}{n^{3 / 2}}
$$
\end{lemma}
\begin{proof}
By the Markov property at time $m=\lfloor\frac{n}{2}\rfloor$, 
\begin{multline*}
\P\left(x+S(n)=y, \tau_{x}>n\right)\\ 
=\sum_{z>0} \P\left(x+S(m)=z, \tau_{x}>n\right) 
\P\left(z+S(n-m)=y, \tau_{z}>n-m\right).
\end{multline*}
We now reverse the time in the second probability
$$
\P\left(z+S(n-m)=y, \tau_{z}>n-m\right)=\P\left(y-S(n-m)=z, \tau_{y}^{\prime}>n-m\right),
$$
where $\tau_{y}^{\prime}=\inf \{k \ge 1: y-S(k) \le 0\}$ and
$V'(y)=\E[S(\tau'_y)]$.

Applying now Lemma~\ref{lem:4.2} to the random walk 
$\{-S(n)\}$ we get
$$
\P\left(x+S(n)=y, \tau_{x}>n\right) 
\le \frac{C_1}{n} V^{\prime}(y) \P\left(\tau_{x}>n / 2\right) 
\le C_2 \frac{V^{\prime}(y) V(x)}{n^{3 / 2}}.
$$
It remains to recall that $V(x)\sim x$ and $V^{\prime}(y)\sim y$ as $x,y\to\infty$  and to notice that this implies the bound 
$V(x)V'(y)\le C(x+1)(y+1)$.
\end{proof}
\begin{lemma}\label{lem:4.4}
There exist constants $C,a>0$ such that  
\[
\limsup _{n \rightarrow \infty} \sup_{|y|>u\sqrt n} \sqrt{n} \P(S(n)=y) \le C e^{-a u^{2}}
\]
and
$$
\limsup_{n \rightarrow \infty} \sup _{x, z \in M_{n, u}} \sqrt{n} \P\left(x+S(n)=z, \tau_{x} \le n\right) \le C e^{-a u^{2}}
$$
where $M_{n, u}:=\{z\colon z \ge u \sqrt{n}\}$.
\end{lemma}
\begin{proof}
Set again $m=\lfloor\frac{n}{2}\rfloor$. For every $y$ with $|y|\ge u\sqrt{n}$ we have 
\begin{align*}
&\P(S(n)=y)\\
&\hspace{5mm}=\P(S(n)=y,|S(m)| \ge u \sqrt{n} / 2)
 +\P(S(n)=y,|S(m)|<m \sqrt{n} / 2) \\
&\hspace{5mm} =\P(S(n)=y,|S(m)| \ge u \sqrt{n} / 2) +\P\left(S(n)=y, |S(n)-S(m)| \ge u \sqrt{n} / 2\right).
\end{align*}
Using the Markov property and taking into account 
\eqref{eq:3}, we obtain
\begin{align*}
&\P(S(n)=y,|S(n)| \ge u \sqrt{n} / 2)\\
&\hspace{1cm}=\sum_{z:\,|z|>u\sqrt{n}/2}\P(S(m)=z)\P(S(n-m)=y-z)\\
&\hspace{1cm}
\le \frac{C}{\sqrt{n-m}}\sum_{z:\,|z|>u\sqrt{n}/2}\P(S(m)=z)\\
&\hspace{1cm}=\frac{C}{\sqrt{n-m}}\P(|S(m)| \ge u \sqrt{n} / 2)
\end{align*}
and
\begin{align*}
&\P(S(n)=y,|S(n)-S(m)| \ge u \sqrt{n} / 2)\\
&\hspace{1cm}=\sum_{z:\,|z-y|>u\sqrt{n}/2}\P(S(m)=z)\P(S(n-m)=y-z)\\
&\hspace{1cm}\le\frac{C}{\sqrt{m}}\sum_{z:\,|z-y|>u\sqrt{n}/2}\P(S(n-m)=y-z)\\
&\hspace{1cm}= \frac{C}{\sqrt{m}} \P(|S(n-m)| \ge u \sqrt{n} / 2) .
\end{align*}
Applying now the central limit theorem, we obtain the first estimate.

To prove the second estimate, we notice that if $z \ge u \sqrt{n}$ then
\begin{align*}
&\P\left(x+S(n)=z, \tau_{x} \le n / 2\right) \\
&\hspace{1cm}=\sum_{k=1}^{n/2}\sum_{y=-\infty}^0
\P(x+S(n)=y,\tau_x=k)\P(y+S(n-k)=z)\\
&\hspace{1cm}\le \max _{\frac{n}{2} \le k \le n} \sup _{|y| \ge u \sqrt{n}} \P(S(k)=y).
\end{align*}
Furthermore, $x \geq u \sqrt{n}$ then, inverting the time, we obtain
\begin{align*}
P\left(x+S(n)=z, \frac{n}{2}<\tau_{x} \le n\right) 
&\le \P\left(z-S(n)=x, \tau'_{z} \le n/2\right) \\
&\leq \max _{\frac{n}{2} \leq k \le n } \sup_{|y|\ge u\sqrt n} \P(S(k)=y).
\end{align*}
Using now the first claim, we finish the proof.
\end{proof}
\begin{proof}[Proof of Theorem~\ref{th:4.1}]
Assume first that $y$ is such that $y-x\notin D_n$. Then, by \eqref{eq:lclt-2},
$$
\P(x+S(n)=y,\tau_x>n)\le \P(S(n)=x-y)=0.
$$
Thus, it remains to consider values $y$ such that $y-x\in D_n$.
 Fix $\varepsilon>0$ and $A>0$.\\
Since $z e^{-z^{2} / 2}$ goes to zero as $z \rightarrow 0$ and $z \rightarrow \infty$, it suffices to show that
\begin{align}
\label{eq:A}
\lim _{\varepsilon \rightarrow 0} \limsup _{n \rightarrow \infty}\  
n\max_{y \leq \varepsilon\sqrt{n}} \P\left(x+S(n)=y, \tau_{x}>n\right)=0,
\end{align}
\begin{align}
\label{eq:B}
\lim _{A \rightarrow \infty} \limsup_{n \rightarrow \infty}\ 
n \max _{y \ge A \sqrt{n}} \P\left(x+S(n)=y, \tau_{x}>n\right)=0
\end{align}
and
\begin{align}\label{eq:C}
\nonumber
\lim _{\varepsilon \rightarrow 0} \limsup _{n\rightarrow \infty} \sup _{y \in\left[\varepsilon\sqrt n, A\sqrt n\right]\cap(x+D_n)} 
&\bigg| 
n \P\left(x+S(n)=y, \tau_{x}>n\right)\\
&\hspace{5mm}-\left(\frac{y}{\sqrt{n}}\right) e^{-y^{2} / 2}
\sqrt{n}\P\left(\tau_{x}>n\right)\bigg|=0.
\end{align}
Equation~\eqref{eq:A} is immediate from Lemma~\eqref{lem:4.3}. 

To show~\eqref{eq:B} we set $m=\lfloor\frac{n}{2}\rfloor$ and notice that
\begin{align*}
\P\left(x+S(n)=y, \tau_{x}>n\right)&=\P\left(x+S(n)=y, \tau_{x}>u,|S(m)| \leq \frac{A \sqrt{n}}{2}\right) \\
&\phantom{XX} +\P\left(x+S(n)=y, \tau_{x}>n, |S(m)|>\frac{A \sqrt{n}}{2}\right). 
\end{align*}
Using Markov property at time $m$ and using first \eqref{eq:3} and then Theorem~\ref{eq:1}, we obtain
\begin{align*}
& \P\left(x+S(n)=y, \tau_{x}>n,|S(m)|>\frac{A \sqrt{n}}{2}\right)\\
&\hspace{1cm}\le \sum_{z>0}\P(x+S(m)=z,\tau_x>m)\P(z+S(n-m)=y)\\
&\hspace{1cm}
\le \frac{C_1}{\sqrt{n}} P\left(|S(m)|>\frac{A \sqrt{n}}{2}, \tau_{x}>m\right) \\
&\hspace{1cm} \le \frac{C_2}{n} 
P\left(|S(m)|>\frac{A \sqrt{n}}{2}\Big|\, \tau_{x}>m\right) 
\le \frac{C_3}{n} e^{-A^{2} / 2}.
\end{align*}
Applying Lemma~\ref{lem:4.4} we obtain, 
\begin{align*}
\P\left(x+S(n)=y, \tau_{x}>n,|S(m)| \le A \sqrt{n}\right) &\le \P\left(\tau_{x}>m\right) \sup _{z>A \sqrt{n} / 2} \P(S(n-m)=z)\\ 
&\le C \frac{V(x)}{n} e^{-a A^{2}}. 
\end{align*}
Combining these bounds and letting first $n \rightarrow \infty$ and then $A \rightarrow \infty$, we obtain \eqref{eq:B}.

We now turn to the central part: 
$$
y \in[\varepsilon \sqrt{n}, A \sqrt{n}]\cap(x+D_n).
$$
For this range of $y$ we set $m=\left[\varepsilon^{3} n\right]$ and write the following representation:
\begin{align*}
&\P\left(x+S(n)=y, \tau_{x}>n\right)\\
&\hspace{1cm}=\sum_{z>0} \P\left(x+S(n-m)=z, \tau_{x}>n-m\right) \P\left(z+S(m)=y, \tau_{z}>m\right).
\end{align*}
Set $R_{1}=\left\{z:|z-y|<\frac{\varepsilon}{2} \sqrt{n}\right\}$. We have
\begin{align*}
\sum_{z \notin R_{1}} & \P\left(x+S(n-m)=z, \tau_{x}>n-m\right) 
\P\left(z+S(n)=y, \tau_{z}>m\right) \\
& \le \sum_{z \notin R_{1}} \P\left(x+S(n-m)=z, \tau_{x}>n-m\right) \P(S(m)=y-z) \\
& \le \P\left(\tau_{x}>n-m\right) \max _{|w|>\varepsilon / 2 \sqrt{n}} P(S(m)=w).
\end{align*}
Since $m \sim \varepsilon^{3} n, \quad \frac{\varepsilon}{2} \sqrt{n} \sim \frac{\varepsilon}{2} \sqrt{\frac{m}{\varepsilon^{3}}}=\frac{1}{2 \varepsilon^{1 / 2}} \sqrt{m}$, 
using Lemma~\ref{lem:4.4}, we conclude that
\begin{multline}
    \label{eq:4}
 \sum_{z \notin R_{1}} \P\left(x+S(n-m)=z, \tau_{x}>n-m\right) \P\left(z+S(m)=y, \tau_{z}>m\right) \\
 \le C \frac{V(x)}{\sqrt{n}} \cdot \frac{C}{\sqrt{\varepsilon^{3} n}} e^{-a / \varepsilon}=C \frac{V(x)}{n} \varepsilon^{-3 / 2} e^{-a / \varepsilon}.  
\end{multline}
We can now inter that this part is negligible due to the fact that $\varepsilon^{-3 / 2} e^{-a / \varepsilon} \rightarrow 0$ as $\varepsilon \rightarrow 0$.

It remains to consider $z \in R_{1}$. Here we use
$$
\P\left(z+S(m)=y, \tau_{z}>m\right)=
\P(z+S(m)=y)-\P\left(z+S(m)=y, \tau_{z} \leq m\right) \text {. }
$$
Applying Lemma~\ref{lem:4.4} once again, we have
$$
\left|\P\left(z+S(m)=y, \tau_{x}>m\right)-P(z+S(m)=y)\right| \le \frac{C}{\sqrt{n}} \varepsilon^{-3 / 2} e^{-a / \varepsilon},
$$
uniformly in $z \in R_{1}, y>\varepsilon \sqrt{n}$. Therefore,
\begin{align*}
& \Bigl| \sum_{z \in R_{1}} \P\left(x+S(n-m)=z, \tau_{x}>n-m\right) \P\left(z+S(m)=y, \tau_{z}>m\right) \\
&\hspace{1cm}-\sum_{z \in R_{1}} \P\left(x+S(n-m)=z, \tau_{x}>n-m\right)
\P(z+S(m)=y) \Bigr | \\
&\hspace{3cm} \le C \frac{V(x)}{n} \varepsilon^{-3 / 2} e^{-a / \varepsilon}.
\end{align*}
Combining this with \eqref{eq:4}, we obtain
\begin{align*}
&\bigg| \P\left(x+S(n) =y,  \tau_{x}>n\right)\\
&\hspace{1cm}-\sum_{z>0} \P\left(x+S(n-m)=z,\tau_{x}>n-m\right) \P(z+S(m)=y) \bigg| \\
 &\hspace{3cm}\le C \frac{V(x)}{n} \varepsilon^{-3/2} e^{-a / \varepsilon}.
\end{align*}
We know from \eqref{eq:lclt-1}, \eqref{eq:lclt-2} that
$$
\P(z+S(m)=y)=\frac{d}{\sqrt{2 \pi  m}} e^{-(y-z)^{2} / 2 m}+o\left(\frac{1}{\sqrt{m}}\right)
$$
uniformly in $z$ and $y$ such that $y-z\in D_m$
and
$$
\P(z+S(m)=y)=0\quad \text{if }y-z\notin D_m.
$$
Therefore,
$$
\begin{aligned}
& \sum_{z>0} \P\left(x+S(n-m)=z, \tau_{x}>n-m\right) \P(z+S(m)=y) \\
&=\sum_{z>0} \P\left(x+S(n-m)=z, \tau_{x}>n-m\right) \frac{d}{\sqrt{2 \pi n}} \varepsilon^{-3 / 2} e^{-(y-z)^{2} / 2 \varepsilon^{3} n}+o\left(\frac{1}{n}\right) \\
&=\frac{d}{\varepsilon^{3 / 2} \sqrt{2 \pi n}} \E\left[\exp \left\{-\frac{(y-x-S(n-m))^{2}}{2 \varepsilon^{3} n}\right\}; \tau_{x}>n-m\right]+o\left(\frac{1}{n}\right) \\
&= d\frac{\P\left(\tau_{x}>n-m\right)}{\varepsilon^{3 / 2} \sqrt{2 \pi n}} \E\left[\exp \left\{-\frac{(S(n-m)+x-y)^{2}}{\frac{2 \varepsilon^{3}}{1-\varepsilon^{3}}(n-m)}| \, \Big|\,  \tau_{x}>n-m\right]+o\left(\frac{1}{n}\right)\right. .
\end{aligned}
$$
It follows from \eqref{eq:1} that 
$$
\begin{aligned}
& E\left[\exp \left\{-\frac{(S(n-m)+x-y)^{2}}{2 \frac{\varepsilon^{3}}{1-\varepsilon^{3}}(n-m)}\, \Big|\,  \tau_{x}>n-m\right]\right. \\
& \quad=\int_{0}^{\infty} u e^{-u^{2} / 2} \exp \left\{-\frac{\left(1-\varepsilon^{3}\right)}{2 \varepsilon^{3}}\left(u-\frac{y}{\sqrt{n-w}}\right)^{2}\right\}+o(1).
\end{aligned}
$$
It remains to notice that, as $\varepsilon \rightarrow 0$,
the following convergence takes place 
$$
\frac{1}{\sqrt{2\pi}} \varepsilon^{-3 / 2} \int_{0}^{\infty} u e^{-u^{2} / 2} \exp \left\{-\frac{\left(1-\varepsilon^{3}\right)}{2 \varepsilon^{3}}(u-v)^{2}\right\} du \rightarrow v e^{-v^{2} / 2}.
$$
\end{proof}


\section{Markov chains}\label{sec:markov}

In this section we shall demonstrate that the universality method used in Sections~\ref{sec:universality} and \ref{sec:iid} works also in the case when the increments of the walk are no longer independent. We shall consider a time-homogeneous real-valued Markov chain $\{X(n)\}$ and study the first time when this chain becomes non-positive:
$$
\tau:=\inf\{n\ge1: X(n)\le 0\}.
$$
Let $\xi(x)$ denote the jump of the chain $\{X(n)\}$ from the state $x$, that is,
\[
\pr(\xi(x)\in B) = \pr(X_1-x
\in B\mid X_0=x). 
\]
We shall use the standard for Markov chains agreement and denote by $\P_x$ the distribution of the chain conditioned on $X_0=x$.

We shall assume that  
\begin{equation}\label{ass:m2}
\E[\xi(x)]=0
\text{ and }
\E[\xi^2(x)]=1 \text{ for all } x.
\end{equation}
Furthermore, we shall  assume that
there exists a positive 
random variable $Y$ with a 
finite second moment such that, for all $x$,
\begin{equation}\label{ass:m1}
 \pr(|\xi(x)|>y)
\le \pr(Y>y),\quad  y\ge 0.
\end{equation}
We shall also assume, without loss of generality, that the tail of $Y$ is regularly varying with index $-2$.

Assumptions \eqref{ass:m2} and \eqref{ass:m1} ensure that one can apply martingale versions of the central limit theorem to the chain $\{X(n)\}$. This implies that
the sequence
$$
x^{(n)}(t):=\frac{X(\lfloor nt\rfloor)
+(nt-\lfloor nt\rfloor)(X(\lfloor nt\rfloor+1)-X(\lfloor nt\rfloor))}{\sqrt{n}}
$$
converges weakly on $C[0,1]$ towards the standard Brownian motion $B(t)$. Thus, one may expect that the tail behaviour of $\tau$ should be similar to the tail behaviour of the corresponding exit time for $B(t)$.

\vspace{6pt}

The moment assumptions in \eqref{ass:m2} are imposed only to simplify the technical arguments. A rather similar universality idea has been used in the monograph \cite{Denisov_Korshunov_Wachtel_2025} to study stopping times for the chains with asymptotically zero drift. This class of chains is characterised by the following moment assumptions on the jumps:
$$
\E[\xi(x)]\sim\frac{\mu}{x}\quad\text{and}\quad
\E[\xi^2(x)]\sim b\in(0,\infty)\quad\text{as }x\to\infty.
$$
Under these assumptions one compares discrete time chains with Bessel processes. Thus fact underlines that the universality method is not restricted to the Brownian motion. 

Below we apply universality idea to the behaviour of the Green functions of killed discrete time Markov chain and killed Brownian motion. This type of the universality has been earlier used in \cite{denisov_wachtel_2024} to construct harmonic functions for multidimensional random walks killed at leaving a cone. The approach in \cite{Denisov_Korshunov_Wachtel_2025} is more sophisticated and does not use Green functions.

The approach described below is a simplified version of the work \cite{denisov_zhang_2024}, where multi-dimensional Markov chains in cones have been considered.

\subsection{Construction of the positive  harmonic function}
As we have seen in the analysis of random walks, the dependence of $\tau$ from the starting point $x$ is described by an appropriate positive harmonic function for a killed random walk.  So, we first construct such a function for $X(n)$ killed at leaving half-axis $(0,\infty)$. To this end we shall apply the universality idea to Green functions. 

Let us start with an example, which illustrates the role of Green functions. 
Assume for a moment that $\{X(n)\}$ is integer-valued. Set
\[
G(x,y) = \sum_{n=0} ^\infty  \P_x(X_n=y,\tau>n),\quad x,y>0
\]
and 
\[
a^*(x):= \E_x[X_1;X_1>0]-x,\quad x>0.
\]
Then the function
\begin{equation}\label{eq:U-star}
U^*(x):= x+\sum_{y=1}^{\infty} G(x,y)a^*(y),\quad x>0
\end{equation}
is harmonic for $\{X(n)\}$ killed at the stopping time $\tau$.
Indeed,
\begin{align*}
&\E_x[U^*(X_1);\tau>1]\\
&\hspace{5mm}=\sum_{z=1}^\infty \pr_x(X_1=z) U^*(z)\\
&\hspace{5mm}= \E_x[X_1;X_1>0]
+\sum_{z=1}^\infty 
\pr_x(X_1=z,\tau>1)
\sum_{n=0} ^\infty  \sum_{y=1}^\infty\pr_z(X_n=y,\tau>n)a^*(y)\\
&\hspace{5mm}=x+a^*(x)+\sum_{n=0}^\infty \sum_{y=1}^\infty \pr_x(X_{n+1}=y,\tau>n+1)a^*(y)\\
&\hspace{5mm}=x+\sum_{y=1}^\infty G(x,y)a^*(y).
\end{align*}
The first summand in the definition of $U^*$ is the harmonic function for $B(t)$ killed at leaving $(0,\infty)$; the second one is a correction. 
Since we have no information about the Green function $G(x,y)$ we cannot even guarantee that the  series $\sum_{y=1}^{\infty} G(x,y)a^*(y)$ converges. 
To avoid this problem we can 
use instead the Green function 
of the Brownian motion given by $2\min(x,y)$ 
which (in view of invariance principle) 
would give approximately the required harmonic function. 
This suggests to make use of the 
following function 
\[
U^{**}(x):= x+2 \sum_{y=1}^{\infty} \min(x,y) a^*(y).
\]
We will need then to correct this function a little bit more to get rid of errors of the diffusion approximation to obtain the final harmonic function. 

We first construct a superharmonic function $W(x)$ using the majorant of $a^*(x)$ given by 
\begin{align*}
a(x)&=-\E\left[x-Y; x-Y \leq 0\right]=\int_{x}^{\infty} \P(Y \geq y) d y. \\
\end{align*}
Instead of the series $\sum_{y=1}^{\infty} \min(x,y) a^*(y)$ we then shall use the correction
\begin{align*}
m(x)&=\int_{0}^{\infty} \min(x,y) a(y) d y.
\end{align*}
Putting  
\[
b(x)=\int_{x}^{\infty} a(y) d y
\]
wee can rewrite the correction term 
applying the integration by parts 
as follows 
\begin{align*}
m(x)= \int_0^x ya(y)dy +x b(x) 
= - \int_0^x yb'(y)dy+x b(x)
=\int_0^x b(y) dy. 
\end{align*}

\begin{lemma}\label{lem:3.15.mc.new}
There exist positive constants $A$ and $R$ 
such that function 
$$
W(x)=x+R+Am(x+R)
$$
is superharmonic for $X$ killed at $\tau$. In other words
\begin{equation}\label{eq:start.mc.new}
\E_x\left[W(X(1)) ; \tau>1\right] \leq W(x) \text { for all } x \ge 0
\end{equation}
or, equivalently, the sequence $W(X(n))\ind\{\tau>n\}$, $n\ge0$ is a supermartingale.
\end{lemma}
\begin{proof}
We want to show that
$$
\Delta(x):=\E_x\left[W(X(1)) ; \tau>1\right]-W(x) \le 0 \text { for all } x \ge 0.
$$
for a suitable choice of $A$ and $R$. 

Recalling that $\xi(x)$ denotes the jump from the state $x$, we have
\begin{align*}
\Delta(x)
&=\E\left[W(x+\xi(x))-W(x)\right]\\
&\le \E [W(x+\xi(x))-W(x); |\xi(x)|< x+R]\\
&\hspace{1cm}+\E [W(x+\xi(x))-W(x); \xi(x)\ge  x+R]\\
&=:\Delta_1(x)+\Delta_2(x).
\end{align*}
By the mean value theorem, 
\begin{align*}
\Delta_1(x)
&=W'(x) \E[\xi(x);|\xi(x)|<x+R]\\
&\hspace{2cm}
+\frac{1}{2}\E[W''(x+\theta \xi(x))\xi^2(x);|\xi(x)|< x+R].
\end{align*}
It is immediate from the definition of $W(y)$ that 
$$
W'(y)=1+Ab(y+R)\quad\text{and}\quad W''(y)=-Aa(y+R).
$$
Using these equalities and noting that $a(y+R)$ decreases, we obtain the estimate 
\begin{align*}
\Delta_1(x)&\le \left(1+Ab(x+R)\right)\E[\xi(x);|\xi(x)|<x+R]\\
&\hspace{2cm}-\frac{1}{2}a(x+2R)\E[\xi^2(x);|\xi(x)|<x+R].
\end{align*}
It follows from the assumptions \eqref{ass:m2} and \eqref{ass:m1} that
\begin{align*}
\E[\xi(x);|\xi(x)|<x+R]    
&=-\E[\xi(x);|\xi(x)\ge x+R]\\
&\le \E[Y;Y\ge x+R]=a(x+R)
\end{align*}
and
\begin{align*}
\E[\xi^2(x);|\xi(x)|<x+R]
&=1-\E[\xi^2(x);|\xi(x)|\ge x+R]\\
&\ge 1-\E[Y^2;Y\ge R].
\end{align*}
Applying these bounds and noting that the function $b(y)$ decreases, we conclude that
$$
\Delta_1(x)\le (1+Ab(R))a(x+R)-\frac{A}{2}\left(1-\E[Y^2;Y\ge R]\right)a(x+2R).
$$
The assumption that the tail of $Y$ is regularly varying with index $-2$ implies that $a(y)$ is regularly varying with index $-1$. Thus, we can choose $R$ so large that $a(x+2R)\ge \frac{1}{4}a(x+R)$ for all $x\ge0$. Furthermore,
$\E[Y^2;Y\ge R]\le\frac{1}{2}$ for all sufficiently large $R$. Therefore, there exists $R_0$ such that 
$$
\Delta_1(x)\le \left(1+Ab(R)-\frac{A}{16}\right)a(x+R).
$$
Furthermore, recalling that $m'(y)=b(y)$ decreases, we obtain
\begin{align*}
\Delta_2(x)
&=\E[\xi(x);\xi(x)\ge x+R]+A\E[m(x+R+\xi(x))-m(x+R);\xi(x)\ge x+R] \\
&\le \left(1+Ab(R)\right) a(x+R).
\end{align*}
Combining this with the bound for $\Delta_1(x)$, we conclude that
$$
\Delta(x)\le \left(2+2Ab(R)-\frac{A}{16}\right)a(x+R),\quad x\ge0
$$
for all $R\ge R_0$.
Since $\lim_{y\to\infty}b(y)=0$, there exists $R_1\ge R_0$ such that
$b(y)<\frac{1}{64}$ for all $R\ge R_1$. Then, for all $R\ge R_1$ and all $A\ge 32$
we have
$$
\Delta(x)\le 0,\quad x\ge0.
$$
Thus, the proof is completed.
\end{proof}
We can now make use of this superharmonic function 
to find a positive harmonic function for the killed Markov chain. 
\begin{lemma}
    Assume that \eqref{ass:m2} and \eqref{ass:m2} hold. Then the function 
    \[
    V(x) = x- \E_x X(\tau) 
    \]
    is positive harmonic for $\{X(n)\}$ killed at leaving $(0,\infty)$.
    Moreover, 
    \begin{equation}
    \label{eq:V_less_W}
    V(x)\le W(x)\quad\text{for all }x>0
    \end{equation}
    and
    \begin{equation}\label{eq.renewal.V}
    \lim_{x\to \infty} \frac{V(x)}{x} =1.  
    \end{equation}
\end{lemma}
\begin{proof}
    The supermartingale property of the sequence $W(X(n))\ind\{\tau>n\}$ implies that
    \begin{equation}\label{eq:new1}
    \E_x [X(n);\tau>n] \le \E_x[W(X(n));\tau>n]\le W(x).
    \end{equation}
    Applying the optional stopping theorem to the martingale $X_n$, we obtain
    \begin{align}
    \label{eq:ost}
    \nonumber
    \E_x [X(n);\tau>n] &= \E_x [X(\tau\wedge n);\tau>n]\\
    \nonumber
    &=\E_x [X(\tau\wedge n)]-\E_x [X(\tau\wedge n);\tau\le n]\\
    &= x -\E_x [X(\tau);\tau\le n],\quad x>0.
    \end{align}
    Therefore, 
    \[
    \E_x [-X(\tau);\tau\le n] \le W(x)-x=R+Am(x+R),\quad x>0.
    \]
    Since $-X(\tau)\ge 0$, by the monotone convergence theorem we obtain that 
    \[
    0\le \E_x [-X_{\tau}]\le R+Am(x+R). 
    \]
    Thus function $V(x)$ is well-defined. 
    Moreover, since $W(x)-x=o(x)$ as $x\to\infty$, we also 
    obtain~\eqref{eq.renewal.V}. 
    
    The harmonicity follows from the following application of the Markov property and the assumption $\E_x[X_1-x]=0$,
    \begin{align*}
    \E_x[V(X_1);\tau>1]
    &= \int_0^\infty \pr_x(X_1\in dy)(y-\E_yX_\tau)\\
    &=\E_x[X_1;\tau>1]-\E_x[X_\tau;\tau>1]\\
    &=x-\E_x[X_1;\tau=1]-\E_x[X_\tau;\tau>1]
    =V(x). 
    \end{align*}
    Due to \eqref{eq:ost},
    $$
    V(x)=\lim_{n\to\infty}\E_x[X(n);\tau>n].
    $$
    Applying now \eqref{eq:new1}, we obtain \eqref{eq:V_less_W}. Thus, the proof ic finished.
\end{proof}

\subsection{Coupling of Markov chain and a simple random walk}
To find  asymptotics and  upper bounds      for 
$\pr_x(\tau>n)$ 
we make use of the following 
coupling that follows from Sakhanenko~\cite[Corollary 3]{sakhanenko1988}.
\begin{proposition}\label{prop.invariance}
There exist sequences $\gamma_n =o(\sqrt n)\to 0$ and 
$\pi_n \to 0$ such that 
for each $n$ and $x$ 
one can construct Markov Chain $\{X(k)\}_{k=0}^n$ starting from $x$ and 
a symmetric simple random walk 
$\{S(k)\}_{k=0}^n$
on the same probability space in such a way that 
\begin{equation}\label{eq:coupling.bm.mc}
   \pr_x\left(\sup_{k\le n}|X(k)-x-S(k)|> \gamma_n\right) \le \pi_n.
\end{equation}
\label{lemma:convergence.bm}
\end{proposition}
\begin{remark}
Here $\pi_n = D(\gamma_n)$ and 
\[
D(\gamma_n)\le n
\left(
\E\left[
h_3\left(\frac{2Y}{\gamma_n}\right)
\right]
+\E\left[
h_3\left(\frac{2S(1)}{\gamma_n}\right)
\right]
\right)
\]
with $h_3(x)=\min(|x|^2,|x|^3)$. 
A similar statement can be made for 
linearly interpolated Markov chain $x^{(n)}(t)$ and 
a  Brownian motion 
$(B(t))_{t\ge 0}$ 
instead of a simple random walk. 
\end{remark}

\subsection{Upper bounds for $\pr_x(\tau>n)$}
We first prove a crude estimate for $\P_x(\tau>n)$, which will later be used to obtain a rather sharp upper bound for this probability.
\begin{lemma}\label{lem:near.boundary}
There exists a constant $C$ 
and a sequence $\delta_n\to 0$ such that 
\[
\pr_x(\tau>n) \le C\varepsilon+\delta_n, \text{ for any } x\in [0,\varepsilon \sqrt n]
.  
\]
\end{lemma}
\begin{proof}
To prove we make use of Proposition~\ref{prop.invariance} and 
construct $X(k)$ and $x+S(k)$ on the same probability space. 
Let 
\[
\widetilde \tau_x :=\inf\{k\ge 1\colon x+S(k)\le 0\}
\]
Then, for sequences $\gamma_n=o(\sqrt n) \to 0$  and $\pi_n\to 0$, 
\begin{align*}
\pr_x(\tau>n) &\le 
\pr(x+\min_{1\le k\le n}(X(k)-x)>0, \max_{k\le n}|X(k)-x-S(k))|\le \gamma_n)
+\pi_n\\
&\le \pr(x+\gamma_n+\min_{1\le k\le n} S(k)>0)
+\pi_n
\\
&\le \pr(\widetilde \tau_{x+\gamma_n}>n)
+\pi_n. 
\end{align*}
Applying~\eqref{eq:tau_x.upper} we hence  obtain 
\[
\pr_x(\tau>n)
\le C\frac{x+1+\gamma_n}{\sqrt n}
+\pi_n
\le C\varepsilon +
C\frac{1+\gamma_n}{\sqrt n}
+\pi_n.
\]
Thus, the desired inequality holds with
$\delta_n=\frac{C(1+\gamma_n)}{\sqrt{n}}$.
\end{proof}
\begin{lemma}\label{lem:tau.upper}
    There exists a constant $C$ such that 
    \begin{equation}\label{eq:tau.upper}
    \pr_x(\tau>n) \le C\frac{x+1}{\sqrt n}
    \end{equation}
    for any $x\ge 0$ and $n\ge 1$. 
\end{lemma}
\begin{proof}
    Since $\pr_x(\tau>n)$ is decreasing in $n$,
    it is sufficient to prove the claim for $n=2^m, m\ge 1$. 
    For every fixed $\varepsilon>0$ we have
    \[\pr_x(\tau>n) \le 
    \pr_x(\tau>n, X(n/2)>\varepsilon \sqrt n)\\
    +\pr_x(\tau>n, X(n/2)\le \varepsilon \sqrt n). 
    \]
    Applying the Markov inequality to the first summand and recalling that
    the superharmonic function $W(x)$ constructed in Lemma~\ref{lem:3.15.mc.new}
    satisfies $x\le W(x)$ for all $x\ge0$, we obtain 
    \begin{align*}
    \pr_x(\tau>n, X(n/2)\ge \varepsilon \sqrt n)
    &\le \frac{\E_x [X(n/2);\tau>n/2]}{\varepsilon \sqrt {n/2}}\\
    &\le \frac{\E_x [W(X(n/2));\tau>n/2]}{\varepsilon \sqrt {n/2}}.
    \end{align*}
    Applying Lemma~\ref{lem:3.15.mc.new}, we obtain the bound
    $$
    \pr_x(\tau>n, X(n/2)\ge \varepsilon \sqrt n)\le
    \frac{W(x)}{\varepsilon\sqrt{n/2}}.
    $$
    It  follows from the definition of $W$ that there exists a constant $C$ such that $W(x)\le C(x+1)$ for all $x\ge 0$. 
    Therefore, 
    \begin{equation}\label{eq:upper.1}
    \pr_x(\tau>n, X(n/2)\ge 
    \varepsilon \sqrt n)\le \sqrt 2 C\frac{x+1}{\varepsilon \sqrt n}. 
    \end{equation}
    Next, using the Markov property and applying Lemma~\ref{lem:near.boundary},
    we obtain
    \begin{align*}
    \pr_x(\tau>n, X(n/2)\le \varepsilon \sqrt n) 
    &=\int_0^{\varepsilon \sqrt n}
    \pr_x(\tau>n/2, X(n/2)\in dy)
    \pr_y(\tau>n/2)\\ 
    &\le C \pr_x(\tau>n/2)(\varepsilon+\delta_n).
    \end{align*}
    
    Combination of the latter and former bounds gives 
    \[
    \pr_x(\tau>n)\le 
     \sqrt 2 C\frac{x+1}{\varepsilon \sqrt n}
     +C \pr_x(\tau>n/2)(\varepsilon+\delta_n).
    \]
    Now pick $\varepsilon$ such that $\widetilde \varepsilon:= 2C\varepsilon <1/2$. Let $n_0$ be such that for $n\ge n_0$ sequence $\delta_n\le \varepsilon$. We then obtain 
    \[
    \pr_x(\tau>n)\le 
    \frac{x+1}{\varepsilon^2 \sqrt n}
    +\widetilde \varepsilon \pr_x(\tau>n/2).
    \]
    Now iterate this inequality $N$ times to obtain 
    \[
    \pr_x(\tau>n)\le
    \frac{x+1}{\varepsilon^2}
    \sum_{i=0}^{N-1}
    \widetilde \varepsilon^i \frac{1}{\sqrt{n/2^i}}
    +\widetilde \varepsilon^N \pr_x(\tau>n/2^N).
    \]
    Take $N=\frac{\log n}{-2\log \widetilde \varepsilon}$ 
    to obtain for $n/2^N>n_0$ that 
    \[
    \pr_x(\tau>n)\le 
    \frac{x+1}{\varepsilon^2 \sqrt n} 
    \frac{1}{1-2\widetilde \varepsilon}+\frac{1}{\sqrt n}.
    \]
    The claim then follows. 
\end{proof}

\subsection{Asymptotics for $\pr_x(\tau>n)$ and conditional limit theorem}
We can now extend Corollary~\ref{cor:tau_x-asymp} 
and Corollary~\ref{cor:limit.th-simple} from symmetric simple random walk to Markov chain 
satisfying our assumptions. 
\begin{theorem}\label{thm.mc.conditioned}
Assume that \eqref{ass:m2}, and \eqref{ass:m1} are valid.
Fix any sequence $\delta_n\downarrow 0$ and $x_0>0$. The one has the following statements.
\begin{enumerate}[(i)]
\item Uniformly in $x\in[x_0, \delta_n \sqrt n]$, 
\begin{equation*}
    \mathbf{P}_x(\tau >n) \sim \sqrt{\frac{2}{\pi}} \frac{V(x)}{\sqrt n}, 
    \quad n\to \infty.
\end{equation*}
\item 
For any $v\ge 0$, uniformly in $x\in[x_0, \delta_n \sqrt n]$, 
\begin{equation*}
    \mathbf{P}_x \left( \frac{X(n)}{\sqrt{n}} >v\Big| \tau >n \right) 
    \to  e^{-v^2/2} \quad\text{as } n\to \infty. 
\end{equation*}
\end{enumerate}
\label{thm:p.tau}
\end{theorem}
\begin{proof}
    Take $m=m(n)$ such that $\frac{m(n)}{n}\to0$ sufficiently slow. 
    In particular, we shall assume, without loss of generality, 
    that
    $$
    x\le \sqrt{\delta_n m}.
    $$ 
    Fix  also $\varepsilon\in(0,1)$ and  $A>1$. Using the  Markov property at time $m$, we have
    \begin{align*}
        &\mathbf{P}_x \left(\frac{X(n)}{\sqrt{n}} >v, \tau>n\right) \\
        &\hspace{5mm}=P_1+P_2+P_3\\
        &\hspace{5mm}:= \int_{0}^{\varepsilon \sqrt m} \pr_x(X(m) \in d y, \tau>m) \mathbf{P}_y \left(\frac{X(n-m)}{\sqrt{n}} >v, \tau > n-m\right) \\
        &\hspace{1cm}+\int_{\varepsilon \sqrt m}^{A\sqrt m} \mathbf{P}_x\left(X(m) \in d y, \tau>m\right) \mathbf{P}_y \left(\frac{X(n-m)}{\sqrt{n}} >v, \tau > n-m\right)\\
        &\hspace{1cm}+\int_{A\sqrt m}^\infty  \mathbf{P}_x(X(m) \in d y, \tau>m) \mathbf{P}_y \left(\frac{X(n-m)}{\sqrt{n}} >v, \tau > n-m\right).
    \end{align*}
Applying  Lemma~\ref{lem:tau.upper} we obtain
    \begin{align*}
   P_1 & \le \int_{0}^{\varepsilon \sqrt m}  \mathbf{P}_x(X(m) \in d y, \tau>m) \mathbf{P}_y (\tau>n-m) \\& \le C \int_{0}^{\varepsilon \sqrt m}   \mathbf{P}_x(X(m) \in d y, \tau>m) \frac{1+y}{n^{\frac{1}{2}}}\\
   &\le 2C\varepsilon  \sqrt{\frac{m}{n}}\P_x(\tau>m),
    \end{align*}
    for all $m\ge 1/\varepsilon^2$. 
    Using Lemma~\ref{lem:tau.upper} one more time and noting that $V(x)\ge x\ge x_0$ for all $x\ge x_0$, we conclude that
    \begin{equation}\label{eq.p1}
    P_1\le \varepsilon C_1 
    \sqrt{\frac{m}{n}} \frac{1+x}{\sqrt m}
    \le \varepsilon
    C_2 \frac{V(x)}{\sqrt n}.  
    \end{equation}

To bound $P_3$ we apply again Lemma~\ref{lem:tau.upper}. This gives
\begin{align*}
    P_3
    &\le \frac{C}{n^{\frac{1}{2}}} 
    \int_{A\sqrt m}^\infty\mathbf{P}_x(X(m) \in dy,\tau>m)(1+y)\\
    & \le \frac{C}{n^{\frac{1}{2}}} \e_x[1+X(m), \tau>m, X(m)> A \sqrt{m}]\\
     & \le \frac{2C}{n^{\frac{1}{2}}} \e_x[X(m), \tau>m, X(m)> A \sqrt{m}].
\end{align*}
Next, for all $x\le\sqrt{\delta_n m}$,
\begin{align*} 
&\E_x[X(m);\tau>m, X(m)>A\sqrt m] \\
&\hspace{1cm}\le x \pr_x(\tau>m) 
+\E_x[X(m)-x;\tau>m, X(m)>A\sqrt m] \\
&\hspace{1cm}\le C\frac{(1+x)^2}{\sqrt m}
+\frac{\E_x[(X(m)-x)^2;
\tau>m, X(m)-x>(A-\delta_n^{1/2})\sqrt m]}{(A-\delta_n^{1/2})\sqrt m}.
\end{align*}
The assumption $x\le \sqrt{\delta_n m}$ also implies that 
\[
\frac{1+x}{\sqrt m} \le 2\delta_n^{1/2}. 
\]  
Using the fact that $(X(n)-x)^2-n$ is a martingale, we have 
\begin{align*} 
0&= \E_x[X(\tau\wedge m)-x)^2-\tau\wedge m]\\
 &= \E_x[X(m)-x)^2 -m;\tau>m]+ \E_x[X(\tau)-x)^2 -\tau;\tau\le m]\\
 &\ge \E_x[X(m)-x)^2 -m;\tau>m]
     -\E_x[\tau;\tau\le m].
\end{align*}
Therefore, 
\begin{multline*}
\E_x[X(\tau)-x)^2;\tau>m]
\le m\pr_x(\tau>m) +\E_x[\tau;\tau\le m] 
\le C V(x)\sqrt m. 
\end{multline*}
Then,
\begin{equation}
    P_3 \le C\frac{V(x)}{n^{\frac{1}{2}}} 
    \left(\frac{1}{A}+\delta_n\right).
\label{eq.p3}
\end{equation}
By taking $m(n)$ sufficiently large (but still $o(n)$), 
we can ensure using 
Proposition~\ref{lemma:convergence.bm}
that,  
\begin{equation}
    \mathbf{P}_y\left(\frac{X(n-m)}{\sqrt{n}} >v, \tau>n-m\right) 
    \sim \mathbf{P}\left(\frac{y+S\left(n-m\right)}{\sqrt n}>v, \widetilde \tau_y >n-m\right),
\end{equation}
 uniformly in $y\in [\varepsilon, A]\sqrt m$.
Corollary~\ref{cor:limit.th-simple} implies then,
\[
\mathbf{P}\left(\frac{y+S\left(n-m\right)}{\sqrt n}>v, \widetilde \tau_y >n-m\right)  \sim \sqrt {\frac{2}{\pi}}\frac{y}{n^{\frac{1}{2}}} e^{-\frac{v^2}{2}},
\]
uniformly in $y\in [\varepsilon, A]\sqrt m$.
Then,  
\begin{align*}
   P_2& \sim \sqrt {\frac{2}{\pi n}} 
   e^{-v^2/2}
   \int_{\varepsilon \sqrt m}^{A\sqrt m} \mathbf{P}_x(X(m) \in d y, \tau>m) y.  
\end{align*}
Using the bounds  
for $P_1$ and $P_3$
\begin{equation}
\label{eqn:k2}
    \begin{split}
    &\Biggl|
    P_2 - \sqrt {\frac{2}{\pi n}} 
   e^{-v^2/2}
   \int_{0}^{\infty} \mathbf{P}_x(X(m) \in d y, \tau>m) y
    \Biggr| \\& =\int_{
    y\notin [\varepsilon, A]\sqrt m}\mathbf{P}_x(X(m) \in d y, \tau>m) y e^{-v^2/2}  
    \\& \le \frac{CV(x)}{n^{\frac{1}{2}}} \left(\varepsilon+ \frac{1}{A}+\delta_n\right). 
    \end{split}
\end{equation}
Combining equations \eqref{eq.p1}, 
\eqref{eq.p3} and \eqref{eqn:k2}, the statement of part (ii) follows,
\begin{align*}
&\Bigg|\pr_x\left(\frac{X(n)}{\sqrt{n}}>v,\tau_x>n\right)
-\sqrt{\frac{2}{\pi n^{1/2}}} e^{-v^2/2}
\e_x[X(m);\tau>m]\Bigg|\\
&\hspace{1cm} \le 
CV(x)\frac{\varepsilon+ \frac{1}{A}+\delta_n}{n^{1/2}}.
\end{align*}
uniformly in $x\le \delta_n\sqrt{m}$.
Letting  $\varepsilon \to 0$ and $A \to \infty$ 
we obtain,
\begin{equation}
    \mathbf{P}_x \left( \frac{X(n)}{\sqrt{n}} >v ;\tau >n \right) \sim 
    \sqrt{\frac{2}{\pi n^{1/2}}} V(x) e^{-v^2/2}
\end{equation}
and taking $v=0$
\begin{equation*}
    \pr_x(\tau>n) \sim \sqrt{\frac{2}{\pi }} \frac{V(x)}{\sqrt n}.
\end{equation*}
\end{proof}

\subsection{Markov chain conditioned to stay positive} 
Similar to the case of random walks, harmonic function $V(x)$ can be used to perform the Doob $h$-transform. 
Let $\widehat{\P}$ denote the corresponding measure:
\[
\widehat{\pr}_x (\widehat X(n)\in B) 
=\int_B
\frac{V(y)}{V(x)}\pr_x (X(n)\in y,\tau>n)  
\]
for all $x,n>0$ all all Borel subsets $B$ of $(0,\infty)$.

As in the case of walks with i.i.d. increments, the measure $\widehat{\P}$ can be obtained by conditioning of the the original chain.
\begin{lemma}
\label{lem:hatP-P}
For every $x>0$ and for every $A\in\sigma(X(1),X(2),\ldots)$ with some fixed $k\ge1$,
$$
\widehat{\P_x}(A)=\lim_{n\to\infty}\P_x(A|\tau>n).
$$
\end{lemma}
The proof of this lemma is the verbatim of the proof of Lemma~\ref{lem:doob-transform} and we omit it.

Using Theorem~\ref{thm.mc.conditioned} one can obtain a limit theorem
for the chain $\{X(n)\}$ under the new measure $\widehat{\P}$. The following result is thus an extension of Proposition~\ref{prop:Doob-limit}, where simple random walks have been considered.
\begin{theorem}
\label{thm:Doob-limit}
Assume that the assumptions \eqref{ass:m2} and \eqref{ass:m1} are valid.
Then, for every fixed $x>0$,
$$
\lim_{n\to\infty}\widehat{\P}_x\left(\frac{X(n)}{\sqrt{n}}\ge v\right)
=\sqrt{\frac{2}{\pi}}\int_v^\infty u^2e^{-u^2/2}du
\quad \text{for every }v>0.
$$
\end{theorem}
\begin{proof}
The claim is equivalent to
\begin{align}\label{eq:v1-v2}
\lim_{n\to\infty}\widehat{\P}_x\left(\frac{X(n)}{\sqrt{n}}\in(v_1,v_2]\right)
=\sqrt{\frac{2}{\pi}}\int_v^\infty u^2e^{-u^2/2}du
\quad \text{for all }v_2>v_1>0.
\end{align}
By the definition of the measure $\widehat{\P}$ and by~\eqref{eq.renewal.V}, 
\begin{align*}
&\widehat{\P}_x(X(n)\in (v_1,v_2])\\
&\hspace{1cm}=\frac{1}{V(x)}\int_{v_1\sqrt{n}}^{v_2\sqrt{n}}
V(y)\P_x(X(n)\in dy,\tau>n)\\
&\hspace{1cm}=\frac{1+o(1)}{V(x)}\int_{v_1\sqrt{n}}^{v_2\sqrt{n}}
y\P_x(X(n)\in dy,\tau>n)\\
&\hspace{1cm}=\frac{(1+o(1))\P_x(\tau>n)}{V(x)}\int_{v_1\sqrt{n}}^{v_2\sqrt{n}}
y\P_x(X(n)\in dy|\tau>n).
\end{align*}
Applying now Theorem~\ref{thm.mc.conditioned}(i), we conclude that 
\begin{align*}
\widehat{\P}_x(X(n)\in (v_1,v_2])
&=\sqrt{\frac{2}{\pi}}(1+o(1))
\int_{v_1\sqrt{n}}^{v_2\sqrt{n}}\frac{y}{\sqrt{n}}\P_x(X(n)\in dy|\tau>n)\\
&=\sqrt{\frac{2}{\pi}}(1+o(1))
\int_{v_1}^{v_2}u\P_x\left(\frac{X(n)}{\sqrt{n}}\in du\Big|\tau>n\right).
\end{align*}
Since the limiting law in Theorem~\ref{thm.mc.conditioned}(ii) is absolutely continuous,
$$
\lim_{n\to\infty}
\int_{v_1}^{v_2}u\P_x\left(\frac{X(n)}{\sqrt{n}}\in du\Big|\tau>n\right)
=\int_{v_1}^{v_2}u\cdot ue^{-u^2/2}du.
$$
This completes the proof of \eqref{eq:v1-v2} and, consequently, the proof of the theorem.
\end{proof}

\bibliographystyle{abbrv}
\bibliography{cones}

\end{document}